\documentclass[10pt]{article}

\usepackage{amsfonts,amssymb,amsmath,amsthm,indentfirst}
\usepackage{fontenc,enumerate}
\usepackage{graphics}
\usepackage{color,graphicx}

\usepackage{latexsym,mathrsfs,enumerate,verbatim,bm}
\usepackage[colorlinks, linkcolor=blue, citecolor=blue, urlcolor=blue, pagebackref, hypertexnames=false]{hyperref}

\usepackage[light]{antpolt}    \usepackage[T1]{fontenc}

\usepackage{wasysym}
\usepackage{graphics}

\usepackage{color,graphicx}

\DeclareMathAlphabet{\mathpzc}{OT1}{pzc}{m}{it}

\newcommand{\sett}[1]{\left\{   #1   \right\}}

\newcommand{\paar}[1]{\left(   #1   \right)}

\usepackage[colorinlistoftodos]{todonotes}
\makeatletter
\providecommand\@dotsep{5}
\def\listtodoname{List of Todos}
\def\listoftodos{\@starttoc{tdo}\listtodoname}
\makeatother
\allowdisplaybreaks
\usepackage{xcolor}
\usepackage[pagewise,mathlines]{lineno}

\newcommand{\vertiii}[1]{{\left\vert\kern-0.25ex\left\vert\kern-0.25ex\left\vert #1 
		\right\vert\kern-0.25ex\right\vert\kern-0.25ex\right\vert}}
\numberwithin{equation}{section}
\newtheorem{theorem}{\quad Theorem}[section]
\newtheorem{lemma}[theorem]{\quad Lemma}

\newtheorem{remark}[theorem]{\quad Remark}

\newtheorem{proposition}[theorem]{\quad Proposition}

\newtheorem{definition}{Definition}[section]

\newcommand{\R}{\mathbb{R}}

\newcommand{\dd}{{\rm d}}

\newcommand{\ii}{{\rm i}}

\newcommand{\N}{\mathbb{N}}

\newcommand{\al}{\alpha}
\newcommand{\rr}{\mathbb{R}}

\allowdisplaybreaks
  \author{Vicente Alvarez and  Amin Esfahani\footnote{Department of Mathematics, Nazarbayev University, Astana 010000, Kazakhstan\newline   E-mail: saesfahani@gmail.com, amin.esfahani@nu.edu.kz} }

\title{Multi-soliton solutions of  Klein–Gordon–Zakharov  system    
	\footnotetext{2020 Mathematical subject classification: 35L51, 35C07, 35B40}
	\footnotetext{Keywords:  Klein-Gordon-Zakharov system, Multi-soliton, Asymptotic behavior}}

\date{}

\begin{document}
	
	% \thanks{$^\ast$Corresponding author. \\Department of Mathematics, Nazarbayev University, Astana 010000, Kazakhstan\\
	% 	E-mail: {\tt vicente.alvarez@nu.edu.kz, amin.esfahani@nu.edu.kz.}}
	
	% \title[ ]{ 
	% 	\footnotetext{2020 Mathematical subject classification: 35Q35, 35C08, 35Q51, 35B40, 37K40}
	% 	\footnotetext{Keywords:   , Multi-soliton, Asymptotic behavior.
	% 	}
	% }   
	
	  \maketitle
	
	% \begin{center}
		
	% 	{\bf Vicente  Alvarez and  Amin Esfahani$^{\ast}$   
	% 		\vskip.2cm}
	% \end{center}
	
	% \vskip.3cm
	
	% %	\noindent{{\bm Abstract.} In this work,
	% 	%	}
	
	\begin{abstract}
In this study, we investigate the Klein-Gordon-Zakharov system with a focus on identifying multi-soliton solutions. Specifically, for a given number \( N \) of solitons, we demonstrate the existence of a multi-soliton solution that asymptotically converges, in the energy space, to the sum of these solitons. Our proof extends and builds upon the previous results in \cite{cote, cotem, IA} concerning the nonlinear Schrödinger equation and the generalized Klein-Gordon equation.  
In contrast to the method used in \cite{cotem} to establish the existence of multi-solitons for the Klein–Gordon equation, where the difficulty arises from the directions imposed by the coercivity property, requiring the identification of eigenfunctions of the coercivity operator to derive new control estimates—the structure of the present system allows for a more refined result. Specifically, the directional constraints can be eliminated by employing orthogonality arguments derived from localization and modulation techniques.

	\end{abstract}
 
	\section{Introduction}

In this paper, we consider the following system of  Klein–Gordon–Zakharov (KGZ) equations	
 	\begin{equation}\label{system1}
 	    \begin{cases}
      u_{tt}-u_{xx}+u+\al uv+\beta|u|^2u=0,\\
      v_{tt}-v_{xx}=(|u|^2)_{xx},\qquad x,t\in\rr.
  \end{cases}
 	\end{equation}
Here $\alpha\in\rr$ and $\beta\neq0$. In theoretical investigations of the dynamics of strong Langmuir turbulence in plasma physics, various forms of Zakharov equations play a crucial role (see Guo \cite{guo}, Ozawa, Tsutaya, and Tsutsumi \cite{ott, ott-2}, Thornhill and ter Haar \cite{Tt}, and Tsutaya \cite{Tsu}). 
Several studies have focused on the Klein–Gordon–Zakharov system \eqref{system1} with 
$\beta=0$, which describes the interaction between a Langmuir wave and an ion sound wave in plasma. In our notation, 
$u$ represents the fast-scale component of the electric field, while 
$v$ denotes the deviation of ion density \cite{zakh}. The first local well-posedness result for this system ($\beta=0$) appears to trace back to \cite{gv,ott}. Since the solutions in \cite{gv} were obtained using a fixed-point method in a Strichartz-type space, they were only conditionally unique. In an interesting development, Masmoudi and Nakanishi \cite{mana} demonstrated that, under additional smoothness assumptions, the solutions are also unconditionally unique (i.e., unique in a large  energy space). Significant progress was made in the late 1990s concerning the global existence of solutions to \eqref{system1}. It turns out that the conservation laws associated with \eqref{system1} (and its higher-dimensional analogs) are sufficient to control only small solutions, which were shown to exist globally \cite{ott,ott-2}. Notably, different propagation speeds and the (higher) dimension played a crucial role in the success of these approaches.
In \cite{hss2013}, it was recently proved in the case $\beta=0$ that \eqref{system1} is locally well-posed in $H^s\times H^{s-1}\times H^{s-1}\times  H^{s-2}$ with $s>1/2$, and the flow-map is Lipschitz.

Regarding the dynamics of standing waves, when \(\beta=0\), it was proved in \cite{hss2013} that \eqref{system1} possesses a solitary (standing) wave of the form \((u(x,t),v(x,t))=(\exp(\ii \omega t)\phi_\omega(x),\psi(x))\), where \(\phi\) and \(\psi\) are either real-valued periodic functions with a fixed fundamental period \(L\) or functions that vanish at infinity (in the whole line context). In the latter case, it is shown that \(\psi=-\phi^2\) and
\[
\phi(x)=\sqrt{2(1-\omega^2)}\;\text{sech}\left(\sqrt{1-\omega^2}\;x\right).
\]
Gan \cite{gan}, Gan and Zhang \cite{gabzhang} (see also \cite{ggz}), and subsequently Ohta and Odorova \cite{ohtaT}, proved the existence of a standing wave with a ground state of \eqref{system1} for \(\beta=0\), by applying an intricate variational argument. They obtained the instability of the standing wave by using Payne and Sattinger's potential well argument and Levine's concavity method in high-dimensional cases. In dimension one, Lin \cite{lin} investigated the orbital stability. More precisely, he showed that the system is orbitally unstable if \(\omega<1/\sqrt{2}\) and stable if \(1>\omega>1/\sqrt{2}\). In the degenerate case \(\omega=1/\sqrt{2}\), it was shown in \cite{yin} using a modulation argument combined with the virial identity that the standing wave is unstable.

  We focus on constructing some special solutions for \eqref{system1}. More precisely, let $R_j$, $j=1,\ldots, N$,   be N solitary waves (solitons) of the KGZ system; our goal is to construct a solution of the system that behaves at infinity like the sum of these solitons. See Definition \ref{def-ms-1} below.
 Multi-solitons were first constructed for the Korteweg–de Vries (KdV) equation using the inverse scattering method, as detailed in \cite{miura}. In the realm of non-integrable equations, the development of multi-solitons began with the nonlinear Schrödinger and generalized KdV equations, where Merle \cite{merle-1}, Martel \cite{martel-1}, and Martel–Merle \cite{marmer-1} made significant contributions in the \( L^2 \)-critical and subcritical cases. Progress on the supercritical regime was achieved by Côté and Martel \cite{CoteMartel}, while Combet \cite{combet, comb} classified multi-solitons for the \( L^2 \)-supercritical generalized KdV equation. The concept of multi-solitons has since been extended to other dispersive systems, including the nonlinear Klein–Gordon equation \cite{cotem} and the water wave system \cite{mingr}. More recently, Van Tin \cite{vantian} established the existence of multi-solitons and multi kink-solitons for derivative nonlinear Schrödinger equations using fixed point methods and Strichartz estimates. The dynamics and stability of multi-solitons, particularly in KdV-type equations, have also been thoroughly explored, with stability results presented in \cite{mmt}.

 To express  \eqref{system1} in a {\it better} Hamiltonian form and present our results, we rewrite \eqref{system1} into the following form:
  	\begin{equation}\label{system2}
  	  \begin{cases}
  u_t=-\rho,\\
  \rho_t=-u_{xx}+u+\al uv+\beta|u|^2u,\\
v_t=n_x,\\
      n_{t}-v_{x}=(|u|^2)_{x}.
  \end{cases}  
  	\end{equation}
 So, by  solitons of \eqref{system2}, we refer to solutions of the form 
$$
\begin{aligned}
& u(x, t)=e^{-\ii \omega t} \tilde{\phi}_{\omega,v}(x-c t), \quad v(x, t)=\psi_{\omega, c}(x-c t), \\
& \rho(x, t)=e^{-\ii \omega t} \tilde{\rho}_{\omega, c}(x-c t), \quad n(x, t)=\varphi_{\omega, c}(x-c t), 
\end{aligned}
$$
where $\tilde{\phi}_{\omega,v}(x-ct)=e^{\ii \frac{\omega c}{1-c^2}(x-c t)} \phi_{\omega, c}(x-c t)$ and $\tilde{\rho}_{\omega, c}(x-ct)=e^{\ii \frac{\omega c}{1-c^2}(x-c t)}\rho_{\omega, c}(x-c t)$. 
Hence, \begin{equation}\label{formsolitons}
    \begin{aligned}
& u(x, t)=e^{\ii \left( \frac{\omega c}{1-c^2}x-\frac{\omega}{1-c^2} t\right)} \phi_{\omega, c}(x-c t), \quad v(x, t)=\psi_{\omega, c}(x-c t), \\
& \rho(x, t)=e^{\ii \left( \frac{\omega c}{1-c^2}x-\frac{\omega}{1-c^2} t\right)}\rho_{\omega, c}(x-c t), \quad n(x, t)=\varphi_{\omega, c}(x-c t).
\end{aligned}
\end{equation}
So, we have
\begin{equation}\label{system3}
  \left\{\begin{array}{l}
\tilde{\rho}_{\omega, c}=\ii \omega \tilde{\phi}_{\omega, c}+c \tilde{\phi}_{\omega, c}^{\prime},  \\
-\ii \omega \tilde{\rho}_{\omega, c}-c \tilde{\rho}_{\omega, c}^{\prime}=-\tilde{\phi}_{\omega, c}^{\prime \prime}+\tilde{\phi}_{\omega, c}+\alpha \tilde{\phi}_{\omega, c} \psi_{\omega, c}+\beta\left|\tilde{\phi}_{\omega, c}\right|^{2} \tilde{\phi}_{\omega, c}, \\
-c\psi_{\omega, c}^{\prime}=\varphi_{\omega, c}^{\prime}, \\
-c \varphi_{\omega, c}^{\prime}=\psi_{\omega, c}^{\prime}+\left(\left|\tilde{\phi}_{\omega, c}\right|^{2}\right)'.
\end{array}\right.  
\end{equation}
By some calculations, we can obtain that
\begin{equation}\label{equation0}
\psi_{\omega, c}=-\frac{1}{1-c^{2}} \phi_{\omega, c}^{2} 
\end{equation}
and 
\begin{equation}\label{equation1}
\phi_{\omega, c}^{\prime \prime}-\frac{1-c^{2}-\omega^{2}}{\left(1-c^{2}\right)^{2}} \phi_{\omega, c}+\frac{\alpha-\beta\left(1-c^{2}\right)}{\left(1-c^{2}\right)^{2}} \phi_{\omega, c}^{3}=0.
\end{equation}

\noindent Then, let 
\[
I = \frac{1 - c^2 - \omega^2}{(1 - c^2)^2}, \quad m = -\frac{\alpha - \beta (1 - c^2)}{(1 - c^2)^2},
\]
so that equation \eqref{equation1} becomes the following linear equation:
\begin{equation}\label{equation2}
\phi_{\omega, c}^{\prime\prime} - I \phi_{\omega, c} - m \phi_{\omega, c}^3 = 0.
\end{equation}

\noindent Then, by applying the first integral method, we obtain the following explicit exact solution to equation \eqref{equation2}:
\begin{equation}\label{form1}
\phi_{\omega, c}(\xi) = \sqrt{\frac{2(1 - c^2 - \omega^2)}{\alpha - \beta(1 - c^2)}} \operatorname{sech}\left(\frac{\sqrt{1 - c^2 - \omega^2}}{1 - c^2} \xi \right),
\end{equation}
as long as the parameters satisfy \begin{equation}\label{cond1}
    \begin{aligned}
    &1 - c^2 - \omega^2 > 0 \\& \alpha - \beta(1 - c^2) > 0.
\end{aligned}
\end{equation} 
\noindent Hence, we obtain the following lemma.
\begin{proposition}
 For any real constants $\alpha, \beta, \omega, c$ satisfying the condition \eqref{cond1}, equation \eqref{equation1} has a unique bounded positive analytic solution of the form \eqref{form1}. In particular, when $\alpha = 1$ and $\beta = 0$, the solution to equation \eqref{equation1} is given by:
\begin{equation*}
\phi_{\omega, c}(\xi) = \sqrt{2(1 - c^2 - \omega^2)} \operatorname{sech}\left(\frac{\sqrt{1 - c^2 - \omega^2}}{1 - c^2} \xi\right).
\end{equation*}

\end{proposition}
Additionally, we also have the following result (see \cite[Theorem 8.1.1]{cazenave}).
\begin{proposition}\label{decaimentoquadratico}
			Consider $\Phi \in H^1(\R) $ a solution of \eqref{equation1}.
			Then,
			\begin{itemize}
				\item[(i)] $\Phi \in C^{2}(\R)$ and $\left|\partial_x^{\beta} \Phi(x)\right|\stackrel{|x|\to \infty}{\longrightarrow } 0$ for all $|\beta| \leq 2$.
				\item[(ii)] For any  $0<\epsilon <1$,
				\[
				e^{\epsilon|x|} \left(\left|\Phi(x)\right|+\left|\partial_{x} \Phi(x)\right|\right) \in L^{\infty}(\R).
				\]    
			\end{itemize} 
		\end{proposition} 
Thus from the above discussion, we write the explicit form of the solitons which are the subject of our study. In fact, 
\begin{equation}\label{solitons1}
\begin{aligned}
\vec{\Phi}_{\omega, c}(x)&=\left(\tilde{\phi}_{\omega, c}(x), \tilde{\rho}_{\omega, c}(x), \psi_{\omega, c}(x), \varphi_{\omega, c}(x)\right) \\
&=\left(e^{\ii\theta x}\phi_{\omega, c}(x), e^{\ii\theta x} \rho_{\omega, c}(x), \psi_{\omega, c}(x), \varphi_{\omega, c}(x)\right) ,
\end{aligned}
\end{equation}
where $\theta=\frac{\omega c}{1-c^2}$ and $\phi_{\omega, c}(x), \rho_{\omega, c}(x), \psi_{\omega, c}(x), \varphi_{\omega, c}(x)$ defined by

$$
\left\{\begin{array}{l}
\phi_{\omega, c}(x)=\sqrt{\frac{2\left(1-c^{2}-\omega^{2}\right)}{\alpha-\beta\left(1-c^{2}\right)}} \operatorname{sech}\left(\frac{\sqrt{1-c^{2}-\omega^{2}}}{1-c^{2}} x\right),  \\
\rho_{\omega, c}(x)=(\ii \omega+c\partial_x) \phi_{\omega, c}(x), \\
\psi_{\omega, c}(x)=-\frac{2(1-c^{2}-\omega^{2})}{\left(1-c^{2}\right)(\alpha-\beta\left(1-c^{2}\right))} \operatorname{sech}^{2}\left(\frac{\sqrt{1-c^{2}-\omega^{2}}}{1-c^{2}} x\right), \\
\varphi_{\omega, c}(x)=\frac{\left.2 c\left(1-c^{2}-\omega^{2}\right)\right]}{\left(1-c^{2}\right)\left[\alpha-\beta\left(1-c^{2}\right)\right]} \operatorname{sech}^{2}\left(\frac{\sqrt{1-c^{2}-\omega^{2}}}{1-c^{2}} x\right).
\end{array}\right.
$$
 
     % well-posedness   $C([0,T]; H^{1}\times L^2\times L^2\times L^2 )$

% Moreover, in \cite{hss2013}, when $\beta=0$ we have well-posedness,namely,  $C([0,T]; H^{s}(\R)\times H^{s-1}(\R)\times H^{s-1}(\R)\times H^{s-1}(\R) )$ for $s>1/2$ and the cubic Klein-Gorodn is well-posed (\cite{Nakoz}) in $C([0,T]; H^{s}(\R)\times H^{s-1}(\R))$ with $s\geq1/2$.
 
\bigskip
In this article, we are interested in the existence of multi-soliton solutions of \eqref{system2} in the associated     energy space
$$X:=H^1_{\rm complex}(\R)\times L^2_{\rm complex}(\R)\times L^2_{\rm real}(\R)\times L^2_{\rm real}(\R).$$
\begin{definition}\label{def-ms-1}
For  $\alpha, \beta,  \omega, c$ satisfying the condition \eqref{cond1} and   $\vec{\phi}_{\omega,c}$  solution as in \eqref{solitons1}, we say that a solitary wave solution of \eqref{system2} traveling along the line $x = x_0 + ct$ is defined by $$\vec{R}_0=\left(e^{\ii\lambda(x,t)}\phi_{\omega, c}(x-ct-x_0), e^{\ii\lambda(x,t)}\rho_{\omega, c}(x-ct-x_0), \psi_{\omega, c}(x-ct-x_0), \varphi_{\omega, c}(x-ct-x_0)\right),$$
where $\lambda(x,t)=(\theta( x-ct)-\omega t +\gamma_0)$ and $\gamma_0=-\theta x_0$.
\end{definition}

\medskip
In this sense, we observe the strict relationship that exists between each component of  $\vec{\Phi}_{\omega, c}$, so that for any   $0<\epsilon <1$, we can obtain the following estimate:
\begin{equation}\label{decay}    
				 \left|\vec{\Phi}_{\omega, c}(x)\right|+\left|\partial_{x} \vec{\Phi}_{\omega, c}(x)\right|\leq Ce^{-\epsilon \sqrt{I}|x|},\qquad I=\frac{1-c^{2}-\omega^{2}}{(1-c^{2})^2}.
	\end{equation}

\noindent Now,   for $j\in \mathbb{N}$,   $c_{j}\in (-1,1)$, and $x_{j}, \gamma_j,  \omega_{j}  \in \R$,  let $\Vec{\Phi}_{\omega_{j},c_j} \in X$ be a solution as in \eqref{solitons1}. We will denote the corresponding soliton associated with \eqref{system2} the solitary along the line $x=x_j-c_j t$  is defined by $\Vec{R}_j=(R_j^{(1)},R_j^{(2)},R_j^{(3)},R_j^{(4)})$ such that \begin{equation}\label{defsolitons}
\Vec{R}_j=A_j(t)\Vec{\Phi}_{\omega_j,c_j}(x-c_jt-x_j),
			\end{equation}
   where $A_j(t)\vec{a}=\left(e^{-i\omega_j  t} a_1,e^{-i\omega_j  t}a_2,a_3,a_4\right)$.

\bigskip
 Our main goal is to demonstrate the existence of solutions to the system \eqref{system2} that exhibit asymptotic behavior similar to a sum of different solitons. In other words, our aim is to find solutions that can be approximated by a combination of solitary waves with different wave speeds.
			
			\begin{definition}\label{multisolitons}
				Let $N \in \mathbb{N} \setminus\{0,1\}$. Consider $c_1,c_2,\ldots,c_N \in (-1,1)$, $\omega_{1}, \ldots, \omega_{N}\in\R$, $x_{1}, \ldots, x_{N} \in \R$, $\gamma_{1}, \ldots, \gamma_{N} \in \R$, and $\Vec{\Phi}_{\omega_{1}}, \ldots, \Vec{\Phi}_{\omega_{N}} \in X$ be solutions as in \eqref{solitons1}. We denote by $\Vec{R}=(R_1,R_2,R_3,R_4)$, as in \eqref{defsolitons}, such that 
\begin{equation}\label{multisol}
    R_{k}(x,t):=\sum_{j=1}^{N} R_{j}^{(k)}(x,t), \, \,\, \,  \text{for   $k=1,2,3,4$.}
\end{equation}
A multi-soliton is a solution $\Vec{u}$ of \eqref{system2} for which there exists $T_{0} \in \R$ such that $\Vec{u}$ is defined on $[T_{0}, +\infty)$ and
				\[\lim_{t \rightarrow +\infty}\|\Vec{u}(t)-\Vec{R}(t)\|_{X}=0.\]
			\end{definition}			
			%\red{At this point, it is important to note that the developed theory can be applied to any bound state, i.e., it does not matter whether the solution of the system \eqref{system2} is a ground state or an excited state.}

We will lay the groundwork by presenting the main result and then constructing arguments to support this finding.
\begin{theorem}\label{maintheorem}
Let \( j \in \{1, 2, \ldots, N\} \), with \( |c_j| \leq 1 \), and let \( x_j, c_j, \omega_j \in \mathbb{R} \) satisfy condition \eqref{cond1}. Assume that \( \left( \Phi_{\omega_j} \right) \) are the solitary waves defined in \eqref{solitons1}. If \( \vec{R} \) is as defined in Definition \ref{multisolitons}, and

\[
\omega_{\star} := \frac{1}{256}\min \left\{ \frac{1 - c_j^2 - \omega_j^2}{(1 - c_j^2)^2}  : j = 1, \ldots, N \right\},
\]
\[
c_{\star} := \min \left\{ \left| c_j - c_k \right| : j, k = 1, \ldots, N, \, j \neq k \right\},
\]
then there exists \( T_0 \in \mathbb{R} \) and a solution \( \vec{u} \) to \eqref{system2} defined on \( \left[ T_0, +\infty \right) \) such that for all \( t \in \left[ T_0, +\infty \right) \), the following estimate holds:
\begin{equation}\label{desprin}
\left\| \vec{u}(t) - \vec{R}(t) \right\|_X \leq e^{- \omega_{\star}^{ \frac{1}{2}}} c_{\star} t .
\end{equation}
 \end{theorem}

 The main tools used to establish this result combine various techniques, including energy estimates, bootstrap, modulation, and localization arguments. More precisely, to proceed with the proof of the main theorem, we first define a sequence of approximate solutions. After deriving some estimates and taking the limit, we construct the desired solution. Specifically, let ${T^{n}}$ be an increasing sequence converging to infinity. Our strategy involves solving a final-value problem associated with \eqref{system2}, where the final data consists of a pair of solitons at time $T^{n}$. For each $n \in \mathbb{N}$, let $\Vec{u}^{n}$ denote the solution of \eqref{system2} defined on the interval $(T_{n}^{\star}, T^{n}]$, where $T_{n}^{\star}$ is the maximum time of existence for each $n$, and such that $\Vec{u}^{n}(T^{n}) = \Vec{R}(T^{n})$. Unlike the approach in \cite{cotem} for establishing the existence of multisolitons in the Klein–Gordon equation, where the challenge arises from the directions governed by the coercivity property, necessitating the identification of eigenfunctions of the coercivity operator to obtain new control estimates, in this system, due to its structure, we can derive a sharper result. Specifically, directional constraints can be omitted by applying orthogonality arguments, which are derived from localization and modulation techniques. Consequently, we establish the following: 
 \[
\left\langle S_j^{\prime \prime}\left(\vec{R}\right) {\vec\eta}, \vec{\eta}\right\rangle \geq C\|{\vec\eta}\|_{ X }^2,  
\]
under the condition \begin{equation}\label{cond0} \left\langle\vec{\eta}, \partial_x \vec{\tilde{D}}\right\rangle=\left\langle\vec{\eta}, \vec{\tilde{\Gamma}}\right\rangle=\left\langle\vec{\eta}, \vec{\tilde{\Psi}}\right\rangle=0, \end{equation} where
$\vec{\tilde{D}}=\sum_{j=1}^{N}\left(\partial_{x}R_{j}^{(1)},0,0,0\right) $ $\vec{\tilde{\Gamma}}=\sum_{j=1}^{N}\left(iR_{j}^{(1)},0,0,0\right) $  and  $\vec{\tilde{\Psi}}=\sum_{j=1}^{N}\left(\tilde{\Psi}_j,0,0,0\right) $ (See Section \ref{modul-section}). 
			
			We will  establish that our approximate solutions $\Vec{u}^{n}$ indeed satisfy the conclusion of the main theorem.
\begin{proposition}[Uniform Estimates]\label{UniformEstimates}
				There exist $T_{0}\in \R$ and $n_{0} \in\N$ such that, for every $n \geq n_{0}$, each approximate solution $\Vec{u}^{n}$ is defined in $[T_{0},T^{n}]$ and for all $t \in [T_{0},T^{n}]$
				\begin{equation}\label{desiesti}
					\|\Vec{u}^{n}(t)-\Vec{R}(t)\|_{X}\leq e^{-\omega_{\star}^{\frac12}c_{\star} t}.
				\end{equation}
			\end{proposition}
   
   To demonstrate the previous results, it is important to use a bootstrap argument through which we will improve, in a certain way, the existence interval of the multi-soliton solutions. More specifically,
   \begin{proposition}
       \label{Bootstrap3}
There exist $T_{0}\in \R$ depending only on $c_{j}$ for $j \in \{1,\ldots,N\}$, with $c_{\star}$ as in Theorem \ref{maintheorem}, and $n_{0} \in\N$ such that, if for all $n \geq n_{0}$, every approximate solution $\Vec{u}^n$ is defined on $[T_{0},T^{n}]$, and for all $t \in[t_{0},T^{n}]$, with $t_{0}\in [T_{0},T^{n}]$,
\begin{equation*}
\|\Vec{u}^{n}(t)- \Vec{R}(t)\|_{X}\leq  e^{-\omega_{\star}^{\frac12}c_{\star} t},
\end{equation*}
then for all $t\in[t_{0},T^{n}]$, 
\begin{equation*}
\|\Vec{u}^{n}(t)- \Vec{R}(t)\|_{X}\leq \frac{1}{2} e^{-\omega_{\star}^{\frac12}c_{\star} t}.
\end{equation*}
   \end{proposition}
From this, in the following section, we will seek to construct conditions that allow us to prove the previous lemma.

We close this section by establishing a result that ensures the well-posedness of system \eqref{system2}.  

Notice that the solutions $\vec u(t)$ of \eqref{system2} with the initial data $\vec u_0=(u_0,\rho_0,v_0,n_0)^T$ can be equivalently found from the following integral equality
   \begin{equation}\label{integral-form}
       \vec u(t)=\begin{pmatrix}
           u(t)\\
           \rho(t)\\
           v(t)\\n(t)
       \end{pmatrix}
       =\bm{G}(t)\vec u_0+\int_0^t\bm{G}(t-t')\vec F(\vec u(t'))\dd t',
   \end{equation}
   where $\bm{G}={\rm diagonal}(\bm{G}_1(t),\bm{G}_2(t))$,
   \[  
   \widehat{\bm{G}}_1(t)=
   \begin{pmatrix}
    \cos(t\sqrt{1+\xi^2})&-\frac{1}{ \sqrt{1+\xi^2}}\sin(t\sqrt{1+\xi^2})\\
    \sqrt{1+\xi^2}\sin(t\sqrt{1+\xi^2})&\cos(t\sqrt{1+\xi^2})
   \end{pmatrix},
   \]
   \[
   \widehat{\bm{G}}_2(t)=
   \begin{pmatrix}
    \cos(t\xi)&\ii\sin(t\xi)\\
    \ii\sin(t\xi)&\cos(t\xi)
   \end{pmatrix},
   \]
   and
  $\vec F(\vec u)=(0,\al uv+\beta|u|^2u,0,(|u|^2)_x)^T$.
  \begin{theorem}\label{wellpo}
   Let $s>1/2$, $r\in\rr$ and $Y^s=H^s(\rr)\,(\text{or}\,\dot{H}^s(\rr))$.  For any initial data $\vec u_0\in X^{s,r} :=Y^{s} \times Y^{s-1} \times Y^{r} \times Y^{r}  $, there exists a time existence $T>0$ such that \eqref{system2} is well-posed in $X^{s,r}$ provided that   $(r,s)\in \sett{r< s, \;r+1/2\leq2s}\cup  \sett{r< s, \;r+1/2\leq2s}$ and
   $(r,s)\in \sett{s\leq r+1, \;r>-1/2}\cup  \sett{s< r+1, \;r\geq-1/2}$. 
      Moreover, $Q_1$, $Q_2$, and the energy functional  $E$ are time-invariant, where
\begin{equation}\label{ener1}
E( \vec u)=\int_{\R}\left(|u|^{2}+\left|\rho^{2}\right|+\left|u_{x}\right|^{2}+\alpha|u|^{2} v+\frac{\beta}{2}|u|^{4}+\frac{\alpha}{2} v^{2}+\frac{\alpha}{2} n^{2}\right) \, d x \end{equation}
for the solution $\vec u=(u,\rho,v,n)^T$.
  \end{theorem}
  \begin{proof}
According the unitary group of $\bm{G}$, the proof is derived by combining the results given in \cite[Theorem 1]{hss2013} and
 \cite[Theorem 1.1]{Nakoz}.
\end{proof}

 Using the above notations, we see that   $X^{1,0}=X$.

\begin{remark}
Suppose $s$ is a nonnegative integer.  By employing the so-called continuous method and delicate a priori estimates, it was shown in \cite{zsp} (see also \cite{guoyua}) that for any fixed $\Vec{u}_{0}=\left(u_{0}, \rho_{0}, v_{0}, n_{0}\right) \in H^{s+1}(\R) \times H^{s}(\R) \times H^{s}(\R) \times H^{s}(\R)$, there exists a unique smooth solution globally in time for \eqref{system2}, which satisfies 
\[(u(x, t), \rho(x, t), v(x, t), n(x, t)) \in L^{\infty}\left([0, T] ; H^{s+1}(\R)\times H^{s}(\R)\times H^{s}(\R)\times H^{s}(\R)\right)  
\]
with $(u(0), \rho(0), v(0), n(0))=\left(u_{0}, \rho_{0}, v_{0}, n_{0}\right)$. 
\end{remark}
\section{Construction of the solution }
This section has the sole purpose of proving Proposition \ref{Bootstrap3}, using some localization and parameter modulation arguments. In fact, to construct the multi-solitons, we note that \eqref{system2} can be written as the following Hamiltonian system
\begin{equation*}
\frac{d \Vec{u}}{d t}=J E^{\prime}(\Vec{u}) ,
\end{equation*}
where $\Vec{u}=(u, \rho, v, n)$, $J$ is a skew-symmetrically linear operator defined by
$$
J=\left(\begin{array}{cccc}
0 & -\frac{1}{2} & 0 & 0 \\
\frac{1}{2} & 0 & 0 & 0 \\
0 & 0 & 0 & \frac{1}{\alpha} \partial x \\
0 & 0 & \frac{1}{\alpha} \partial x & 0
\end{array}\right)
$$

and

\begin{equation}\label{energy1}
 E^{\prime}(\Vec{u})=\left(\begin{array}{c}
-2 u_{x x}+2 u+2 \alpha u v+2 \beta|u|^{2} u  \tag{34}\\
2 \rho \\
\alpha|u|^{2}+\alpha v \\
\alpha n
\end{array}\right) . 
\end{equation}

Differentiating \eqref{energy1} with respect to $\vec u$, we have
\begin{equation}\label{energy2}
 E^{\prime \prime}(\Vec{u}) \vec{\eta}=\left(\begin{array}{c}
\left(-\partial_{x}^{2}+2+2 \alpha v+2 \beta|u|^{2}\right) \eta_{1}+4 \beta u R e\left(u \overline{\eta_{1}}\right)+2 \alpha u \eta_{3}  \tag{35}\\
2 \eta_{2} \\
2 \alpha \Re\left(u \overline{\eta_{1}}\right)+\alpha \eta_{3} \\
\alpha \eta_{4}
\end{array}\right)   
\end{equation}
where   $\vec{\eta}=\left(\eta_{1}, \eta_{2}, \eta_{3}, \eta_{4}\right)$.
We consider 

$$
B_{1}=\left(\begin{array}{cccc}
0 & -2 \partial_x & 0 & 0 \\
2 \partial_x & 0 & 0 & 0 \\
0 & 0 & 0 & -\alpha \\
0 & 0 & -\alpha & 0
\end{array}\right) \text { and } B_{2}=\left(\begin{array}{cccc}
0 & -2 i & 0 & 0 \\
2 i & 0 & 0 & 0 \\
0 & 0 & 0 & 0 \\
0 & 0 & 0 & 0
\end{array}\right)
$$

and  we define the conserved functionals $Q_{1}$ and $Q_{2}$ as the following
\begin{equation}\label{mom1}
Q_{1}(\Vec{u})=\frac{1}{2}\left\langle B_{1} \Vec{u}, \Vec{u}\right\rangle=2 \Re \int_{\R} u_{x} \bar{\rho}\,dx-\alpha \int_{\R} n v \,d x
\end{equation}
and
\begin{equation}\label{mom2}
Q_{2}(\Vec{u})=\frac{1}{2}\left\langle B_{2} \Vec{u}, \Vec{u}\right\rangle=2 \Im \int_{\R} \bar{u} \rho \,d x.
\end{equation}
Hence, for any $t \in \R$, we have $\Vec{u}(t)$ is a flow of \eqref{system2},
\begin{equation}\label{mom1and2}
Q_{1}(\Vec{u}(t))=Q_{1}(\Vec{u}(0)) \, \, \text { and } \, \,  Q_{2}(\Vec{u}(t))=Q_{2}(\Vec{u}(0)).
\end{equation}
Differentiating \eqref{mom1and2} with respect to $\Vec{u}$, respectively, we have
$$
Q_{1}^{\prime}(\Vec{u})=B_{1} \Vec{u}=\left(\begin{array}{c}
-2 \rho_{x} \\
2 u_{x} \\
-\alpha n \\
-\alpha v
\end{array}\right), \quad Q_{1}^{\prime \prime}(\Vec{u})=B_{1}=\left(\begin{array}{cccc}
0 & -2 \partial_x & 0 & 0 \\
2 \partial_x & 0 & 0 & 0 \\
0 & 0 & 0 & -\alpha \\
0 & 0 & -\alpha & 0
\end{array}\right)
$$
and
$$
Q_{2}^{\prime}(\Vec{u})=B_{2} \Vec{u}=\left(\begin{array}{c}
-2 i \rho  \\
2 i u \\
0 \\
0
\end{array}\right), \quad Q_{2}^{\prime \prime}(\Vec{u})=B_{2}=\left(\begin{array}{cccc}
0 & -2 i & 0 & 0 \\
2 i & 0 & 0 & 0 \\
0 & 0 & 0 & 0 \\
0 & 0 & 0 & 0
\end{array}\right).
$$

			Now, we define some operators that will be used later on. Indeed,
			let us $\mathcal{S}_{j}:  X \to \R$ the functional defined by
			\begin{equation}\label{defSj}
				\mathcal{S}_{j}(\Vec{u}):=E(\Vec{u})-\omega_j Q_2(\Vec{u})-c_j Q_1(\Vec{u}), \, \, \quad \text{$j=1,\cdots, N$}.
			\end{equation}

			Note that by the Sobolev embedding, these functionals are well-defined.
  
We denote $\Vec{R}_j=(R_j^{(1)}, R_j^{(2)}, R_j^{(3)}, R_j^{(4)})$ as in \eqref{multisol}, then it follows from \eqref{system3} that
\begin{equation}\label{pointcrit}
  \mathcal{S}_j^{\prime}(\Vec{R}_j)=\left(\begin{array}{c}
2 R_j^{(1)}-2 \partial_{xx}R_j^{(1)}+2 \alpha R_j^{(1)} R_j^{(3)}+2 \beta|R_j^{(1)}|^{2} R_j^{(1)}+2 c_j \partial_{x}R_j^{(2)}+2 i \omega_j R_j^{(1)} \\
2R_j^{(2)}-2 c_j \partial_{x}R_j^{(1)}-2 i \omega_j R_j^{(1)} \\
\alpha|R_j^{(1)}|^{2}+\alpha R_j^{(3)}+\alpha c_j R_j^{(4)} \\
\alpha R_j^{(4)}+\alpha c_j R_j^{(3)}
\end{array}\right)=\left(\begin{array}{l}
0 \\
0 \\
0 \\
0
\end{array}\right).  
\end{equation}
Now we define, for all $j \in \{1,\dots,N\}$, an operator $H_{j}$ from $X$ to $X^{\star}$ by
\begin{equation}\label{defHj}
H_{j}=\mathcal{S}_j^{\prime \prime}(\Vec{R}_j), 
\end{equation}

that is,
$$
\begin{aligned}
H_{j} \vec{\eta}&=  \left(\begin{array}{c}
2\left(-\partial_{x}^{2}+1+\alpha R_j^{(3)}+\beta|R_j^{(1)}|^{2}\right) \eta_{1}+4 \beta R_j^{(1)} \R e\left(R_j^{(1)} \overline{\eta_{1}}\right)+2 \alpha R_j^{(1)} \eta_{3}+2 c_j \partial_{x}\eta_{2}+2 i \omega_j \eta_{2} \\
2\left(\eta_{2}-c_j \partial_{x} \eta_{1}-i \omega_j \eta_{1}\right) \\
2 \alpha \Re\left(R_j^{(1)}\overline{\eta_{1}}\right)+\alpha \eta_{3}+\alpha c_j \eta_{4} \\
\alpha \eta_{4}+\alpha c_j \eta_{3}
\end{array}\right), 
\end{aligned}
$$
where $\vec{\eta}=\left(\eta_{1}, \eta_{2}, \eta_{3}, \eta_{4}\right)$. 

 Note initially from \eqref{system3}  for all $j \in \{1,\dots,N\}$ that
$$
\begin{aligned}
&H_{j} \partial_{x}\vec{R}_{j}\\&=  \left(\begin{array}{c}
2\left(-\partial_{x}^{2}+1+\alpha R_j^{(3)}+\beta|R_j^{(1)}|^{2}\right) \partial_{x}R_{j}^{(1)}+4 \beta R_j^{(1)} \Re\left(R_j^{(1)} \partial_{x}\overline{R_{j}^{(1)}}\right)+2 \alpha R_j^{(1)} \partial_{x}R_{j}^{(3)}+2 c_j \partial_{x}^2 R_{j}^{(2)}+2 i \omega_j \partial_{x}R_{j}^{(2)} \\
2\left(\partial_{x}R_{j}^{(2)}-c_j  \partial_{x}^2 R_{j}^{(1)}-i \omega_j \partial_{x}R_{j}^{(1)}\right) \\
2 \alpha \Re\left(R_j^{(1)}\partial_{x}\overline{R_{j}^{(1)}}\right)+\alpha \partial_{x}R_{j}^{(3)}+\alpha c_j \partial_{x}R_{j}^{(4)} \\
\alpha \partial_{x}R_{j}^{(4)}+\alpha c_j \partial_{x}R_{j}^{(3)}
\end{array}\right)\\
&=  \left(\begin{array}{c}
0\\0\\0\\0
\end{array}\right),
\end{aligned}
$$
and similarly, for all $j \in \{1,\dots,N\}$,  we have that for $\vec{Y}_j=(iR_j^{(1)},iR_j^{(2)},0,0)$ we have
$$
\begin{aligned}
&H_{j} \vec{Y}_j\\&=  \left(\begin{array}{c}
2\left(-\partial_{x}^{2}+1+\alpha R_j^{(3)}+\beta|R_j^{(1)}|^{2}\right) iR_{j}^{(1)}+4 \beta R_j^{(1)} \Re\left(R_j^{(1)} \overline{iR_{j}^{(1)}}\right)+2 ic_j \partial_{x} R_{j}^{(2)}-2  \omega_j R_{j}^{(2)} \\
2\left(iR_{j}^{(2)}-c_j i   R_{j}^{(1)}+\omega_j R_{j}^{(1)}\right) \\
2 \alpha \Re\left(R_j^{(1)}\overline{iR_{j}^{(1)}}\right) \\
0
\end{array}\right)\\
&=  \left(\begin{array}{c}
0\\0\\0\\0
\end{array}\right).
\end{aligned}
$$
Then based on Lemma 2 of \cite{yin}, we can establish the following result. 
 
\begin{lemma}\label{kernel}
We have for system \eqref{system1} that, for all $j \in \{1,\dots,N\}$, 
    $\mathcal{S}_j^{\prime}(\vec{R}_j)=0$
    and $$Ker(H_j)=\text{span}\{\partial_{x}\vec{R}_j,\vec{Y}_{j}\},
    $$
    where $\mathcal{S}_j, H_j$ are defined by \eqref{defSj} and \eqref{defHj}, respectively.
\end{lemma} 

 In addition, for $ j \in \{1,\dots,N\}$ fixed, observe that $H_{j}$ is self-adjoint in the sense that $H_{j}^{\star}=H_{j}$. Then, the spectrum of $H_{j}$ consists of the real numbers $\lambda$ such that $H_{j}-\lambda I$ is not invertible. We claim that $\lambda=0$ belongs to the spectrum of $H_{j}$. 

We fix $j \in \N$, then for any $\vec{\eta} \in X$, where $\eta_{1}=e^{i\lambda_j(x,t)}\left(\eta_{11}+i \eta_{12}\right), \eta_{2}=e^{i\lambda_j(x,t)}\left(\eta_{21}+i \eta_{22}\right)$ where $\lambda_j(x,t)=\frac{\omega_j c_j}{1-c_j^{2}} (x-c_jt) +(\gamma_j-\omega_jt)$. Then,  we have
$$
\begin{aligned}
&\left\langle H_{j} \vec{\eta}, \vec{\eta}\right\rangle\\&= 2 \Re \int_{\R}\left[\left(-\partial_{x}^{2} \eta_{1}+\eta_{1}+\alpha R_j^{(3)} \eta_{1}+\beta|R_j^{(1)}|^{2} \eta_{1}+2 \beta R_j^{(1)} \Re\left(R_j^{(1)} \overline{\eta_{1}}\right)\right) \overline{\eta_{1}}+2\left(i \omega_j+c_j \partial_{x}\right) \eta_{2} \overline{\eta_{1}}+\eta_{2} \overline{\eta_{2}}\right] \,d x \\
&\quad+4 \alpha \Re \int_{\R} R_j^{(1)} \tilde{\eta_{1}} \eta_{3} \,d x+\alpha \int_{\R}\left(\eta_{3}^{2}+2 c_j \eta_{3} \eta_{4}+\eta_{4}^{2}\right) \,d x \\
&= 2 \Re \int_{\R}\left[-\left(1-c_j^{2}\right) \partial_{x}^{2} \eta_{1}+\eta_{1}+2 i \omega_j c_j \partial_{x} \eta_{1}+\left(\beta (\phi_{\omega_j})^{2}-\frac{\alpha (\phi_{\omega_j})^{2}}{1-c_j^{2}}\right) \eta_{1}-\omega_j^{2} \eta_{1}+2 \beta R_j^{(1)} \Re\left(R_j^{(1)} \overline{\eta_{1}}\right)\right] \overline{\eta_{1}} \,d x \\
&\quad+2 \int_{\R}\left|\eta_{2}-\left(i \omega_j+c_j \partial_{x}\right) \eta_{1}\right|^{2}\, d x+4 \alpha \Re \int_{\R} R_j^{(1)} \overline{\eta_{1}} \eta_{3}\, d x+\alpha \int_{\R}\left[\left(c_j \eta_{3}+\eta_{4}\right)^{2}+\left(1-c_j^{2}\right) \eta_{3}^{2}\right] \,d x \\   
&= 2 \Re \int_{\R}\left[-\left(1-c_j^{2}\right) \partial_{x}^{2} z_{1}+\frac{1-c_j^{2}-\omega_j^{2}}{1-c_j^{2}} z_{1}-\frac{\left(\alpha-\beta\left(1-c_j^{2}\right)\right) \phi_{\omega_j}^{2}}{1-c_j^{2}} z_{1}+2 \beta \phi_{\omega_j} \Re\left(\phi_{\omega_j} \overline{z_{1}}\right)\right] \overline{z_{1}} \,d x \\
&\quad+2 \int_{\R}\left|z_{2}-\frac{i \omega_j}{1-c_j^{2}} z_{1}-c_j \partial_{x} z_{1}\right|^{2} \,d x+4 \alpha \Re\int_{\R} \overline{z_{1}}\phi_{\omega_j} \eta_{3} \,d x+\alpha \int_{\R}\left[\left(c_j \eta_{3}+\eta_{4}\right)^{2}+\left(1-c_j^{2}\right) \eta_{3}^{2}\right] \,d x,
\end{aligned}
$$
where we are considering $\phi_{\omega_j,c_j}=\phi_{\omega_j}$ and the fact that $\psi_{\omega_j,c_j}=-\frac{1}{1-c_j^2}(\phi_{\omega_j,c_j})^2$.
Therefore,
$$
\begin{aligned}
 &\left\langle H_{j} \vec{\eta}, \vec{\eta}\right\rangle \\
&= 2 \int_{\R}\left[-\left(1-c_j^{2}\right) \partial_{x}^{2} y_{1}+\frac{1-c_j^{2}-\omega^{2}}{1-c_j^{2}} y_{1}-\frac{3\left(\alpha-\beta\left(1-c_j^{2}\right)\right) \phi_{\omega_j}^{2}}{1-c_j^{2}} y_{1}\right] y_{1}\, d x+2 \int_{\R}\left[-\left(1-c_j^{2}\right) \partial_{x}^{2} y_{2}\right. \\
&\left.\quad+\frac{1-c_j^{2}-\omega_j^{2}}{1-c_j^{2}} y_{2}-\frac{\left(\alpha-\beta\left(1-c_j^{2}\right)\right) \phi_{\omega_j}^{2}}{1-c_j^{2}} y_{2}\right] y_{2} \,d x+2 \int_{\R}\left|z_{2}-\frac{i \omega}{1-c_j^{2}} z_{1}-c_j \partial_{x} z_{1}\right|^{2}\, d x \\
&\quad +\alpha \int_{\R}\left[\frac{4 \phi_{\omega_j}^{2} y_{1}^{2}}{1-c_j^{2}}+4 \phi_{\omega_j} y_{1} \eta_{3}+\left(1-c_j^{2}\right) \eta_{3}^{2}+\left(c_j \eta_{3}+\eta_{4}\right)^{2}\right]\, d x \\
&= 2\left\langle L_{j}^1 y_{1}, y_{1}\right\rangle+2\left\langle L_{j}^2 y_{2}, y_{2}\right\rangle+2 \int_{\R}\left|z_{2}-\frac{i \omega_j}{1-c_j^{2}} z_{1}-c_j \partial_{x} z_{1}\right|^{2}\, d x+\alpha \int_{\R}\left[\left(c_j \eta_{3}+\eta_{4}\right)^{2}\right. \\
&\left.\quad+\left(\frac{2 \phi_{\omega_j} y_{1}}{\sqrt{1-c_j^{2}}}+\sqrt{1-c_j^{2}} \eta_{3}\right)^{2}\right] \,d x,
\end{aligned}$$
where $$L_{j}^1=-\left(1-c_j^{2}\right) \partial_{x}^{2}+\frac{1-c_j^{2}-\omega_j^{2}}{1-c^{2}}-\frac{3\left(\alpha-\beta\left(1-c_j^{2}\right)\right) \phi_{\omega_j}^{2}}{1-c_j^{2}}$$ and $$L_{j}^2=-\left(1-c_j^{2}\right) \partial_{x}^{2}+\frac{1-c_j^{2}-\omega^{2}}{1-c_j^{2}}-\frac{\left(\alpha-\beta\left(1-c_j^{2}\right)\right) \phi_{\omega_j}^2}{1-c_j^{2}}.$$ 
Hence,
 \begin{equation}\label{form2}
\begin{aligned}
        \left\langle H_{j} \vec{\eta}, \vec{\eta}\right\rangle&= 2\left\langle L_{j}^1 y_{1}, y_{1}\right\rangle+2\left\langle L_{j}^2 y_{2}, y_{2}\right\rangle+2 \int_{\R}\left|z_{2}-\frac{i \omega_j}{1-c_j^{2}} z_{1}-c_j \partial_{x} z_{1}\right|^{2}\, d x+\alpha \int_{\R}\left[\left(c_j \eta_{3}+\eta_{4}\right)^{2}\right. \\
&\left.\quad+\left(\frac{2 \phi_{\omega_j} y_{1}}{\sqrt{1-c_j^{2}}}+\sqrt{1-c_j^{2}} \eta_{3}\right)^{2}\right] \,d x.
\end{aligned}
\end{equation} 
By Weyl's essential spectral theorem, the essential spectra of $L_{j}^1$ and $L_{j}^2$ are $\sigma_{\text {ess}} L_{j}^1=\left[\frac{1-c_j^{2}-\omega_j^{2}}{1-c_j^{2}},+\infty\right)$ and $\sigma_{\text {ess }} L_{j}^2=\left[\frac{1-c_j^{2}-\omega_j^{2}}{1-c_j^{2}},+\infty\right)$, respectively.

Notice that by differentiating \eqref{equation1} with respect to $x$, we have
\begin{equation}
\left(\left(1-c_j^{2}\right) \partial_{x}^{2}-\frac{1-c_j^{2}-\omega_j^{2}}{1-c_j^{2}}+\frac{3\left(\alpha-\beta\left(1-c_j^{2}\right)\right) \phi_{\omega_j}^{2}}{1-c_j^{2}}\right) \partial_{x}\phi_{\omega_j}=0.
\end{equation}
Then we obtain $L_{j}^1 \partial_{x}\phi_{\omega_j}=0$. Moreover, from \eqref{equation1}, we have the following
\begin{equation}
\left(\left(1-c_j^{2}\right) \partial_{x}^{2}-\frac{1-c_j^{2}-\omega_j^{2}}{1-c_j^{2}}+\frac{\left(\alpha-\beta\left(1-c_j^{2}\right)\right) \phi_{\omega_j}^{2}}{1-c_j^{2}}\right) \phi_{\omega_j}=0,
\end{equation}
whence, $L_{j}^2 \phi_{\omega_j}=0$.
Since $\partial_x \phi_{\omega_j}$ has a unique zero at $x=0$, by using the Sturm-Liouville theorem, we know that zero is the second eigenvalue of $L_{j}^1$, and $L_{j}^1$ has exactly one strictly negative eigenvalue $ -\lambda_j^{2} $, with an eigenfunction $ \Psi_{\omega_j}$, that is,
\begin{equation*}
L_j^1  \Psi_{\omega_j} = -\lambda_j^{2}   \Psi_{\omega_j} .
\end{equation*}
 Note that for each $j \in \{1,\ldots,N\}$, the eigenvector $\Psi_{\omega_j}$ can be obtained explicitly as a solution of equation \eqref{equation2}, except for some changes in the constants, which allows us to establish that there is a certain similarity in the explicit form between the solutions $\phi_{\omega_j}$ and $\Psi_{\omega_j}$, except for some constants that differentiate them. This gives $\Psi_{\omega_j}$ certain properties similar to those of $\phi_{\omega_j}$, such as the property of decaying to infinity. 

Also, we know that \( \phi_{\omega_j} \) has a unique zero at \( x = 0 \). By applying the Sturm-Liouville theorem, we conclude that zero is the first eigenvalue of \( L_{j}^2 \), meaning that the operator \( L_{j}^2 \) is nonnegative.

Thus, from \cite{zsp}, we obtain the following results.

\begin{lemma}\label{lemma1}
 For any real functions $z \in H^{1}(\R)$ and for all $j \in \{1,\dots,N\}$ satisfying
\begin{equation*}
\langle z,  \Psi_{\omega_j} \rangle=\left\langle z, \partial_x  \phi_{\omega_j} \right\rangle=0
\end{equation*}
there exists a positive number $\delta_{1}>0$ such that
\begin{equation*}
\left\langle L_{j}^1 z, z\right\rangle \geq \delta_{1}\|z\|_{H^{1}(\R)}^{2}.
\end{equation*}
\end{lemma}

\begin{lemma}\label{lemma2}
For any real functions $z \in H^{1}(\R)$ and for all $j \in \{1,\dots,N\}$ satisfying
\begin{equation*}
\langle z,  \phi_{\omega_j} \rangle=0 
\end{equation*}
there exists a positive number $\delta_{2}>0$ such that
\begin{equation*}
\left\langle L_{j}^2 z, z\right\rangle \geq \delta_{2}\|z\|_{H^{1}(\R)}^{2}.
\end{equation*}
\end{lemma}
%Let
%\begin{equation*}
%P=\left\{p \in X :\, p=\left(e^{i \theta_j x}p_1, e^{i  \theta_j x}p_{2}, p_{3}, p_{4}\right),\left\langle \R e p_{1}, \red{\Psi_{\omega_j}}\right\rangle=\left\langle e^{i \theta_j x} \Re p_{1}, \partial_{x} R_{j}^{(1)}\right\rangle=\left\langle e^{i \theta_j x} \Im p_{1}, R_j^{(1)}\right\rangle=0\right\},
%\end{equation*}
%where $\theta_j=\frac{\omega_j c_j}{1-c_j^{2}}$.
 Furthermore, for $j \in \{1,\dots,N\}$ fixed,  we can observe that if we consider $$\vec{\Upsilon}=\left(e^{i\lambda_j(x,t)}\Psi_{\omega_j},e^{i\lambda_j(x,t)}(i\frac{\theta_j}{c_j}\Psi_{\omega_j}+c_j\partial_{x}\Psi_{\omega_j}),-\frac{2\phi_{\omega_j} \Psi_{\omega_j}}{1-c_j^2},\frac{2c_j\phi_{\omega_j} \Psi_{\omega_j}}{1-c_j^2}\right)$$
we obtain that $$\left\langle H_{j} \vec{\Upsilon}, \vec{\Upsilon}\right\rangle =-2\lambda_{j}^2\langle \Psi_{\omega_j},\Psi_{\omega_j}\rangle<0.$$
Therefore, for all \( j \in \{1, \dots, N\} \), it follows from Lemma \ref{kernel} that \( H_j \) possesses a unique simple negative eigenvalue, zero as another eigenvalue, and the remainder of its spectrum is bounded away from zero.

From Lemma \ref{kernel}, the identity \eqref{form2}, and Lemmas \ref{lemma1} and \ref{lemma2}, the following coercivity lemma can be established (see  \cite[Lemma 5]{zsp}).
 
\begin{lemma}\label{coercivity}
For all $j \in \{1,\dots,N\}$ and  $\vec{\eta}=(e^{i\lambda_j(x,t)}(\eta_{11}+i\eta_{12}),\eta_2,\eta_3,\eta_4)$,  there exists a constant $\delta>0$ such that
\begin{equation*}
\left\langle H_{j} \vec{\eta}, \vec{\eta}\right\rangle \geq \delta\|\vec{\eta}\|_{X}^{2},
\end{equation*}
where $\delta$ is independent of $\vec{\eta}$ provided that $\left(\eta_{11},\partial_{x} \phi_{\omega_j}\right)_{L^{2}(\R)}=\left(\eta_{11},i\phi_{\omega_j}\right)_{L^{2}(\R)}=\left(\eta_{11},\Psi_{\omega_j}\right)_{L^{2}(\R)}=0$.   
\end{lemma}	

\section{Modulation}\label{modul-section}
In this section, we will present a result of parameter modulation that will allow us to control the directions established by the coercivity property.
Next we set $\vec{\bar{R}}=\sum_{j=1}^{N}\vec{\bar{R}}_{j}$, where
\begin{equation}\label{modulatesoliton}
\begin{aligned}
& \bar{R}_{j}^{(1)}(x):=e^{i (\theta_j x+\gamma_j)}\Phi_{\omega_{j}}^{(1)}(x-x_{j}),\, \, \, \bar{R}_{j}^{(2)}(x):=e^{i (\theta_j x+\gamma_j)}\Phi_{\omega_{j}}^{(2)}(x-x_{j}), \\ & \bar{R}_{j}^{(3)}(x):=\Phi_{\omega_{j}}^{(3)}(x-x_{j}), \, \, \, \bar{R}_{j}^{(4)}(x):=\Phi_{\omega_{j}}^{(4)}(x-x_{j}),   
\end{aligned}
\end{equation}
where $\vec{\Phi}_{\omega_{j}}=(\Phi_{\omega_{j}}^{(1)},\Phi_{\omega_{j}}^{(2)},\Phi_{\omega_{j}}^{(3)},\Phi_{\omega_{j}}^{(4)})$ are solutions given in \eqref{solitons1} and $\theta_j=\frac{\omega_j c_j}{1-c_j^{2}}$. For $\alpha>0$, we define $$\mathbf{B}(\alpha):=\{\Vec{u}\in X:\|\Vec{u}-\bar{\vec{R}}\|_{X}< \alpha\}$$
 and
  $$
 {B}_j(\alpha):=\{{u}\in X:\|{u}-\bar{{R}_j}\|_{X}< \alpha\}.
  $$

We now establish a modulation result. Such results are by now quite standard (see, for instance, \cite{martel}, \cite{Del}). However, for the sake of completeness, we provide some details of the proofs with adaptations to our case.

\begin{lemma}\label{lemamodulation0}
There exist   $\alpha_{1}, C>0$ and $\mathcal{C}^{1}$ functions
$$
\tilde{\omega}_{j}:B_{j}(\alpha_{1})\rightarrow(0,+\infty), \quad \tilde{x}_{j}:B_{j}(\alpha_{1}) \rightarrow \mathbb{R}, \quad \tilde{\gamma}_{j}:B_{j}(\alpha_{1}) \rightarrow \mathbb{R}, \quad j=1,2,\ldots,N 
$$
such that for $\Vec{u} \in \mathbf{B}(\alpha_{1})$,  the function  
$ 
\vec{\varepsilon}=\Vec{u}-\Vec{\tilde{R}}
$ 
satisfies   the orthogonality condition
\begin{equation}\label{orthogonality0}
\left(\vec{\varepsilon}, \partial_{x}\vec{\tilde{D}}\right)_{L^2(\R)}=\left(\vec{\varepsilon}, \vec{\tilde{\Gamma}}\right)_{L^2(\R)}=\left(\vec{\varepsilon}, \vec{\tilde{\Psi}}\right)_{L^2(\R)}=0
\end{equation}
  where $ \vec{\tilde{R}}=\sum_{j=1}^N \vec{\tilde{R}}_{j}$, $\vec{\tilde{D}}=\sum_{j=1}^N(\tilde{R}_j^{(1)},0,0,0)$, $\vec{\tilde{\Gamma}}=\sum_{j=1}^N(i\tilde{R}_j^{(1)},0,0,0)$ and $\vec{\tilde{\Psi}}=\sum_{j=1}^N(\tilde{\Psi}_j,0,0,0)$, where $\tilde{\Psi}_j(x)=e^{i\tilde{\lambda}_j(x)
}\Psi_{\omega_j}(x-\tilde{x}_j)$ with $\tilde{\lambda}_j(x)=\theta_jx+\tilde{\gamma}_j$  and the modulated wave $\Vec{\tilde{R}}   =\sum_{j=1}^N\vec{\tilde{R}}_j$  is defined by  \begin{equation}\label{modulatesoliton2}
\begin{aligned}
& \tilde{R}_{j}^{(1)}(x):=e^{i (\theta_j x+\tilde{\gamma}_j)}\Phi_{\tilde{\omega}_{j}}^{(1)}(x-\tilde{x}_{j}),\, \, \, \tilde{R}_{j}^{(2)}(x):=e^{i (\theta_jx+\tilde{\gamma}_j)}\Phi_{\tilde{\omega}_{j}}^{(2)}(x-\tilde{x}_{j}), \\ & \tilde{R}_{j}^{(3)}(x):=\Phi_{\tilde{\omega}_{j}}^{(3)}(x-\tilde{x}_{j}), \, \, \, \tilde{R}_{j}^{(4)}(x):=\Phi_{\tilde{\omega}_{j}}^{(4)}(x-\tilde{x}_{j}),   
\end{aligned}
\end{equation}
with $\theta_j=\frac{\omega_j c_j}{1-c_j^{2}}$.
Moreover, if $\Vec{u} \in \mathbf{B}(\alpha)$ for $0<\alpha<\alpha_{1}$, then 
\begin{equation}\label{segundoboostrap0}
\|\vec{\varepsilon}\|_{X}+\sum_{j=1}^{N}\left\{|\tilde{\omega}_{j}-\omega_{j}|+|\tilde{x}_{j}-x_{j}|+|\tilde{\gamma}_{j}-\gamma_{j}|\right\}\leq C\alpha.
\end{equation}
\end{lemma}
\begin{proof}
    Fix  $\alpha_{1}>0$. The idea is to apply the implicit function theorem to the function  $F:B_{j}(\alpha_{1})\times  \R\times \R \times (0,\infty) \rightarrow \R^{3}$ defined by
$$
F(u, \gamma,y, \omega)\equiv  \left(\begin{array}{c}
F_{1} \\
F_{2} \\
F_{3}
\end{array}\right)=\left(\begin{array}{c}
(u_1-R(\gamma, y,\omega),iR(\gamma, y,\omega))_{L^2(\R)} \\
\left(u_1-R(\gamma, y,\omega), \partial_{x} R(\gamma, y,\omega)\right)_{L^2(\R)}\\(u_1-R(\gamma, y,\omega), \tilde{\Psi}(\gamma, y,\omega))_{L^2(\R)}
\end{array}\right),$$
$R(\gamma,y,\omega):=\sum_{j=1}^{N}e^{i\left(\theta_j x+\gamma\right)} \Phi_{\omega}^{(1)}(x-y)$ and $ \Phi_{\omega}^{(1)}$ is the ground state of \eqref{equation1}.
Note that, clearly, $F(q_0)=0$, where $q_{0}=(\sum_{j=1}^{N}\bar{R}_{j}^{(1)}, \gamma_j, y_{j},\omega_{j})$. Now, since $\Phi_{\omega}^{(1)}$ is an even and positive function, we deduce
$$
\begin{aligned}
&\partial_{\gamma}F_{1}(q_0)= \sum_{j=1}^{N}\|\Phi_{\omega_j}^{(1)}\|_{L^2(\R)}^2+O\left(e^{-3\sqrt{\omega_{\star}}c_{\star}t}\right)>0,\\
&\partial_{y}F_{2}(q_0)=-\left(\partial_{y} R(\gamma, y,\omega), \partial_{x} R(\gamma, y,\omega)\right)_{L^2(\R)}=\sum_{j=1}^{N}\left(\partial_{x}\Phi_{\omega_j}^{(1)}, \partial_{x}\Phi_{\omega_j}^{(1)} \right)_{L^2(\R)}++O\left(e^{-3\sqrt{\omega_{\star}}c_{\star}t}\right)>0,
\end{aligned}
$$
and, after some straightforward calculations,
\[
\begin{split}
\partial_{\omega}F_{3}&=-\left(\partial_{\omega} R( \gamma, y,\omega), \tilde{\Psi}_j(\gamma,y,\omega)\right)_{L^2(\R)}\\ 
&=-\sum_{j=1}^{N}\left(\partial_{\omega} \Phi_{\omega_j}^{(1)}, \Psi_{\omega_j}\right)_{L^2(\R)}+O\left(e^{-3\sqrt{\omega_{\star}}c_{\star}t}\right)\\&=\sum_{j=1}^{N}\Re\int_{\R}\partial_{\omega} \Phi_{\omega_j}^{(1)}\Psi_{\omega_j}\,dx +O\left(e^{-3\sqrt{\omega_{\star}}c_{\star}t}\right) \neq 0.
\end{split}
\]
Note that the expression \( \left(\partial_{\omega} \Phi_{\omega_j}^{(1)}, \Psi_{\omega_j}\right)_{L^2(\mathbb{R})} \) is nonzero. In fact, taking into account the explicit forms of \( \Phi_{\omega_j}^{(1)} \) and \( \Psi_{\omega_j} \), it is not difficult to verify this.

Also, by performing similar calculations, it is not difficult to see that \( \partial_{y}F_{1} = \partial_{\omega}F_{1} = \partial_{\gamma}F_{2} = 0 \). This implies that \( \det DF(q_{0}) \neq 0 \).  
Hence, for some small \( \alpha_1 > 0 \), there exist unique parameters \( (\tilde{\omega}_{j}, \tilde{\gamma}_{j}, \tilde{x}_{j}) : B_{j}(\alpha_{1}) \to (0, \infty) \times \mathbb{R} \times \mathbb{R} \) for which the functions satisfy the orthogonality conditions \eqref{orthogonality0}.  

For the second part, suppose now that \( 0 < \alpha < \alpha_{1} \) and \( \vec{u} \in \mathbf{B}(\alpha) \).
 In view of the mean value Theorem, 
$$
\begin{aligned}
 \left|\tilde{\omega}_{j}-\omega_{j}\right|&=\left|\tilde{\omega}_{j}(u_{j})-\tilde{\omega}_{j}\left(\bar{R}_{j}\right)\right|  \leq C\left\|u_{j}-\bar{R}_{j}\right\|_{H^{1}(\R)} \leq C \alpha.
\end{aligned}
$$
Similarly, we obtain the estimates for $\tilde{x}_{j},\tilde{\gamma}_{j}$. Finally, $$\|\vec{\varepsilon}\|_{X}\leq \|\Vec{u}-\vec{\bar{R}}\|_{X}+\|\vec{\tilde{R}}-\vec{\bar{R}}\|_{X}\leq \alpha+\|\vec{\tilde{R}}-\vec{\bar{R}}\|_{X},$$
But, since
$$\begin{aligned}
\|\vec{\tilde{R}}_{j}-\vec{\bar{R}}_{j}\|_{X}
&\leq \left\|e^{i\left(\tilde{\gamma}_{j}-\gamma_{j}\right) } \vec{\Phi}_{\tilde{\omega}_{j}}\left(x-\tilde{x}_{j}\right)-\vec{\Phi}_{\omega_{j}}\left(x-x_{j}\right)\right\|_{X}\\
&\leq \left\|e^{i\left(\tilde{\gamma}_{j}-\gamma_{j}\right) } \vec{\Phi}_{\tilde{\omega}_{j}}\left(x-\tilde{x}_{j}\right)-\vec{\Phi}_{\tilde{\omega}_{j}}\left(x-\tilde{x}_{j}\right)\right\|_{X}\\
&\quad +\left\|\vec{\Phi}_{\tilde{\omega}_{j}}\left(x-\tilde{x}_{j}\right)-\vec{\Phi}_{\omega_{j}}\left(x-\tilde{x}_{j}\right)\right\|_{X}+\left\|\vec{\Phi}_{{\omega}_{j}}\left(x-\tilde{x}_{j}\right)-\vec{\Phi}_{\omega_{j}}\left(x-x_{j}\right)\right\|_{X}\\
&\leq C\left|e^{i\left(\tilde{\gamma}_{j}-\gamma_{j}\right) }-1\right|\left\|\vec{\Phi}_{\tilde{\omega}_{j}}\right\|_{X}+C\left|\tilde{x}_{j}-x_{j}\right|+C\left|\tilde{\omega}_{j}-\omega_{j}\right|\\
&\leq C\left|\tilde{\gamma}_{j}-\gamma_{j}\right|+C\left|\tilde{\omega}_{j}-\omega_{j}\right|+C\left|\tilde{x}_{j}-x_{j}\right|\\
&\leq C\alpha,
\end{aligned}
$$
estimate \eqref{segundoboostrap0} immediately follows.
\end{proof}
\begin{lemma}\label{lemamodulation}
				There exists $C > 0$ such that if $T_{0}$ is sufficiently large, then there exist $\mathcal{C}^{1}$-class functions
	\[
		 \tilde{\omega}_{j}:\left[t_{0}, T^{n}\right] \rightarrow \R, \quad \tilde{\gamma}_{j}:\left[t_{0}, T^{n}\right] \rightarrow \R\quad\tilde{x}_{j}:\left[t_{0}, T^{n}\right] \rightarrow \R, \quad j=1,2, \ldots, N,
\]
				such that  
$		 
				\vec{\varepsilon}=\Vec{u}(t)-\tilde{\Vec{R}}(t) $ with $ t \in [t_{0}, T^{n}],
$
    satisfies, for $j=1,2, \ldots, N$ and for all $t \in [t_{0}, T^{n}]$, the orthogonality conditions
				 \begin{equation}\label{orthogonality}
					\left(\vec{\varepsilon}(t), \partial_{x}\vec{\tilde{D}}(t)\right)_{L^2(\R)}=\left(\vec{\varepsilon}(t), \vec{\tilde{\Gamma}}(t)\right)_{L^2(\R)}=\left(\vec{\varepsilon}(t), \vec{\tilde{\Psi}}(t)\right)_{L^2(\R)}=0, 
				\end{equation} 
      where $ \vec{\tilde{R}}=\sum_{j=1}^N \vec{\tilde{R}}_{j}$, $\vec{\tilde{D}}(t)=\sum_{j=1}^N(\tilde{R}_j^{(1)},0,0,0)$, $\vec{\tilde{\Gamma}}(t)=\sum_{j=1}^N(i\tilde{R}_j^{(1)},0,0,0)$ and $\vec{\tilde{\Psi}}(t)=\sum_{j=1}^N(\tilde{\Psi}_j,0,0,0)$, where $\tilde{\Psi}_j=e^{i\tilde{\lambda}_j(x,t)}\Psi_{\omega_j}(x-c_jt-\tilde{x}_j)$ with $\tilde{\lambda}_j(x,t)=\theta_j(x-c_jt)-\omega_jt+\tilde{\gamma}_j$  and the modulated waves are defined as in \eqref{defsolitons}
				\[\tilde{R}_{j}^{(m)}(t):= B_j(t)\Phi_{\tilde{\omega}_{j}(t)}^{(m)}\left(x-c_{j
				}t-\tilde{x}_{j}(t)\right), \qquad m=1,2,3,4,
				\]
    where $B_j(t)\vec{b}=(e^{i(\theta_jx-s_jt+\tilde{\gamma}_j)}b_1,e^{i(\theta_jx-s_jt+\tilde{\gamma}_j)}b_2,b_3,b_4)$, $s_j=\frac{\omega_j}{1-c_j^2}$ and $\theta_j=\frac{\omega_j c_j}{1-c_j^2}$.
			 Furthermore, there hold  for any $t \in \left[t_{0}, T^{n}\right]$ that
				\begin{equation}\label{segundoboostrap}
					\|\vec{\varepsilon}\|_{X}+\sum_{j=1}^{N}\paar{|\tilde{\omega}_{j}(t)-\omega_{j}|+|\tilde{x}_{j}(t)-x_{j}|+|\tilde{\gamma}_{j}(t)-\gamma_{j}|}\leq Ce^{-\omega_{\star}^{\frac32}t}.
				\end{equation}
				and
				\begin{equation}\label{modulated}
					\begin{gathered}
						\sum_{j=1}^{N}\left(\left|\partial_{t} \tilde{\omega}_{j}(t)\right|+\left|\partial_{t } \tilde{x}_{j}(t)\right|\right) 
						\leq C \left(\|\vec{\varepsilon}(t)\|_{X} +e^{-3\omega_{\star}^{\frac12}t}\right).
					\end{gathered}
				\end{equation}
			\end{lemma}
			\begin{proof}
Let us take  $t\in[t_{0},T^{n}]$. From \eqref{desiesti}, we have for fixed $j$ that
$$
\left\| u_{1}\left( t\right)-e^{i\left(\theta_j x-s_jt+\gamma_{j}\right)} \Phi_{\omega_{j}}^{(1)}\left(x-c_{j
}t-x_{j}\right)\right\|_{H^{1}(\R)}\leq e^{-\sqrt{\omega_{\star}}c_{\star}t},
$$
or, equivalently, 
$$
\left\| e^{i\left(-(\theta_j c_j-s_j)t\right)} u_{1}\left( x+c_{j}t\right)-e^{i\left(\theta_jx+\gamma_{j}\right)} \Phi_{\omega_{j}}^{(1)}\left(x-x_{j}\right)\right\|_{H^{1}(\R)}\leq e^{-\sqrt{\omega_{\star}}c_{\star}t}.
$$
This means the function $$\eta_{1}(t)=e^{i\left(-(\theta_j c_j-s_j)t\right)} u_{1}\left( \cdot+c_{j}t\right) $$
belongs to the ball $B_{1}(\alpha(t))$, where $\alpha(t)=e^{-\sqrt{\omega_{\star}}c_{\star}t}$. Similarly, we apply this process to each component, using the facts that $\alpha(t)\to 0$ as $t\to\infty$ and the mapping $t \mapsto  \vec{u}(t)$ is continuous in $H^{1}(\R)\times L^2(\R)\times L^2(\R)\times L^2(\R)$, for sufficiently large $T_0$, we are able to apply the  argument of Lemma \ref{lemamodulation0} to obtain continuous functions 
$(\tilde{\omega}_{j}, \tilde{\gamma}_{j}, \tilde{x}_{j}): [t_{0},T^{n}] \to (0,\infty) \times \mathbb{R} \times \mathbb{R}$
given by $$ \tilde{\omega}_{j}(t)=\tilde{\omega}_{j}(\vec{\eta}(t)), \qquad \tilde{x}_{j}(t)=\tilde{x}_{j}(\vec{\eta}(t)), \qquad \tilde{\gamma}_{j}(t)=\tilde{\gamma}_{i}(\vec{\eta}(t)),$$
such that from \eqref{orthogonality0}
$$
\left(\eta_1(t)-e^{i (\theta_j x+\tilde{\gamma}_j)}\Phi_{\tilde{\omega}_{j}}^{(1)}(x-\tilde{x}_{j}), ie^{i (\theta_j x+\tilde{\gamma}_j)}\Phi_{\tilde{\omega}_{j}}^{(1)}(x-\tilde{x}_{j})\right)_{L^2(\R)}=0,
$$
$$
\left(\eta_1(t)-e^{i (\theta_j x+\tilde{\gamma}_j)}\Phi_{\tilde{\omega}_{j}}^{(1)}(x-\tilde{x}_{j}), \tilde{\Psi}_j(x-\tilde{x}_{j})\right)_{L^2(\R)}=0,
$$
$$
\left(\eta_1(t)-e^{i (\theta_j x+\tilde{\gamma}_j)}\Phi_{\tilde{\omega}_{j}}^{(1)}(x-\tilde{x}_{j}),\partial_x e^{i (\theta_j x+\tilde{\gamma}_j)}\Phi_{\tilde{\omega}_{j}}^{(1)}(x-\tilde{x}_{j})\right)_{L^2(\R)}=0,
$$
which in turn are equivalent to the orthogonality conditions \eqref{orthogonality}.		
 Moreover, 
\[
				\begin{aligned}
					\left|\tilde{\omega}_{j}(t)-\omega_{j}\right|&=\left|\tilde{\omega}_{j}(\vec{\eta}(t))-\tilde{\omega}_{j}\left(\vec{\bar{R}}\right)\right| \lesssim\left\|\vec{\eta}(t)-\vec{\bar{R}}\right\|_{X}  =  \left\|\vec{u}(t)-\vec{R}(t)\right\|_{H^{1}(\R)}\lesssim e^{-\omega_{\star}^{\frac32}c_{\star}t}.
				\end{aligned}
	\]
Likewise, we obtain estimates for $\tilde{x}_{j}(t)$ and $\tilde{\gamma}_{j}(t)$.
				
				Finally, note that
\[\left\|\vec{\varepsilon}(t)\right\|_{X}\leq \left\|\vec{u}(t)-\sum_{j=1}^{N}R_{j}^{(1)}(t)\right\|_{X}+\left\|\sum_{j=1}^{N}\tilde{R}_{j}(t)-\sum_{j=1}^{N}R_{j}(t)\right\|_{X}.\]
 Since
 \begin{equation}\label{parametros}
	 \begin{split}
 \|\tilde{R}_{j}^{(1)}(t)-R_{j}^{(1)}(t)\|_{H^{1}(\R)} &=\left\| \Phi_{\tilde{\omega}_{j}}^{(1)}\left(x-c_jt-\tilde{x}_{j}(t)\right)- \Phi_{\omega_{j}}^{(1)}\left(x-c_{j }t-x_{j}\right)\right\|_{H^{1}(\R)}\\
						&\lesssim\left|\tilde{x}_{j}(t)-x_{j}\right|+ \left|\tilde{\omega}_{j}(t)-\omega_{j}\right| \lesssim e^{-\omega_{\star}^{\frac32}c_{\star}t},
					\end{split}
				\end{equation}
    and in the same way, we obtain this estimate for the other components. Thus 
				we conclude that $\|\vec{\varepsilon}(t)\|_{X}\lesssim e^{-\omega_{\star}^{\frac32}c_{\star}t}$.  From which we obtain \eqref{segundoboostrap}.
				
				%Podemos estabelecer que esses parâmetros são funções de classe $\mathcal{C}^1$, para obter mais detalhes, veja o Apêndice \ref{apen}.
				
				To conclude the proof of the lemma, it is necessary to show the estimate \eqref{modulated}. The following evolution equation is satisfied,
				\begin{equation}\label{ecuacionLN}
					\begin{split}
						&\partial_{t} \varepsilon_{1}+\varepsilon_{2} 
						=  -\partial_{t} \tilde{R}_{1}-\tilde{R}_{2},\\
						&\partial_{t} \varepsilon_{2}+\partial_{xx}\varepsilon_1-\varepsilon_1-\alpha \varepsilon_1\varepsilon_3-\alpha \varepsilon_1R_3-\alpha \varepsilon_3R_1-\alpha R_1R_3-\beta|\varepsilon_1+\tilde{R}_1|^2(\varepsilon_3+\tilde{R}_3)\\
						&=  -\partial_{t} \tilde{R}_{2}+\tilde{R}_{1}-\partial_{xx}\tilde{R}_1,\\
      &  \partial_{t} \varepsilon_3 -\partial_{x}\varepsilon_4=-\partial_{t}\tilde{R}_3 +\partial_{x}\tilde{R}_4,\\
      &\partial_{t} \varepsilon_4 -\partial_{x}\varepsilon_3-(|\varepsilon_1+\tilde{R}_1|^2)_{x}=-\partial_{t}\tilde{R}_4 +\partial_{x}\tilde{R}_3.
					\end{split}
				\end{equation}
Which is also equivalent to $$\partial_{t}(\vec{\varepsilon}+\vec{\tilde{R}})=JE^{\prime}(\vec{\varepsilon}+\vec{\tilde{R}}).$$
Then, 
\begin{equation}\label{M1}
\partial_{t} \vec{\xi}+i\sum_{j=1}^{N} (\partial_t \tilde{\gamma}_j-s_j)A\vec{\tilde{R}}_j-\sum_{j=1}^{N}(c_j+\partial_t \tilde{x}_j) e^{i\tilde{\lambda}(x,t)}\partial_{x} \vec{\tilde{R}}_j+\sum_{j=1}^{N}\partial_{t}\tilde{\omega}_j e^{i\tilde{\lambda}_j(x,t)}\partial_{\omega} \vec{\tilde{R}}_j=J E^{\prime}(\vec{\tilde{R}}+\vec{\xi}) ,
\end{equation}
where $A=\left(\begin{array}{llll}
1 & 0 & 0 & 0 \\
0 & 1 & 0 & 0 \\
0 & 0 & 0 & 0 \\
0 & 0 & 0 & 0
\end{array}\right)$ and $\tilde{\lambda}_j(x,t)=\theta_j x-s_j t+\tilde{\gamma}_j.$
Let us now analyze  $E^{\prime}(\vec{\varepsilon}+\vec{\tilde{R}}).$

Using Taylor's expansion, we have

\begin{equation*}
E^{\prime}(\vec{\tilde{R}}+\vec{\varepsilon})=E^{\prime}(\vec{\tilde{R}})+E^{\prime \prime}(\vec{\tilde{R}
}) \vec{\varepsilon}+O\left(\|\vec{\varepsilon}\|_{X}^{2}\right).
\end{equation*}
Now, notice that $$
    E^{\prime}(\vec{\tilde{R}})=\sum_{j=1}^{N}E^{\prime}(\vec{\tilde{R}}_j)+O\left(e^{-3\sqrt{\omega_{\star}}c_{\star}t}\right).$$
In fact, since 
$$\begin{aligned} E( \vec{\tilde{R}})&=\int_{\R}\bigg( \left|\sum_{j=1}^{N}\tilde{R}_j^{(1)}\right|^{2}+\left|\sum_{j=1}^{N}\tilde{R}_j^{(2)}\right|^2+\left|\sum_{j=1}^{N}\partial_x \tilde{R}_j^{(1)}\right|^{2}+\alpha\left|\sum_{j=1}^{N}\tilde{R}_j^{(1)}\right|^{2} \left(\sum_{j=1}^{N}\tilde{R}_j^{(3)}\right)\\&\quad \quad+\frac{\beta}{2}\left|\sum_{j=1}^{N}\tilde{R}_j^{(1)}\right|^{4}+\frac{\alpha}{2} \left(\sum_{j=1}^{N}\tilde{R}_j^{(3)}\right)^{2}+\frac{\alpha}{2} \left(\sum_{j=1}^{N}\tilde{R}_j^{(1)}\right)^{2}\bigg) \, d x
\end{aligned}$$
Notice that all are squared norms or multiples of these, so it suffices to work with one term, and the others are similar. Thus,
$$\begin{aligned}
    \left|\sum_{j=1}^{N}\tilde{R}_j^{(1)}\right|^{2}=\left(\sum_{j=1}^N \tilde{R}_j^{(1)}\right)\left(\sum_{k=1}^N \bar{\tilde{R}}_k^{(1)}\right)=\sum_{j=1}^N|\tilde{R}_j^{(1)}|^2+\sum_{\substack{j, k=1 \\ j \neq k}}^{N} \tilde{R}_{j}^{(1)}  \bar{\tilde{R}}_{k}^{(1)}, 
\end{aligned}$$
Then, since for $j \neq k$ (similarly as in the case of the Lemma \ref{solitons}), we have  $$\left|\tilde{R}_{j}^{(1)}  \bar{\tilde{R}}_{k}^{(1)}\right|\leq C e^{-3\sqrt{\omega_{\star}}c_{\star}t},$$
the desired result follows. 
Therefore, once $$E^{\prime}(\vec{\tilde{R}}_j)-\tilde{\omega}_j Q_{2}^{\prime}(\vec{\tilde{R}}_j)-c_j Q_{1}^{\prime}(\vec{\tilde{R}}_j)=0,$$
from \eqref{M1}, it follows that
\begin{equation}\label{M2}
\begin{aligned}
&\partial_{t} \vec{\xi}+i\sum_{j=1}^{N} (\partial_t \tilde{\gamma}_j-s_j)A\vec{\tilde{R}}_j-\sum_{j=1}^{N}(c_j+\partial_t \tilde{x}_j) e^{i\tilde{\lambda}_j(x,t)}\partial_{x} \vec{\Phi}_{\tilde{\omega}_j}+i\sum_{j=1}^{N}\tilde{\omega}_jA_2\vec{\tilde{R}}_j+\sum_{j=1}^{N}c_jA_2\partial_x\vec{\tilde{R}}_j+\sum_{j=1}^{N}c_jA_1\vec{\tilde{R}}_j\\&\quad+\sum_{j=1}^{N}\partial_{t}\tilde{\omega}_j e^{i\tilde{\lambda}_j(x,t)}\partial_{\omega} \vec{\Phi}_{\tilde{\omega}_j}\\&=JE^{\prime \prime}(\vec{\tilde{R}}) \vec{\varepsilon}+O\left(\|\vec{\varepsilon}\|_{X}^{2}\right)+O\left(e^{-3\sqrt{\omega_{\star}}c_{\star}t}\right),
\end{aligned}
\end{equation}
where $A_2=\left(\begin{array}{llll}
0 & -2 & 0 & 0 \\
2 & 0 & 0 & 0 \\
0 & 0 & 0 & 0 \\
0 & 0 & 0 & 0
\end{array}\right)$, $A_1=\left(\begin{array}{llll}
0 & 0 & 0 & 0 \\
0 & 0 & 0 & 0 \\
0 & 0 & 0 & -\alpha \\
0 & 0 & -\alpha & 0
\end{array}\right).$

We differentiate $\left(\vec{\varepsilon}, \vec{\tilde{\Psi}}\right)_{L^2(\R)}=0$ in \eqref{orthogonality} with respect to time $t$ to get

\begin{equation}\label{M3}
\sum_{j=1}^{N}\left(\partial_{t}\varepsilon_1, \tilde{\Psi}_j\right)_{L^2(\R)}=-\sum_{j=1}^{N}\left(\varepsilon_1, \partial_{t}\left(\tilde{\Psi}_j\right)\right)_{L^2(\R)}=-\sum_{j=1}^{N}\left(\varepsilon_1, \partial_{t}\tilde{\Psi}_j\right)_{L^2(\R)}=O\left((1+|\Vec{X}_j(t)|)\|\vec{\varepsilon}\|_{X}\right),
\end{equation}
where
$$
\Vec{X}_j(t)=(\partial_t \tilde{\gamma}_j-s_j, \partial_t \tilde{x}_j+c_j,\partial_t \tilde{\omega}_j).
$$
Because of the orthogonality condition, the definition of the matrices $A_1$, $A_2$  and taking the inner product of \eqref{M1} with $\vec{\tilde{\Psi}}$, we get
\begin{equation}\label{M4}
\begin{aligned}
&\left(\partial_{t} \vec{\xi},\vec{\tilde{\Psi}}\right)+i\sum_{j=1}^{N} \left((\partial_t \tilde{\gamma}_j-s_j)A\vec{\tilde{R}}_j,\vec{\tilde{\Psi}}\right)-\sum_{j=1}^{N}\left((c_j+\partial_t \tilde{x}_j) e^{i\tilde{\lambda}_j(x,t)}\partial_{x} \vec{\Phi}_{\tilde{\omega}_j},\vec{\tilde{\Psi}}\right)+\sum_{j=1}^{N}\left(\partial_{t}\tilde{\omega}_j e^{i\tilde{\lambda}_j(x,t)}\partial_{\omega} \vec{\Phi}_{\tilde{\omega}_j},\vec{\tilde{\Psi}}\right)\\&=O\left(\|\vec{\varepsilon}\|_{X}^{2}\right)+O\left(e^{-3\sqrt{\omega_{\star}}c_{\star}t}\right),
\end{aligned}
\end{equation}
then,  we have from   \eqref{M3} and \eqref{M4} that
\begin{equation}\label{M5}
\begin{aligned}
&-\sum_{j=1}^{N}(c_j+\partial_t \tilde{x}_j)\left( e^{i\tilde{\lambda}_j(x,t)}\partial_{x} \vec{\Phi}_{\tilde{\omega}_j},\vec{\tilde{\Psi}}\right)+\partial_{t}\tilde{\omega}_j\left( e^{i\tilde{\lambda}_j(x,t)}\partial_{\omega} \vec{\Phi}_{\tilde{\omega}_j},\vec{\tilde{\Psi}}\right)\\&=O\left(\|\vec{\varepsilon}\|_{X}^{2}\right)+O\left(e^{-3\sqrt{\omega_{\star}}c_{\star}t}\right)+O\left((1+|\Vec{X}_j(t)|)\|\vec{\varepsilon}\|_{X}\right).
\end{aligned}
\end{equation}
Taking the inner product of \eqref{M1} with $\vec{\tilde{\Gamma}}$ and $\partial_{x} \vec{\tilde{R}}$, respectively, by similar arguments, we get  
\begin{equation}\label{M6}
\begin{aligned}
&\sum_{j=1}^{N}(c_j+\partial_t \tilde{x}_j)\|\partial_{x}\Phi_{\tilde{\omega}_j}^{(1)}\|_{L^2(\R)}^2 +\sum_{j=1}^{N}\partial_{t}\tilde{\omega}_j\left(\partial_{\omega}\Phi_{\omega_j}^{(1)},\partial_{x}\Phi_{\omega_j}^{
(1)}\right)_{L^2(\R)}\\&=O\left(\|\vec{\varepsilon}\|_{X}^{2}\right)+O\left(e^{-3\sqrt{\omega_{\star}}c_{\star}t}\right)+O\left((1+|\Vec{X}_j(t)|)\|\vec{\varepsilon}\|_{X}\right).
\end{aligned}
\end{equation}
and 
\begin{equation}\label{M7}
\begin{aligned}
&\sum_{j=1}^{N}(\partial_t \tilde{\gamma}_j-s_j)\|\Phi_{\tilde{\omega}_j}^{(1)}\|_{L^2(\R)}^2 =O\left(\|\vec{\varepsilon}\|_{X}^{2}\right)+O\left(e^{-3\sqrt{\omega_{\star}}c_{\star}t}\right)+O\left((1+|\Vec{X}(t)|)\|\vec{\varepsilon}\|_{X}\right).
\end{aligned}
\end{equation}
Indeed, we have
\begin{equation*}
\sum_{j=1}^{N}M_j \Vec{X}_j(t)=O\left((1+|\Vec{X}(t)|)\|\vec{\varepsilon}\|_{X}\right)+O\left(\|\vec{\varepsilon}\|_{X}^{2}\right) +O\left(e^{-3\sqrt{\omega_{\star}}c_{\star}t}\right),
\end{equation*}
where
$$
M_j=\left(\begin{array}{ccc}
\|\Phi_{\tilde{\omega}_j}^{(1)}\|_{L^2(\R)}^2 & 0 & 0 \\
0 & \|\partial_x \Phi_{\tilde{\omega}_j}^{(1)}\|_{L^2(\R)}^2 & \left(\partial_{\omega}\Phi_{\tilde{\omega}_j}^{(1)},\partial_{x}\Phi_{\tilde{\omega}_j}^{(1)}\right)_{L^2(\R)} \\
0 & \left( \partial_{x} \Phi_{\tilde{\omega}_j}^{(1)},\tilde{\Psi}_j\right) &  \left( \partial_{\omega} \Phi_{\tilde{\omega}_j}^{(1)},\tilde{\Psi}_j\right)
\end{array}\right)
$$
is an invertible matrix.
Therefore,
\begin{equation*}
|\Vec{X}(t)| \leq C\|\vec{\varepsilon}\|_{X}+O\left(\|\vec{\varepsilon}\|_{X}^{2}\right)+O\left(e^{-3\sqrt{\omega_{\star}}c_{\star}t}\right) ;
\end{equation*}
which concludes the proof.
			\end{proof}
Now, we will obtain some control estimates for the solitary waves that compose $\Vec{u}$. For this purpose, we will use an argument that involves localizing certain quantities. Therefore, we assume, without loss of generality, that the propagation speeds of the solitons satisfy $c_k \neq c_m$ for all $k \neq m$. In fact, we will further assume that $c_1 < c_2 < \cdots < c_N$.
			Let $\psi: \R \rightarrow \R$ be a $C^{\infty}$ cutoff function such that $\psi(s)=0$ for $s<-1$, $\psi(s) \in [0,1]$ if $s \in [-1,1]$, and $\psi(s)=1$ for $s>1$. We define
$ m_{j}:= (c_{j-1}+c_{j})/2, $
			and introduce the following cutoff functions, for all $(x,t)\in \R \times\R$, 
	\[
			\begin{array}{ll}
				\psi_{1}(x,t):=1, & \psi_{j}(x,t):=\psi\left(\frac{1}{\sqrt{t}}\left(x -m_{j} t\right)\right), \, \,\text { for} \,\,j=2, \ldots, N.
			\end{array}
			\]
			Next, we define
	$ 
			\phi_{j}=\psi_{j}-\psi_{j+1} $  for   $j=1, \ldots, N-1, $ and $ \phi_{N}=\psi_{N }.
$ 
			By the definition of $\phi_j$, we have ${\rm supp}(\phi_{1})\subset (-\infty,\sqrt{t}+m_{2}t]$, ${\rm supp}(\phi_{N})\subset [-\sqrt{t}+m_{N}t,\infty)$ and ${\rm supp}(\phi_{j})\subset [-\sqrt{t}+m_jt,\sqrt{t}+m_jt]$ for all $j=2,\ldots,N-1$.

 With this setup, we proceed to identify the quantities conserved by the flow. In fact, for $j=1,2,\ldots,N,$ we define  
   \begin{equation}\label{eneloc}
E_j(\Vec{u})=\int_{\R}\left(|u|^{2}+\left|\rho^{2}\right|+\left|u_{x}\right|^{2}+\alpha|u|^{2} v+\frac{\beta}{2}|u|^{4}+\frac{\alpha}{2} v^{2}+\frac{\alpha}{2} n^{2}\right)\phi_j(x) \, d x .\end{equation}
\begin{equation}\label{momloc1}
Q_{1,j}(\Vec{u})=2 \Re \int_{\R} u_{x} \bar{\rho}\phi_j(x)
\,dx-\alpha \int_{\R} n v \phi_j(x) \,d x
\end{equation}
and
\begin{equation}\label{momloc2}
Q_{2,j}(\Vec{u})=2 \Im \int_{\R} \bar{u} \rho \phi_j(x) \,d x.
\end{equation}
Finally, we define the operators 
\begin{equation}\label{definitionSj}
    \mathcal{S}_{j,loc}(\Vec{u})=E_j(\Vec{u})-c_jQ_{1,j}(\Vec{u})-\tilde{\omega}_jQ_{2,j}(\Vec{u})
\end{equation}
and 
\begin{equation}\label{definitionS}
  \mathcal{S}(\Vec{u})=\sum_{j=1}^N\mathcal{S}_{j,loc}(\Vec{u}).  
\end{equation}
Following the parameter modulation strategy, we proceed to establish a lemma for the interaction of modulated solitons.
\begin{lemma}\label{solitons}
				Let $m, n\in \{1,2,3,4\}$. There exists $C>0$ such that for all $t$ sufficiently large and for all $j\neq k\in {1, \ldots, N}$,
		\[
				\begin{gathered}
					\int_{\R}\left(\left|\tilde{R}_{k}^{(m)}(t)\right|+\left|\partial_{x} \tilde{R}_{k}^{(m)}(t)\right|\right) \phi_{j}(x,t) \,d x\lesssim e^{-4 \omega_{\star}^{\frac12}c
_{\star}t}, \\
					\int_{\R}\left(\left|\tilde{R}_{k}^{(m)}(t)\right|+\left|\partial_{x} \tilde{R}_{k}^{(m)}(t)\right|\right)\left(1-\phi_{k}( x,t)\right)\,dx\lesssim e^{-4 \omega_{\star}^{\frac12}c
_{\star}t},\\
					\int_{\R}\left(|\tilde{R}_{k}^{(m)}(t)|+|\partial_{x} \tilde{R}_{k}^{(m)}(t)|\right)\left(|\tilde{R}_{j}^{(n)}(t)|+|\partial_{x} \tilde{R}_{j}^{(n)}(t)|\right)\,dx\lesssim  e^{-4 \omega_{\star}^{\frac12}c
_{\star}t}.
				\end{gathered}
	\]
			\end{lemma}
			\begin{proof}
				Let us fix $m=1$, the cases $m=2,3,4$ follow  similarly. Indeed, using Proposition \ref{decaimentoquadratico} and   \eqref{decay}, we obtain
				\begin{equation}\label{1interaccion}
					\begin{aligned}
						\int_{\R}\left|\tilde{R}_{k}^{(1)}(t)\right| \phi_{j}(x,t) \,d x &\leq \int_{\R}e^{-\frac{2}{3} \sqrt{I_k}| x-c_k t-x_{k} |} \phi_{j}(x,t) \,d x \\
						& \leq C \int_{\R} e^{-\frac{2}{3} \sqrt{I_k}\left|x-c_k t\right|} \phi_{j}(x,t) \,d x,
					\end{aligned}
				\end{equation}
                where $I_k=\frac{1-c_j^2-\omega_j^2}{(1-c_j^2)^2}.$ 				Suppose $k>j$, with $j\in \{2, \ldots, N-1\}$.
				If $k=j+1$, using the support properties of $\phi_j$ and \eqref{1interaccion}, we have
\[
				\begin{aligned}
					\int_{\R}\left|\tilde{R}_{k}^{(1)}(t)\right| \phi_{j}(x,t) \,d x &\lesssim \int_{\R}e^{-\frac{2}{3} \sqrt{I_k}| x-c_{j+1} t |} \phi_{j}(x,t) \,d x \lesssim \int_{-\sqrt{t}+m_{j}t}^{\sqrt{t}+m_{j+1}t}e^{-\frac{2}{3} \sqrt{I_k}| x-c_{j+1} t |} \, d x\\
					&\lesssim \int_{-\sqrt{t}+m_{j+1}t}^{\sqrt{t}+m_{j+1}t}e^{-\frac{2}{3} \sqrt{I_k}| x-c_{j+1} t |} \, d x+  \int_{-\infty}^{-\sqrt{t}+m_{j+1}t}e^{-\frac{2}{3} \sqrt{I_k}| x-c_{j+1} t |} \, d x\\
					&\lesssim \int_{-\sqrt{t}+m_{j+1}t}^{\sqrt{t}+m_{j+1}t}e^{-\frac{2}{3} \sqrt{I_k}| x-m_{j+1}t+m_{j+1}t-c_{j+1} t |} \, d x+  \int_{-\infty}^{-\sqrt{t}-\frac{c_{j+1}-c_{j}}{2}t}e^{-\frac{2}{3} \sqrt{I_k}| x |} \, d x\\
					&\lesssim\int_{-\sqrt{t}+m_{j+1}t}^{\sqrt{t}+m_{j+1}t}e^{-\frac{2}{3} \sqrt{I_k}\left| x-m_{j+1}t+\frac{c_{j}-c_{j+1}}{2}t \right|} \, d x+  \int_{-\infty}^{-\sqrt{t}-\frac{c_{j+1}-c_{j}}{2}t}e^{\frac{2}{3} \sqrt{I_k} x } \, d x.
				\end{aligned}
\]
				Therefore, for $k=j+1$, we get
				\begin{equation}\label{estimative0}
					\begin{aligned}					\int_{\R}\left|\tilde{R}_{k}^{(1)}(t)\right| \phi_{j}(x,t) \,d x&\lesssim e^{-\frac{1}{3}\sqrt{I_k}c_{\star}t}\int_{-\sqrt{t}+m_{j+1}t}^{\sqrt{t}+m_{j+1}t}e^{\frac{2}{3} \sqrt{I_k}\left| x-m_{j+1}t \right|} \, d x+  e^{-\frac{2}{3} \sqrt{I_k} \left(\sqrt{t}+\frac{c_{\star}}{2}t\right)}\\
						&\lesssim e^{-\frac{1}{3}\sqrt{I_k}c_{\star}t}\int_{-\sqrt{t}}^{\sqrt{t}}e^{\frac{2}{3} \sqrt{I_k}\left| x \right|} \, d x+  e^{-\frac{2}{3} \sqrt{I_k} \left(\sqrt{t}+\frac{c_{\star}}{2}t\right)}\\
						&\lesssim e^{-\frac{1}{3}\sqrt{I_k}c_{\star}t}\left(2\sqrt{t}\right)e^{\frac{2}{3} \sqrt{I_k}\sqrt{t}} +  e^{-\frac{2}{3} \sqrt{I_k} \left(\sqrt{t}+\frac{c_{\star}}{2}t\right)}\\
						&\lesssim e^{-\frac{1}{3}\sqrt{I_k}c_{\star}t}\left(\left(2\sqrt{t}\right)e^{\frac{2}{3} \sqrt{I_k}\sqrt{t}}+e^{-\frac{2}{3} \sqrt{I_k} \sqrt{t}}\right).
					\end{aligned}
				\end{equation}
				Since, for $t$ sufficiently large, 
				\[
				\left(2\sqrt{t}\right)e^{\frac{2}{3} \sqrt{I_k}\sqrt{t}}+e^{-\frac{2}{3} \sqrt{I_k} \sqrt{t}}\leq e^{\frac{1}{12}\sqrt{I_k}c_{\star}t},
				\]
				then, from \eqref{estimative0}, we have 
				\begin{equation}\label{estimative1}
					\int_{\R}\left|\tilde{R}_{k}^{(1)}(t)\right| \phi_{j}(x,t) \,d x
					\leq e^{-\frac{1}{4}\sqrt{I_k}c_{\star}t}.
				\end{equation}
				Now, if $k>j+1$, using the support properties of $\phi_{j}$ and \eqref{1interaccion}, we obtain
	\[
				\begin{aligned}					\int_{\R} \left|\tilde{R}_{k}^{(1)}(t)\right| \phi_{j}(x,t) \,d x &\lesssim \int_{\R}e^{-\frac{2}{3} \sqrt{I_k}| x-c_k t |} \phi_{j}(x,t) \,d x\\
					&= \int_{\R} e^{-\frac{1}{3} \sqrt{I_k}| x-c_{k} t |}e^{-\frac{1}{3} \sqrt{I_k}| x-c_k t |} \phi_{j}(x,t) \,d x\\
					&=  \int_{\R}e^{-\frac{1}{3} \sqrt{I_k}\left|x-m_{j} t+m_{j} t-c_k t \right|} e^{-\frac{1}{3} \sqrt{I_k}| x-c_k t |} \phi_{j}(x,t) \,d x\\
					&\lesssim \int_{\R}e^{-\frac{1}{3} \sqrt{I_k}\left|x-m_{j} t+\frac{c_{j-1}+c_{j}}{2} t-c_k t\right|} e^{-\frac{1}{3} \sqrt{I_k}| x-c_k t |} \phi_{j}(x,t) \,d x
					\\
					& =  \int_{\R}e^{-\frac{1}{3} \sqrt{I_k}\left|x-m_{j} t+\frac{c_{j-1}-c_k}{2} t+\frac{c_{j}-c_k}{2} \right|} e^{-\frac{1}{3} \sqrt{I_k}| x-c_k t |} \phi_{j}(x,t) \,d x\\
					&\lesssim e^{-\frac{1}{3} \sqrt{I_k}\frac{| c_{j}-c_k  |t}{2}}\int_{-\sqrt{t}+m_{j}t}^{\sqrt{t}+m_{j+1}t}e^{\frac{1}{3} \sqrt{I_k}\left|x-\frac{c_{j}+c_k}{2} t \right|} e^{-\frac{1}{3} \sqrt{I_k}| x-c_k t |} \,d x.
				\end{aligned}
\]
				Therefore, for $k>j+1$, 
\[
				\begin{aligned}
					\int_{\R}&\left|\tilde{R}_{k}^{(1)}(t)\right| \phi_{j}(x,t) \,d x\\
					&\lesssim e^{-\frac{1}{6} \sqrt{I_k}c_{\star}t}\int_{-\sqrt{t}+\frac{c_{j-1}-c_k}{2}t}^{\sqrt{t}+\frac{c_{j+1}-c_k}{2}t}e^{\frac{1}{3} \sqrt{I_k}\left|x \right|} e^{-\frac{1}{3} \sqrt{I_k}\left| x+\frac{c_{j}-c_k}{2}t \right|} \,d x\\
					&\lesssim e^{-\frac{1}{6} \sqrt{I_k}c_{\star}t}\left(\int_{-\sqrt{t}+\frac{c_{j-1}-c_k}{2}t}^{\sqrt{t}+\frac{c_{j+1}-c_k}{2}t}e^{\frac{1}{3} \sqrt{I_k}\left|x \right|}\,dx\right)^{1/2}\left(\int_{\R} e^{-\frac{1}{3} \sqrt{I_k}\left| x+\frac{c_{j}-c_k}{2}t \right|} \,d x\right)^{1/2}.
				\end{aligned}
	\]
				As $k>j+1$, we have $c_k>c_{j+1}$, and therefore
				\begin{equation}\label{estimative2}
	 \begin{aligned}						\int_{\R}&\left|\tilde{R}_{k}^{(1)}(t)\right| \phi_{j}(x,t) \,d x \\
						&\lesssim e^{-\frac{1}{3} \sqrt{I_k}c_{\star}t}\left(\int_{-\sqrt{t}+\frac{c_{j-1}-c_k}{2}t}^{\sqrt{t}+\frac{c_{j+1}-c_k}{2}t}e^{-\frac{1}{3} \sqrt{I_k}x }\,dx\right)^{1/2}\left(\int_{\R} e^{-\frac{1}{3} \sqrt{I_k}\left| x+\frac{c_{j}-c_k}{2}t \right|} \,d x\right)^{1/2}\\
						&\lesssim e^{-\frac{1}{3} \sqrt{I_k} c_{\star} t}.  
					\end{aligned}
				\end{equation}
				Thus, from \eqref{estimative1} and \eqref{estimative2}, we have that, for $k>j$,
				\begin{equation*}
					\begin{aligned}
						\int_{\R}\left|\tilde{R}_{k}^{(1)}(t)\right| \phi_{j}(x,t) \,d x \leq C e^{-\frac{1}{4} \sqrt{I_k} c_{\star} t} \leq Ce^{-4 \omega_{\star}^{\frac12} c_{\star}t},
					\end{aligned}    
				\end{equation*}
				whenever $t$ is sufficiently large.
				Then, from \eqref{segundoboostrap}, for $T_0$ large enough, $\frac{1-c_k^2-\tilde{\omega}_k^2}{(1-c_k^2)^2}\leq \frac{1-c_k^2-\omega_k^2}{(1-c_k^2)^2},$
				hence
				for $k>j$,
				\begin{equation}\label{2interaccion}
					\begin{aligned}
						\int_{\R}\left|\tilde{R}_{k}^{(1)}(t)\right| \phi_{j}(x,t) \,d x \leq C e^{-\frac{1}{4} \sqrt{I_k} c_{\star} t} \leq Ce^{-4 \sqrt{\omega_{\star}} c_{\star} t}.
					\end{aligned}    
				\end{equation}
				Now, if $j>k$, using the support properties of $\phi_{j}$ and \eqref{1interaccion}, we have
\[
				\begin{aligned}
					\int_{\R}\left|\tilde{R}_{k}^{(1)}(t)\right| \phi_{j}(x,t) \,d x 
					\lesssim \int_{\R} e^{-\frac{2}{3} \sqrt{I_k}\left|x-c_k t\right|} \phi_{j}(x,t) \,d x = \int_{-\sqrt{t}+m_{j} t}^{\sqrt{t}+m_{j+1} t} e^{-\frac{2}{3} \sqrt{I_k}\left|x-c_kt\right|} \,d x.
				\end{aligned}\]
				Since $j>k$, $c_{j}>c_k$, so from the above estimate, it follows that
				\begin{equation}\label{cassoj>k}
					\begin{aligned}
						&\int_{\R}\left|\tilde{R}_{k}^{(1)}(t)\right| \phi_{j}(x,t) \,d x\\&\lesssim \int_{-\sqrt{t}+m_{k+1}t}^{\infty} e^{-\frac{2}{3} \sqrt{I_k} | x-c_kt|} \,d x \\
						& \lesssim \int_{-\sqrt{t}+m_{k+1} t}^{\sqrt{t}+m_{k+1} t} e^{-\frac{2}{3} \sqrt{I_k}\left|x-c_k t\right|} \,d x+  \int_{\sqrt{t}+m_{k+1}t}^{\infty} e^{-\frac{2}{3} \sqrt{I_k} |x-c_kt|} \,d x\\
						& =  \int_{-\sqrt{t}+m_{k+1} t}^{\sqrt{t}+m_{k+1} t} e^{-\frac{3}{2} \sqrt{I_k}\left|x-m_{k+1}t+m_{k+1}t-c_k t\right|} \,d x+  \int_{\sqrt{t}+\frac{v_{k+1}-\omega_k}{2}t}^{\infty} e^{-\frac{2}{3} \sqrt{I_k} |x|} \,d x\\
						&\lesssim e^{-\frac{3}{2} \sqrt{I_k} \frac{|c_{k+1}-c_k|}{2}t}\int_{-\sqrt{t}+m_{k+1} t}^{\sqrt{t}+m_{k+1} t}e^{\frac{2}{3} \sqrt{I_k} |x-m_{k+1}t|}dx+ e^{-\frac{2}{3} \sqrt{I_k} \sqrt{t}}e^{-\frac{1}{3} \sqrt{I_k} c_{\star}t}\\
						&\lesssim e^{-\frac{1}{3} \sqrt{I_k} c_{\star}t}(2\sqrt{t})e^{\frac{2}{3} \sqrt{I_k} \sqrt{t}}+ e^{-\frac{2}{3} \sqrt{I_k} \sqrt{t}}e^{-\frac{1}{3} \sqrt{I_k} c_{\star}t}\\
						&\lesssim e^{-\frac{1}{3} \sqrt{I_k} c_{\star}t}\left((2\sqrt{t})e^{\frac{2}{3} \sqrt{I_k} \sqrt{t}}+e^{-\frac{2}{3} \sqrt{I_k} \sqrt{t}}\right).
					\end{aligned}    
				\end{equation}
				Since for $t$ sufficiently large %such that
	\[(2\sqrt{t})e^{\frac{2}{3} \sqrt{I_k} \sqrt{t}}+e^{-\frac{2}{3} \sqrt{I_k} \sqrt{t}} \leq e^{\frac{1}{12} \sqrt{I_k} c_{\star}t},
		\]
				we have that estimate \eqref{cassoj>k} becomes
				\begin{equation}\label{3interaccion}
					\begin{aligned}
						\int_{\R}\left|\tilde{R}_{k}^{(1)}(t)\right| \phi_{j}(x,t) \,d x\leq e^{-4 \sqrt{\omega_{\star}} c_{\star}t}.
					\end{aligned}    
				\end{equation}
				Therefore, for $j\neq k$ with $j\in \{2, \ldots, N-1\}$, from \eqref{2interaccion} and \eqref{3interaccion}
		\[\int_{\R}\left|\tilde{R}_{k}^{(1)}(t)\right| \phi_{j}(x,t) \,d x\lesssim e^{-4 \sqrt{\omega_{\star}} c_{\star}t}.\]
				The estimate for the term $\left|\partial_{x} \tilde{R}_{k}^{(1)}\right| \phi_{j}$ is obtained in a similar manner. Furthermore, the cases $j=1$ and $j=N$ can be obtained using a similar argument.
				
				We conclude that for sufficiently large $t$,
	\[\int_{\R}\left(\left|\tilde{R}_{k}^{(1)}(t)\right|+\left|\partial_{x} \tilde{R}_{k}^{(1)}(t)\right|\right) \phi_{j}(x,t) \,d x\lesssim e^{-4 \omega_{\star}^{\frac12}c
_{\star}t}, \, \, \text{$j\neq k$}.\]
				The second estimate follows immediately, since  we have $1-\phi_{k}=\sum_{j\neq k}\phi_{j}$ for $k \in \{1,\ldots,N\}$.
				To obtain the last estimate of the present lemma, we can proceed in the same way.
			\end{proof}

		Now we propose a result that relates the definition of the operator \(\mathcal{S}_{j,loc}\) in the sum of solitons with the operator \(\mathcal{S}_{j,loc}\) defined in the solitons.
\begin{lemma}\label{taylorS}
There exists \(T_{0}\) such that if \(t_{0} > T_{0}\), then for all \(t \in [t_{0}, T^{n}]\),
\begin{equation}\label{formSfinal}
    \mathcal{S}( \Vec{u}(t), t) = \sum_{j=1}^{N}\mathcal{S}_{j,loc}\left( \Vec{\tilde{R}}_{j}(t), t \right) + \mathcal{H}_{loc}( \vec{\varepsilon}(t), t) + O\left(e^{-3  \sqrt{\omega_{\star}} c_{\star} t}\right),
\end{equation}
where $$\begin{aligned}
    \mathcal{H}_{loc}( \vec{\varepsilon}(t), t)&=\int_{\R}|\partial_{x}\varepsilon_1|^2\,dx+\int_{\R}|\varepsilon_1|^2\,dx+\int_{\R}|\varepsilon_2|^2\,dx\\&\quad+\frac{\alpha}{2}\int_{\R}|\varepsilon_3|^2\,dx+\frac{\alpha}{2}\int_{\R}|\varepsilon_4|^2\,dx+\frac{\beta}{2}\int_{\R}|\varepsilon_1|^4\,dx\\
    &\quad +\sum_{j=1}^N \alpha \int_{\R}|\varepsilon_1|^2R_{j}^{(3)}\,dx+\sum_{j=1}^N 2\alpha \int_{\R} R_{j}^{(1)}\varepsilon_3\overline{\varepsilon}_1\,dx\\&\quad+\sum_{j=1}^N \alpha c_j \int_{\R} \varepsilon_3 \varepsilon_4 \phi_j\,dx +\sum_{j=1}^N 2 \omega_j \int_{\R} \varepsilon_2 \overline{\varepsilon}_1\phi_j\,dx \\
    &\quad +\sum_{j=1}^N 2 \beta \int_{\R} \Re(R_j^{(1)} \overline{\varepsilon}_1)R_j^{(1)} \overline{\varepsilon}_1\,dx+\sum_{j=1}^N 2 c_j \int_{\R} \varepsilon_2 \partial_{x}\overline{\varepsilon}_1\,dx.
\end{aligned}$$

\end{lemma}
\begin{proof} 
Let $t \in [t_{0},T^{n}]$ and fix $j \in\{1, \ldots, N\}$. Observe that, from the definition \eqref{definitionSj} and \eqref{definitionS}, we have 
\begin{equation}\label{primerestima}
\begin{aligned}
\mathcal{S}( \Vec{u}) & = E(\vec{u})-\sum_{j=1}^N c_j Q_{1,j}(\vec{u})-\sum_{j=1}^N \tilde{\omega}_j Q_{2,j}(\vec{u}).
\end{aligned}
\end{equation}
Initially, we will estimate each term that composes $\mathcal{S}$ by taking $\vec{u}=\vec{\tilde{R}}+\vec{\varepsilon}.$

In fact, \begin{equation}
    \begin{aligned}        E(\vec{u})&=E\left(\sum_{j=1}^{N}\vec{\tilde{R}}_j +\vec{\varepsilon}\right)\\&=\int_\rr\left|\sum_{s=1}^{N} \partial_{x} \tilde{R}_{s}^{(1)}+\varepsilon_1\right|^{2} \,dx + \int_\rr\left|\sum_{k=1}^{N}  \tilde{R}_{k}^{(2)}+\varepsilon_2\right|^{2} \,dx+ \frac{\alpha}{2}\int_\rr\left|\sum_{k=1}^{N}  \tilde{R}_{k}^{(4)}+\varepsilon_4\right|^{2} \,dx \\
    & \quad + \frac{\beta}{2} \int_\rr\left|\sum_{k=1}^{N}  \tilde{R}_{k}^{(1)}+\varepsilon_1\right|^{4} \,dx + \alpha\Re \int_\rr\left(\sum_{k=1}^{N} \tilde{R}_{k}^{(3)}+\varepsilon_3\right)\left|\sum_{s=1}^{N} \bar{\tilde{R}}_{s}^{(1)}+\varepsilon_1\right|^{2} \,dx\\
    &\quad + \frac{\alpha}{2}\int_\rr\left|\sum_{k=1}^{N}  \tilde{R}_{k}^{(3)}+\varepsilon_3\right|^{2} \,dx+ \frac{\alpha}{2}\int_\rr\left|\sum_{k=1}^{N}  \tilde{R}_{k}^{(1)}+\varepsilon_1\right|^{2} \,dx.
    \end{aligned}
\end{equation}

To estimate the terms of $E$, we will use Lemma \ref{solitons} part $3$, as follows: $$\begin{aligned}
    \int_\rr\left|\sum_{s=1}^{N} \partial_{x} \tilde{R}_{s}^{(1)}+\varepsilon_1\right|^{2} \,dx&=\int_\rr\left|\sum_{s=1}^{N} \partial_{x} \tilde{R}_{s}^{(1)}\right|^{2} \,dx+2\Re\int_{\R}\left(\sum_{s=1}^{N} \partial_{x} \tilde{R}_{s}^{(1)}\right)\partial_x \overline{\varepsilon}_1 \,dx+\int_{\R}\left|\partial_x\varepsilon_1\right|^{2} \,dx. 
\end{aligned}$$
We will expand the terms of the above integrals, so
$$
\begin{aligned}
 \left|\sum_{s=1}^{N} \partial_{x} \tilde{R}_{s}^{(1)}\right|^{2}&=\left(\partial_{x} \tilde{R}_{1}^{(1)}+\cdots+\partial_{x} \tilde{R}_{N}^{(1)}\right)\left(\partial_{x} \bar{R}_{1}^{(1)}+\cdots+\partial_{x} \bar{R}_{N}^{(1)}\right),
\end{aligned}
$$
whence,
\begin{equation*}
\left|\sum_{s=1}^{N} \partial_{x} \tilde{R}_{s}^{(1)}\right|^{2}=\sum_{s=1}^{N}\left|\partial_{x} \tilde{R}_{s}^{(1)}\right|^{2}+\sum_{\substack{s, m=1 \\ s \neq m}}^{N} \partial_{x} \tilde{R}_{s}^{(1)} \partial_{x} \bar{\tilde{R}}_{j}^{(1)}.
\end{equation*} 
Therefore, from Lemma \ref{solitons},$$
\begin{aligned}
\int_\rr\left|\sum_{s=1}^{N} \partial_{x} \tilde{R}_{s}^{(1)}+\varepsilon_1\right|^{2} \,dx
    &=\int_\rr\left|\sum_{s=1}^{N} \partial_{x} \tilde{R}_{s}^{(1)}\right|^{2} \,dx+2\Re\int_{\R}\left(\sum_{s=1}^{N} \partial_{x} \tilde{R}_{s}^{(1)}\right)\partial_x \overline{\varepsilon}_1 \,dx+\int_{\R}\left|\partial_x\varepsilon_1\right|^{2} \,dx\\
    &=\sum_{j=1}^{N} \int_{\R}\left|\partial_{x} \tilde{R}_{j}^{(1)}\right|^{2} \,dx-\sum_{j=1}^{N}\Re\int_{\R}  2\partial_{xx}^{2} \tilde{R}_{j}^{(1)} \overline{\varepsilon}_1 \,dx+\int_{\R}\left|\partial_{x}\varepsilon_1\right|^{2} \,dx\\
    &\quad+O\left(e^{-3\sqrt{\omega_{\star}}c_{\star}t}\right).
    \end{aligned}$$
Similarly to the previous computations, we have, $$\begin{aligned}
    \int_\rr\left|\sum_{s=1}^{N} \tilde{R}_{s}^{(2)}+\varepsilon_2\right|^{2} \,dx&=\int_\rr\left|\sum_{s=1}^{N}  \tilde{R}_{s}^{(2)}\right|^{2} \,dx+2\Re\int_{\R}\left(\sum_{s=1}^{N} \tilde{R}_{s}^{(2)}\right) \overline{\varepsilon}_2 \,dx+\int_{\R}\left|\varepsilon_2\right|^{2} \,dx \\
    &=\sum_{j=1}^{N} \int_{\R}\left| \tilde{R}_{j}^{(2)}\right|^{2} \,dx+\sum_{j=1}^{N}\Re\int_{\R}  2 \tilde{R}_{j}^{(2)} \overline{\varepsilon}_2 \,dx+\int_{\R}\left|\varepsilon_2\right|^{2} \,dx+O\left(e^{-3\sqrt{\omega_{\star}}c_{\star}t}\right),
\end{aligned}$$
$$\begin{aligned}
    \int_\rr\left|\sum_{s=1}^{N} \tilde{R}_{s}^{(1)}+\varepsilon_1\right|^{2} \,dx&=\int_\rr\left|\sum_{s=1}^{N}  \tilde{R}_{s}^{(1)}\right|^{2} \,dx+2\Re\int_{\R}\left(\sum_{s=1}^{N} \tilde{R}_{s}^{(1)}\right) \overline{\varepsilon}_1 \,dx+\int_{\R}\left|\varepsilon_1\right|^{2} \,dx \\
    &=\sum_{j=1}^{N} \int_{\R}\left| \tilde{R}_{j}^{(1)}\right|^{2} \,dx+\sum_{j=1}^{N}\Re\int_{\R}  2 \tilde{R}_{j}^{(1)} \overline{\varepsilon}_1 \,dx+\int_{\R}\left|\varepsilon_1\right|^{2} \,dx+O\left(e^{-3\sqrt{\omega_{\star}}c_{\star}t}\right), 
\end{aligned}$$
$$\begin{aligned}
    \frac{\alpha}{2}\int_\rr\left(\sum_{s=1}^{N} \tilde{R}_{s}^{(3)}+\varepsilon_3\right)^{2} \,dx&=\frac{\alpha}{2}\int_\rr\left(\sum_{s=1}^{N}  \tilde{R}_{s}^{(3)}\right)^{2} \,dx+2\frac{\alpha}{2}\Re\int_{\R}\left(\sum_{s=1}^{N} \tilde{R}_{s}^{(3)}\right) \overline{\varepsilon}_3 \,dx+\frac{\alpha}{2}\int_{\R}\left(\varepsilon_3\right)^{2} \,dx \\
    &=\frac{\alpha}{2}\sum_{j=1}^{N} \int_{\R}\left( \tilde{R}_{j}^{(3)}\right)^{2} \,dx+\sum_{j=1}^{N}\Re\int_{\R}  \alpha \tilde{R}_{j}^{(3)} \overline{\varepsilon}_3 \,dx+\frac{\alpha}{2}\int_{\R}\left(\varepsilon_1\right)^{2} \,dx\\
&\quad+O\left(e^{-3\sqrt{\omega_{\star}}c_{\star}t}\right)
\end{aligned}$$
and $$\begin{aligned}
    &\frac{\alpha}{2}\int_\rr\left(\sum_{s=1}^{N} \tilde{R}_{s}^{(4)}+\varepsilon_4\right)^{2} \,dx\\&=\frac{\alpha}{2}\int_\rr\left(\sum_{s=1}^{N}  \tilde{R}_{s}^{(4)}\right)^{2} \,dx+2\frac{\alpha}{2}\Re\int_{\R}\left(\sum_{s=1}^{N} \tilde{R}_{s}^{(4)}\right) \overline{\varepsilon}_4 \,dx+\frac{\alpha}{2}\int_{\R}\left(\varepsilon_4\right)^{2} \,dx \\
    &=\frac{\alpha}{2}\sum_{j=1}^{N} \int_{\R}\left( \tilde{R}_{j}^{(4)}\right)^{2} \,dx+\sum_{j=1}^{N}\Re\int_{\R}  \alpha \tilde{R}_{j}^{(4)} \overline{\varepsilon}_4\,dx+\frac{\alpha}{2}\int_{\R}\left(\varepsilon_4\right)^{2} \,dx+O\left(e^{-3\sqrt{\omega_{\star}}c_{\star}t}\right).
\end{aligned}$$
Now, we have 
$$\begin{aligned}
    &\alpha\int_\rr\left(\sum_{s=1}^{N} \tilde{R}_{s}^{(3)}+\varepsilon_3\right)\left|\sum_{s=1}^{N} \tilde{R}_{s}^{(1)}+\varepsilon_1\right|^2\,dx\\&=\alpha\int_\rr\left(\sum_{s=1}^{N}  \tilde{R}_{s}^{(3)}\right) \left|\sum_{s=1}^{N} \tilde{R}_{s}^{(1)}\right|^2\,dx+\alpha\int_\rr\left(\sum_{s=1}^{N}  \tilde{R}_{s}^{(3)}\right) \left|\varepsilon_1\right|^2\,dx\\&\quad+\alpha\int_\rr\left|\sum_{s=1}^{N}  \tilde{R}_{s}^{(1)}\right|^2 \varepsilon_3\,dx+\alpha\int_\rr\varepsilon_3 |\varepsilon_1|^2\,dx\\
    &\quad+\alpha \int_{\R}\left( \sum_{j=1}^{N}\tilde{R}_{j}^{(3)}\right) 2\Re \left(\sum_{k=1}^{N}\tilde{R}_{k}^{(1)} \overline{\varepsilon}_1\right)\,dx+\alpha \int_{\R} 2\Re \left(\sum_{k=1}^{N}\tilde{R}_{k}^{(1)} \overline{\varepsilon}_1\right)\varepsilon_3\,dx.
\end{aligned}$$
Furthermore, note that 
\begin{equation}\label{somadouble}
\begin{aligned}
&\left(\sum_{k=1}^{N} \tilde{R}_{k}^{(3)}\right)\left|\sum_{s=1}^{N} \bar{\tilde{R}}_{s}^{(1)}\right|^{2}\\ &=\tilde{R}_{j}^{(3)}\left(\bar{\tilde{R}}_{j}^{(1)}\right)^{2}+2\sum_{\substack{s=1\\ s\neq j}}^{N} \bar{\tilde{R}}_{s}^{(1)}\bar{\tilde{R}}_{j}^{(1)}\tilde{R}_{j}^{(3)}+ \sum_{\substack{k=1\\k\neq j}}^{N} \tilde{R}_{k}^{(3)}\,(\bar{\tilde{R}}_{j}^{(1)})^{2} + \left(\sum_{\substack{s=1\\s\neq j}}^{N} \bar{\tilde{R}}_{s}^{(1)}\right)^{2}\tilde{R}_{j}^{(3)}\\
&\quad +2\left(\sum_{\substack{k=1\\k\neq j}}^{N} \tilde{R}_{k}^{(3)}\right)\left(\sum_{\substack{s=1\\s\neq j}}^{N} \bar{\tilde{R}}_{s}^{(1)}\right)\bar{\tilde{R}}_{j}^{(1)}+\left(\sum_{\substack{k=1\\k\neq j}}^{N} \tilde{R}_{k}^{(3)}\right)\left(\sum_{\substack{s=1\\s\neq j}}^{N} \bar{\tilde{R}}_{s}^{(1)}\right)^{2}.
\end{aligned}
\end{equation}
Thus, using Lemma \ref{solitons} along with the above estimates, we obtain
$$\begin{aligned}
    &\alpha\int_\rr\left(\sum_{s=1}^{N} \tilde{R}_{s}^{(3)}+\varepsilon_3\right)\left|\sum_{s=1}^{N} \tilde{R}_{s}^{(1)}+\varepsilon_1\right|^2\,dx\\&=\sum_{j=1}^{N}\int_\rr\alpha\left(  \tilde{R}_{j}^{(3)}\right) \left| \tilde{R}_{j}^{(1)}\right|^2\,dx+\sum_{j=1}^{N}\int_\rr\alpha\left(  \tilde{R}_{j}^{(3)}\right) \left|\varepsilon_1\right|^2\,dx+\sum_{j=1}^{N}\int_\rr \alpha\left| \tilde{R}_{j}^{(1)}\right|^2 \varepsilon_3\,dx\\&\quad+\alpha\int_\rr\varepsilon_3 |\varepsilon_1|^2\,dx+ \sum_{j=1}^{N}\int_{\R}\left( \tilde{R}_{j}^{(3)}\right) 2\alpha\Re \left(\tilde{R}_{j}^{(1)} \overline{\varepsilon}_1\right)\,dx+\sum_{j=1}^{N}\int_{\R} 2\alpha \Re \left(\tilde{R}_{j}^{(1)} \overline{\varepsilon}_1\right)\varepsilon_3\,dx
\\&\quad +O\left(e^{-3\sqrt{\omega_{\star}}c_{\star}t}\right).
\end{aligned}$$
Finally, it only remains to analyze $$\begin{aligned}
   & \frac{\beta}{2}\int_\rr\left|\sum_{s=1}^{N} \tilde{R}_{s}^{(1)}+\varepsilon_1\right|^4\,dx=\frac{\beta}{2}\int_\rr\left(\left|\sum_{s=1}^{N} \tilde{R}_{s}^{(1)}\right|^2+2\Re\left(\sum_{s=1}^{N} \tilde{R}_{s}^{(1)} \overline{\varepsilon}_1\right)+|\varepsilon_1|^2\right)^2\,dx\\
    &=\frac{\beta}{2}\int_\rr\left(\left|\sum_{s=1}^{N} \tilde{R}_{s}^{(1)}\right|^2+2\Re\left(\sum_{s=1}^{N} \tilde{R}_{s}^{(1)} \overline{\varepsilon}_1\right)\right)^2\,dx+\beta\int_\rr\left(\left|\sum_{s=1}^{N} \tilde{R}_{s}^{(1)}\right|^2+2\Re\left(\sum_{s=1}^{N} \tilde{R}_{s}^{(1)} \overline{\varepsilon}_1\right)\right)|\varepsilon_1|^2\,dx\\
    &\quad+\frac{\beta}{2}\int_\rr|\varepsilon_1|^4\,dx\\
    &=\frac{\beta}{2}\int_\rr\left|\sum_{s=1}^{N} \tilde{R}_{s}^{(1)}\right|^4\,dx+2\beta\int_\rr\left|\sum_{s=1}^{N} \tilde{R}_{s}^{(1)}\right|^2\Re\left(\sum_{s=1}^{N} \tilde{R}_{s}^{(1)} \overline{\varepsilon}_1 \right)\,dx+2\beta\int_\rr\Re\left(\sum_{s=1}^{N} \tilde{R}_{s}^{(1)} \overline{\varepsilon}_1 \right)^2\,dx\\
    &\quad+\beta\int_\rr\left|\sum_{s=1}^{N} \tilde{R}_{s}^{(1)}\right|^2|\varepsilon_1|^2\,dx +2\beta\int_\rr\Re\left(\sum_{s=1}^{N} \tilde{R}_{s}^{(1)} \overline{\varepsilon}_1 \right)|\varepsilon_1|^2\,dx+\frac{\beta}{2}\int_\rr|\varepsilon_1|^4\,dx.
\end{aligned}$$
Following the same steps as in the previous estimates and applying Lemma \ref{solitons}, we obtain
$$\begin{aligned}
   & \frac{\beta}{2}\int_\rr\left|\sum_{s=1}^{N} \tilde{R}_{s}^{(1)}+\varepsilon_1\right|^4\,dx\\&=\sum_{j=1}^{N}\int_\rr\frac{\beta}{2}\left| \tilde{R}_{j}^{(1)}\right|^4\,dx+\sum_{s=1}^{N}\int_\rr2\beta\Re\left( \tilde{R}_{j}^{(1)} \overline{\varepsilon}_1 \right)^2\,dx++\frac{\beta}{2}\int_\rr|\varepsilon_1|^4\,dx+O\left(e^{-3\sqrt{\omega_{\star}}c_{\star}t}\right).
\end{aligned}$$
Combining the previous estimates, we have
\begin{equation}
    \begin{aligned}        &E(\vec{u})\\&=\sum_{j=1}^{N} \int_{\R}\left|\partial_{x} \tilde{R}_{j}^{(1)}\right|^{2} \,dx-\sum_{j=1}^{N}\Re\int_{\R}  2\partial_{xx}^{2} \tilde{R}_{j}^{(1)} \overline{\varepsilon}_1 \,dx+\int_{\R}\left|\partial_{x}\varepsilon_1\right|^{2} \,dx+\sum_{j=1}^{N} \int_{\R}\left| \tilde{R}_{j}^{(1)}\right|^{2} \,dx +\sum_{j=1}^{N}\Re\int_{\R}  2 \tilde{R}_{j}^{(1)} \overline{\varepsilon}_1 \,dx\\&\quad+\int_{\R}\left|\varepsilon_1\right|^{2} \,dx  + \sum_{j=1}^{N} \int_{\R}\left| \tilde{R}_{j}^{(2)}\right|^{2} \,dx+\sum_{j=1}^{N}\Re\int_{\R}  2 \tilde{R}_{j}^{(2)} \overline{\varepsilon}_2 \,dx+\int_{\R}\left|\varepsilon_2\right|^{2} \,dx +\frac{\alpha}{2}\sum_{j=1}^{N} \int_{\R}\left( \tilde{R}_{j}^{(3)}\right)^{2} \,dx\\
    &\quad+\sum_{j=1}^{N}\Re\int_{\R}  \alpha \tilde{R}_{j}^{(3)} \overline{\varepsilon}_3 \,dx+\frac{\alpha}{2}\int_{\R}\left(\varepsilon_1\right)^{2} \,dx +\frac{\alpha}{2}\sum_{j=1}^{N} \int_{\R}\left( \tilde{R}_{j}^{(4)}\right)^{2} \,dx+\sum_{j=1}^{N}\Re\int_{\R}  \alpha \tilde{R}_{j}^{(4)} \overline{\varepsilon}_4 \,dx+\frac{\alpha}{2}\int_{\R}\left(\varepsilon_4\right)^{2} \,dx\\
    &\quad+\sum_{j=1}^{N}\int_\rr\alpha\left(  \tilde{R}_{j}^{(3)}\right) \left| \tilde{R}_{j}^{(1)}\right|^2\,dx+\sum_{j=1}^{N}\int_\rr\alpha\left(  \tilde{R}_{j}^{(3)}\right) \left|\varepsilon_1\right|^2\,dx+\sum_{j=1}^{N}\int_\rr \alpha\left| \tilde{R}_{j}^{(1)}\right|^2 \varepsilon_3\,dx\\&\quad+ \sum_{j=1}^{N}\int_{\R}\left( \tilde{R}_{j}^{(3)}\right) 2\alpha\Re \left(\tilde{R}_{j}^{(1)} \overline{\varepsilon}_1\right)\,dx+\sum_{j=1}^{N}\int_{\R} 2\alpha \Re \left(\tilde{R}_{j}^{(1)} \overline{\varepsilon}_1\right)\varepsilon_3\,dx+\sum_{s=1}^{N}\int_\rr2\beta\Re\left( \tilde{R}_{j}^{(1)} \overline{\varepsilon}_1 \right)^2\,dx\\&\quad+\frac{\beta}{2}\int_\rr|\varepsilon_1|^4\,dx+\sum_{j=1}^{N}\int_\rr\frac{\beta}{2}\left| \tilde{R}_{j}^{(1)}\right|^4\,dx+O\left(e^{-3\sqrt{\omega_{\star}}c_{\star}t}\right).
    \end{aligned}
\end{equation}
That is, \begin{equation}\label{1accou}
    \begin{aligned}        E(\vec{u})&=\sum_{j=1}^{N}E(\vec{\tilde{R}}_j) -\sum_{j=1}^{N}\Re\int_{\R}  2\partial_{xx}^{2} \tilde{R}_{j}^{(1)} \overline{\varepsilon}_1 \,dx+\int_{\R}\left|\partial_{x}\varepsilon_1\right|^{2} \,dx +\sum_{j=1}^{N}\Re\int_{\R}  2 \tilde{R}_{j}^{(1)} \overline{\varepsilon}_1 \,dx\\&\quad+\int_{\R}\left|\varepsilon_1\right|^{2} \,dx  +\sum_{j=1}^{N}\Re\int_{\R}  2 \tilde{R}_{j}^{(2)} \overline{\varepsilon}_2 \,dx+\int_{\R}\left|\varepsilon_2\right|^{2} \,dx+\sum_{j=1}^{N}\Re\int_{\R}  \alpha \tilde{R}_{j}^{(3)} \overline{\varepsilon}_3 \,dx\\
    &\quad+\frac{\alpha}{2}\int_{\R}\left(\varepsilon_1\right)^{2} \,dx+\sum_{j=1}^{N}\Re\int_{\R}  \alpha \tilde{R}_{j}^{(4)} \overline{\varepsilon}_4 \,dx+\frac{\alpha}{2}\int_{\R}\left(\varepsilon_4\right)^{2} \,dx+\sum_{j=1}^{N}\int_\rr\alpha\left(  \tilde{R}_{j}^{(3)}\right) \left|\varepsilon_1\right|^2\,dx\\
    &\quad+\sum_{j=1}^{N}\int_\rr \alpha\left| \tilde{R}_{j}^{(1)}\right|^2 \varepsilon_3\,dx+ \sum_{j=1}^{N}\int_{\R}\left( \tilde{R}_{j}^{(3)}\right) 2\alpha\Re \left(\tilde{R}_{j}^{(1)} \overline{\varepsilon}_1\right)\,dx+\sum_{j=1}^{N}\int_{\R} 2\alpha \Re \left(\tilde{R}_{j}^{(1)} \overline{\varepsilon}_1\right)\varepsilon_3\,dx\\&\quad+\sum_{s=1}^{N}\int_\rr2\beta\Re\left( \tilde{R}_{j}^{(1)} \overline{\varepsilon}_1 \right)^2\,dx+\frac{\beta}{2}\int_\rr|\varepsilon_1|^4\,dx+O\left(e^{-3\sqrt{\omega_{\star}}c_{\star}t}\right).
    \end{aligned}
\end{equation}
We now proceed to estimate the other components of $S$ in \eqref{primerestima} :
$$\begin{aligned}
    -\sum_{j=1}^N c_j Q_{1,j}(\vec{u})&=-\sum_{j=1}^N c_j \bigg[2\Re \int_{\R}\left(\sum_{k=1}^{N} \partial_{x} \tilde{R}_{k}^{(1)}+\partial_{x} \varepsilon_1\right)\overline{\left(\sum_{m=1}^{N}  \tilde{R}_{m}^{(2)}+ \varepsilon_2\right)}\phi_{j}\,dx\\&\quad \quad \quad \quad \quad \quad-\alpha \int_{\R}\left(\sum_{k=1}^{N}  \tilde{R}_{k}^{(3)}+\varepsilon_3\right)\left(\sum_{m=1}^{N}  \tilde{R}_{m}^{(4)}+ \varepsilon_4\right)\phi_{j}\,dx\bigg]\\
    &=-\sum_{j=1}^N c_j \bigg[2\Re \int_{\R}\left(\sum_{k=1}^{N} \partial_{x} \tilde{R}_{k}^{(1)}\right)\left(\sum_{m=1}^{N}  \overline{\tilde{R}_{m}^{(2)}}\right)\phi_{j}\,dx+2\Re \int_{\R}\left(\sum_{k=1}^{N} \partial_{x} \tilde{R}_{k}^{(1)}\right)\overline{\varepsilon}_2\phi_{j}\,dx\\
    &\quad \quad \quad \quad \quad \quad+2\Re \int_{\R}\left(\sum_{m=1}^{N}  \overline{\tilde{R}_{m}^{(2)}}\right)\partial_{x}\varepsilon_1\phi_{j}\,dx+2\Re \int_{\R}\partial_{x}\varepsilon_1 \overline{\varepsilon}_2\phi_{j}\,dx\\&\quad \quad \quad \quad \quad \quad-\alpha \int_{\R}\left(\sum_{k=1}^{N}  \tilde{R}_{k}^{(3)}\right)\left(\sum_{m=1}^{N}  \tilde{R}_{m}^{(4)}\right)\phi_{j}\,dx-\alpha \int_{\R}\left(\sum_{k=1}^{N}  \tilde{R}_{k}^{(3)}\right)\varepsilon_4\phi_{j}\,dx\\&\quad \quad \quad \quad \quad \quad -\alpha \int_{\R}\left(\sum_{k=1}^{N}  \tilde{R}_{k}^{(4)}\right)\varepsilon_3\phi_{j}\,dx-\alpha \int_{\R}\varepsilon_3\varepsilon_4\phi_{j}\,dx\bigg].
\end{aligned}$$
Now, using Lemma \ref{solitons} and similar calculations to those above, we arrive at $$\begin{aligned}
    -\sum_{j=1}^N c_j Q_{1,j}(\vec{u})
    &=-\sum_{j=1}^N c_j \bigg[\Re \int_{\R}2 \sum_{k=1}^{N}\partial_{x} \tilde{R}_{k}^{(1)}  \overline{\tilde{R}_{k}^{(2)}}\phi_{j}\,dx+ \Re\int_{\R} 2 \sum_{k=1}^{N} \partial_{x} \tilde{R}_{k}^{(1)}\overline{\varepsilon}_2\phi_{j}\,dx\\
    &\quad \quad \quad \quad \quad \quad+2\Re \int_{\R}\sum_{m=1}^{N}  \overline{\tilde{R}_{m}^{(2)}}\partial_{x}\varepsilon_1\phi_{j}\,dx+2\Re \int_{\R}\partial_{x}\varepsilon_1 \overline{\varepsilon}_2\phi_{j}\,dx\\&\quad \quad \quad \quad \quad \quad-\alpha \int_{\R}\sum_{k=1}^{N}  \tilde{R}_{k}^{(3)}  \tilde{R}_{k}^{(4)}\phi_{j}\,dx-\alpha \int_{\R}\sum_{k=1}^{N}  \tilde{R}_{k}^{(3)}\varepsilon_4\phi_{j}\,dx\\&\quad \quad \quad \quad \quad \quad -\alpha \int_{\R}\sum_{k=1}^{N}  \tilde{R}_{k}^{(4)}\varepsilon_3\phi_{j}\,dx-\alpha \int_{\R}\varepsilon_3\varepsilon_4\phi_{j}\,dx\\
    &\quad  \quad \quad \quad \quad \quad + O\left(e^{-3  \sqrt{\omega_{\star}} c_{\star} t}\right)\bigg].
\end{aligned}$$
From the definition of $\phi_j$, we have $$\begin{aligned}
    -\sum_{j=1}^N c_j Q_{1,j}(\vec{u})
    &=-\sum_{j=1}^N c_j \bigg[\Re \int_{\R}2 \partial_{x} \tilde{R}_{j}^{(1)}  \overline{\tilde{R}_{j}^{(2)}}\phi_{j}\,dx+ \Re\int_{\R} 2  \partial_{x} \tilde{R}_{j}^{(1)}\overline{\varepsilon}_2\phi_{j}\,dx\\
    &\quad \quad \quad \quad \quad \quad+2\Re \int_{\R}  \overline{\tilde{R}_{j}^{(2)}}\partial_{x}\varepsilon_1\phi_{j}\,dx+2\Re \int_{\R}\partial_{x}\varepsilon_1 \overline{\varepsilon}_2\phi_{j}\,dx\\&\quad \quad \quad \quad \quad \quad-\alpha \int_{\R}  \tilde{R}_{j}^{(3)}  \tilde{R}_{j}^{(4)}\phi_{j}\,dx-\alpha \int_{\R}  \tilde{R}_{j}^{(3)}\varepsilon_4\phi_{j}\,dx\\&\quad \quad \quad \quad \quad \quad -\alpha \int_{\R}  \tilde{R}_{j}^{(4)}\varepsilon_3\phi_{j}\,dx-\alpha \int_{\R}\varepsilon_3\varepsilon_4\phi_{j}\,dx\\
    &\quad  \quad \quad \quad \quad \quad + O\left(e^{-3  \sqrt{\omega_{\star}} c_{\star} t}\right)\bigg].
\end{aligned}$$
Therefore, \begin{equation}\label{accou2}
    \begin{aligned}
    -\sum_{j=1}^N c_j Q_{1,j}(\vec{u})
    &= -\sum_{j=1}^N c_j Q_{1,j}(\vec{\tilde{R}}_j)-\sum_{j=1}^N c_j \bigg[ \Re\int_{\R} 2  \partial_{x} \tilde{R}_{j}^{(1)}\overline{\varepsilon}_2\phi_{j}\,dx+2\Re \int_{\R}  \overline{\tilde{R}_{j}^{(2)}}\partial_{x}\varepsilon_1\phi_{j}\,dx\\
    &\quad \quad \quad \quad \quad \quad-\alpha \int_{\R}  \tilde{R}_{j}^{(3)}\varepsilon_4\phi_{j}\,dx-\alpha \int_{\R}  \tilde{R}_{j}^{(4)}\varepsilon_3\phi_{j}\,dx-\alpha \int_{\R}\varepsilon_3\varepsilon_4\phi_{j}\,dx\\
    &\quad \quad \quad \quad \quad \quad + O\left(e^{-3  \sqrt{\omega_{\star}} c_{\star} t}\right)\bigg].
\end{aligned}
\end{equation}
Finally, by performing a process very similar to the previous one, we can obtain that 
\begin{equation}\label{accou3}
    \begin{aligned}
    -\sum_{j=1}^N \tilde{\omega}_j Q_{2,j}(\vec{u})
    &= -\sum_{j=1}^N \tilde{\omega}_j Q_{2,j}(\vec{\tilde{R}}_j)-\sum_{j=1}^N \tilde{\omega}_j \bigg[ \Im\int_{\R} 2   \overline{\tilde{R}_{j}^{(1)}}\varepsilon_2\phi_{j}\,dx+\Im \int_{\R} 2 \tilde{R}_{j}^{(2)}\overline{\varepsilon}_1\phi_{j}\,dx+\int_{\R}2\overline{\varepsilon}_1\varepsilon_2\phi_{j}\,dx\\
    &\quad \quad \quad \quad \quad \quad + O\left(e^{-3  \sqrt{\omega_{\star}} c_{\star} t}\right)\bigg].
\end{aligned}
\end{equation}
Using the previous estimates, \eqref{1accou}-\eqref{accou3}, along with the fact that $\vec{\tilde{R}}_j$ satisfies \eqref{pointcrit} for each $j$, it follows that for all $j \in \{1, \ldots, N\},$
$$
\mathcal{S}( \Vec{u}(t), t) = \sum_{j=1}^{N}\left(E(\vec{\tilde{R}}_j)-c_jQ_{1,J}(\vec{\tilde{R}}_j)-\tilde{\omega}_jQ_{2,j}(\vec{\tilde{R}}_j)\right) + \mathcal{H}_{loc}( \vec{\varepsilon}(t), t) + O\left(e^{-3  \sqrt{\omega_{\star}} c_{\star} t}\right),
$$
where $$\begin{aligned}
    \mathcal{H}_{loc}( \vec{\varepsilon}(t), t)&=\int_{\R}|\partial_{x}\varepsilon_1|^2\,dx+\int_{\R}|\varepsilon_1|^2\,dx+\int_{\R}|\varepsilon_2|^2\,dx\\&\quad+\frac{\alpha}{2}\int_{\R}|\varepsilon_3|^2\,dx+\frac{\alpha}{2}\int_{\R}|\varepsilon_4|^2\,dx+\frac{\beta}{2}\int_{\R}|\varepsilon_1|^4\,dx\\
    &\quad +\sum_{j=1}^N \alpha \int_{\R}|\varepsilon_1|^2R_{j}^{(3)}\,dx+\sum_{j=1}^N 2\alpha \int_{\R} R_{j}^{(1)}\varepsilon_3\overline{\varepsilon}_1\,dx\\&\quad+\sum_{j=1}^N \alpha c_j \int_{\R} \varepsilon_3 \varepsilon_4 \phi_j\,dx +\sum_{j=1}^N 2 \tilde{\omega}_j \int_{\R} \varepsilon_2 \overline{\varepsilon}_1\phi_j\,dx \\
    &\quad +\sum_{j=1}^N 2 \beta \int_{\R} \Re(R_j^{(1)} \overline{\varepsilon}_1)R_j^{(1)} \overline{\varepsilon}_1\,dx+\sum_{j=1}^N 2 c_j \int_{\R} \varepsilon_2 \partial_{x}\overline{\varepsilon}_1\,dx.
\end{aligned}$$
Finally, note that $$E_j(\vec{\tilde{R}}_j)=E(\vec{\tilde{R}}_j)++ O\left(e^{-3  \sqrt{\omega_{\star}} c_{\star} t}\right),$$
which is due to the definition of $\phi_j$ for each $j$. For example, we can take a component of the sums in $E_j$ and note that: $$\int_\rr\left|\sum_{s=1}^{N} \partial_{x} \tilde{R}_{s}^{(1)}\right|^{2} \phi_{j}\,dx=\int_\rr\left|\partial_{x} \tilde{R}_{j}^{(1)}\right|^{2} \,dx-\int_\rr\left|\partial_{x} \tilde{R}_{j}^{(1)}\right|^{2}\sum_{\substack{k=1 \\ k \neq j}}^{N}\phi_{k} \,dx,$$
Thus, from Lemma \ref{solitons}, the assertion follows.
\end{proof}
We will now prove that the operator $\mathcal{H}_{\operatorname{loc}}$ still satisfies the coercivity property.
\begin{proposition}\label{coercivity2}
				There exists $K>0$ such that
				\[\begin{aligned}
					\langle \mathcal{H}_{loc} \vec{\varepsilon},\vec{\varepsilon}\rangle &\geq K\|\vec{\varepsilon}\|_{X}^2
				\end{aligned}\]
			\end{proposition}
			\begin{proof}
				First, we give a localized version of Proposition \ref{coercivity}. Let $\Phi: \mathbb{R} \rightarrow \mathbb{R}$ be a $C^{2}$-function such that $\Phi(x)=\Phi(-x), \Phi^{\prime} \leqslant 0$ on $\mathbb{R}^{+}$with
				\[
				\begin{array}{ll}
					\Phi(x)=1 & \text { on }[0,1] ; \quad \Phi(x)=e^{-x} \quad \text { on }[2,+\infty) \\
					& e^{-x} \leqslant \Phi(x) \leqslant 3 e^{-x} \quad \text { on } \mathbb{R} .
				\end{array}
				\]
				
				Let $B>0$, and let $\Phi_{B}(x)=\Phi(x / B)$. Set
				\[
				\begin{aligned}
					\frac{1}{2}\langle \mathcal{H}_{\Phi_{B}}\vec{\varepsilon}, \vec{\varepsilon}\rangle&=\int_{\R}\Phi_{B}\left(\cdot-x_{0}\right)|\partial_{x}\varepsilon_1|^2\,dx+\int_{\R}\Phi_{B}\left(\cdot-x_{0}\right)|\varepsilon_1|^2\,dx+\int_{\R}\Phi_{B}\left(\cdot-x_{0}\right)|\varepsilon_2|^2\,dx\\&\quad+\frac{\alpha}{2}\int_{\R}\Phi_{B}\left(\cdot-x_{0}\right)|\varepsilon_3|^2\,dx+\frac{\alpha}{2}\int_{\R}\Phi_{B}\left(\cdot-x_{0}\right)|\varepsilon_4|^2\,dx+\frac{\beta}{2}\int_{\R}\Phi_{B}\left(\cdot-x_{0}\right)|\varepsilon_1|^4\,dx\\
    &\quad + \alpha \int_{\R}|\varepsilon_1|^2R_{0}^{(3)}\,dx+ 2\alpha \int_{\R} R_{0}^{(1)}\varepsilon_3\overline{\varepsilon}_1\,dx\\&\quad+ \alpha c_0 \int_{\R}\Phi_{B}\left(\cdot-x_{0}\right) \varepsilon_3 \varepsilon_4 \,dx +2 \omega_0 \int_{\R}\Phi_{B}\left(\cdot-x_{0}\right) \varepsilon_2 \overline{\varepsilon}_1\,dx \\
    &\quad + 2 \beta \int_{\R} \Re(R_0^{(1)} \overline{\varepsilon}_1)R_0^{(1)} \overline{\varepsilon}_1\,dx+2 c_0 \int_{\R}\Phi_{B}\left(\cdot-x_{0}\right) \varepsilon_2 \partial_{x}\overline{\varepsilon}_1\,dx.
				\end{aligned}
		\]

				For the sake of simplicity, we assume that $x_{0}=0$. We set $z_1=\varepsilon_1\sqrt{\Phi_{B}}$ , $z_2=\varepsilon_2\sqrt{\Phi_{B}}$, $z_3=\varepsilon_3\sqrt{\Phi_{B}}$ and $z_4=\varepsilon_4\sqrt{\Phi_{B}}$. Then, by simple calculations,
\[
 	\int_{\R}\left|\partial_{x} \varepsilon_1\right|^{2} \Phi_{B}\,dx=\int_{\R}\left|\partial_{x} z_1\right|^{2}\,dx+\frac{1}{4} \int|z_1|^{2}\left(\frac{\Phi_{B}^{\prime}}{\Phi_{B}}\right)^{2}\,dx-2 \int_{\R} \partial_{x} z_1 \bar{z}_1 \frac{\Phi_{B}^{\prime}}{\Phi_{B}}\,dx,\]
  $$\int_{\R}\Phi_{B}\varepsilon_2 \partial_{x}\overline{\varepsilon}_1\,dx=\int_{\R}z_2 \partial_{x}\overline{z}_1\,dx-\frac{1}{2}\int_{\R}z_2 \overline{z}_1 \frac{\Phi_{B}^{\prime}}{\Phi_{B}}\,dx$$
  \[
  \int_{\R}|\varepsilon_j|^{2} \Phi_{B}\,dx=\int_{\R}|z_j|^{2}\,dx,
  \quad j=1,2,3,4\] 
  \[
  \int_{\R}|\varepsilon_1|^{4} \Phi_{B}\,dx=\int_{\R}|z_1|^{2}\frac{1}{\Phi_{B}}\,dx,
  \] 
				and \[\int_{\R} \varepsilon_k \varepsilon_j \Phi_{B}\,dx=\int_{\R}z_j z_k\,dx. \quad j \neq k
\]
	 	Since, by definition of $\Phi_{B}$, we have $\left|\Phi_{B}^{\prime}\right| \leqslant(C / B) \Phi_{B}$, we obtain
 \[
 \begin{split}
 \int_{\R}\left|\partial_{x} z_1\right|^{2}\,dx-\frac{C}{B} \int_{\R}\left(\left|\partial_{x} z_1\right|^{2}+|z_1|^{2}\right) \,dx& \leqslant \int_{\R}\left|\partial_{x} \varepsilon_1\right|^{2} \Phi_{B}\,dx\\&
 \leqslant \int_{\R}\left|\partial_{x} z_1\right|^{2}\,dx+\frac{C}{B} \int_{\R}\left(\left|\partial_{x} z_1\right|^{2}+|z_1|^{2}\right)\,dx.
 \end{split}\]
 Also,  \[
  \int_{\R}|\varepsilon_1|^{4} \Phi_{B}\,dx\geq \int_{\R}|z_1|^{2}\,dx,
  \] 
 We also have
 $$\alpha \int_{\R}|\varepsilon_1|^2R_{0}^{(3)}\,dx=\alpha \int_{\R}|z_1|^2R_{0}^{(3)}\frac{1}{\Phi_{B}}\,dx,$$
\[
 \begin{aligned} 	2\alpha \int_{\R} R_{0}^{(1)}\varepsilon_3\overline{\varepsilon}_1\,dx=2\alpha \int_{\R} R_{0}^{(1)}z_3\overline{z}_1\frac{1}{\Phi_{B}}\,dx
				\end{aligned}
\]
and $$ 2 \beta \int_{\R} \Re(R_0^{(1)} \overline{\varepsilon}_1)R_0^{(1)} \overline{\varepsilon}_1\,dx= 2 \beta \int_{\R} \Re(R_0^{(1)} \overline{z}_1)R_0^{(1)} \overline{z}_1\frac{1}{\Phi_{B}}\,dx$$
								Since $\Phi_{B} \equiv 1$ on $[-B, B]$ and $\Phi_{\omega_{0}}^{(1)}(x) \leqslant C e^{-\left(\sqrt{\omega_{0}} / 2\right)|x|}$, we have, for all $x \in \mathbb{R}$,
				
\[
				\left|\frac{1}{\Phi_{B}}-1\right| \Phi_{\omega_{0}}^{(1)}(x) \leqslant e^{-\left(\sqrt{\omega_{0}}-2 / B\right)|x| / 2} \leqslant C e^{-\sqrt{\omega_{0}} B / 4} \leqslant \frac{1}{B}
\]
								for $B$ large enough. Thus,
\[
				\begin{aligned}
					\alpha \int_{\R}|\varepsilon_1|^2R_{0}^{(3)}\,dx\leqslant C\alpha \int_{\R}|z_1|^2R_{0}^{(3)}\,dx+\frac{C}{B} \int_{\R}|z_1|^{2}\,dx.
				\end{aligned}
\]
\[
 \begin{aligned} 	2\alpha \int_{\R} R_{0}^{(1)}\varepsilon_3\overline{\varepsilon}_1\,dx\leq C \int_{\R} R_{0}^{(1)}z_3\overline{z}_1\,dx+\frac{C}{B}\int_{\R} z_3\overline{z}_1\,dx
				\end{aligned}
\]
and $$ 2 \beta \int_{\R} \Re(R_0^{(1)} \overline{\varepsilon}_1)R_0^{(1)} \overline{\varepsilon}_1\,dx\leq C  \int_{\R} \Re(R_0^{(1)} \overline{z}_1)R_0^{(1)} \overline{z}_1\,dx+\frac{C}{B}\int_{\R} \Re( \overline{z}_1) \overline{z}_1\,dx.$$
Collecting these calculations, we obtain
\[
				\langle \mathcal{H}_{\Phi_{B}}\vec{\varepsilon}, \vec{\varepsilon}\rangle  \geqslant \langle \mathcal{H}_{0}{\vec{z}}, {\vec{z}}\rangle-\frac{C}{B} \int\left(\left|\partial_{x} z_1\right|^{2}+|z_1|^{2}+ |z_2|^{2}+|z_3|^{2}+|z_4|^{2}+|z_1|^{4}\right).
\]
								Thanks to the orthogonality conditions on ${\vec z}=(z_1,z_2,z_3,z_4)$, we verify easily using the property of $\Phi_{B}$ that
$ 
(\vec{z},\partial_{x}\vec{R})_{L^2(\R)}=(\vec{z},\vec{\Gamma})_{L^2(\R)}=(\vec{z},\vec{\Upsilon})_{L^2(\R)}=0,
$ 
				for $B$ large enough. By Proposition \ref{coercivity}, we obtain  for $B$ large enough that
\[
				\begin{aligned}
   \langle \mathcal{H}_{\Phi_{B}}\vec{\varepsilon}, \vec{\varepsilon}\rangle  & \geqslant\left(C-\frac{C}{B}\right)\|{\vec{z}}\|_{X}^{2}\\& \geqslant \frac{C}{2}\|{\vec{z}}\|_{X}^{2} \\&\geqslant \frac{C}{2}\left(1-\frac{C}{B}\right) \int_{\R}\left(\left|\partial_{x} z_1\right|^{2}+|z_1|^{2}+ |z_2|^{2}+|z_3|^{2}+|z_4|^{2}+|z_1|^{4}\right) \Phi_{B}\,dx \\
					& \geqslant \frac{C}{4} \int_{\R}\left(\left|\partial_{x} z_1\right|^{2}+|z_1|^{2}+ |z_2|^{2}+|z_3|^{2}+|z_4|^{2}+|z_1|^{4}\right) \Phi_{B}\,dx,
				\end{aligned}
\]
				implying \begin{equation}\label{claim1}
					\langle \mathcal{H}_{\Phi_{B}}\vec{\varepsilon}, \vec{\varepsilon}\rangle \geqslant C\|\vec{\varepsilon}\|_{X}^{2}.   
				\end{equation}
 Now, let $B>B_{0}$, and   $L>0$. Since $\sum_{k=1}^{N} \phi_{k}(t) \equiv 1$, we decompose $\langle \mathcal{H}\vec{\varepsilon}, \vec{\varepsilon}\rangle$ as follows:
\[
				\begin{aligned}
					&\langle \mathcal{H}_{loc}\vec{\varepsilon}, \vec{\varepsilon}\rangle=\\ & \sum_{k=1}^{N} \int_{\R} \Phi_{B}\left(\cdot-c_kt-\tilde{x}_{k}\right)\bigg\{\left|\partial_{x} \varepsilon_1\right|^{2}+|\varepsilon_1|^{2}+\frac{\beta}{2}|\varepsilon_1|^{4}+|\varepsilon_2|^{2}+\frac{\alpha}{2}|\varepsilon_3|^{2}+\frac{\alpha}{2}|\varepsilon_4|^{2}+2\tilde{\omega}_k \overline{\varepsilon}_1 \varepsilon_2\\
     &\quad +\alpha c_{k}\varepsilon_3 \varepsilon_4+2c_{k}\partial\overline{\varepsilon}_1 \varepsilon_2\bigg\}\,dx + \alpha \int_{\R}|\varepsilon_1|^2\tilde{R}_{k}^{(3)}\,dx+ 2\alpha \int_{\R} \tilde{R}_{k}^{(1)}\varepsilon_3\overline{\varepsilon}_1\,dx+2 \beta \int_{\R} \Re(\tilde{R}_k^{(1)} \overline{\varepsilon}_1)\tilde{R}_k^{(1)} \overline{\varepsilon}_1\,dx\\
					& +\sum_{k=1}^{N}\int_{\R}\left(\phi_{k}-\Phi_{B}\left(\cdot-c_kt-\tilde{x}_{k}\right)\right)\bigg\{\left|\partial_{x} \varepsilon_1\right|^{2}+|\varepsilon_1|^{2}+\frac{\beta}{2}|\varepsilon_1|^{4}+|\varepsilon_2|^{2}+\frac{\alpha}{2}|\varepsilon_3|^{2}+\frac{\alpha}{2}|\varepsilon_4|^{2}-2\tilde{\omega}_k \overline{\varepsilon}_1 \varepsilon_2\\
     &\quad -\alpha c_{k}\varepsilon_3 \varepsilon_4-2c_{k}\partial\overline{\varepsilon}_1 \varepsilon_2\bigg\}\,dx.
				\end{aligned}
	\]
				By \eqref{claim1}, for any $k=1, \ldots, N$, we have  for $B$ large enough that
\[
				\begin{aligned}
			 &\int_{\R} \Phi_{B}\left(\cdot-x_k(t)\right)		\bigg\{\left|\partial_{x} \varepsilon_1\right|^{2}+|\varepsilon_1|^{2}+\frac{\beta}{2}|\varepsilon_1|^{4}+|\varepsilon_2|^{2}+\frac{\alpha}{2}|\varepsilon_3|^{2}+\frac{\alpha}{2}|\varepsilon_4|^{2}-2\tilde{\omega}_k \overline{\varepsilon}_1 \varepsilon_2\\
     &\quad -\alpha c_{k}\varepsilon_3 \varepsilon_4-2c_{k}\partial\overline{\varepsilon}_1 \varepsilon_2\bigg\}\,dx   + \alpha \int_{\R}|\varepsilon_1|^2\tilde{R}_{k}^{(3)}\,dx+ 2\alpha \int_{\R} \tilde{R}_{k}^{(1)}\varepsilon_3\overline{\varepsilon}_1\,dx+2 \beta \int_{\R} \Re(\tilde{R}_k^{(1)} \overline{\varepsilon}_1)\tilde{R}_k^{(1)} \overline{\varepsilon}_1\,dx
					\\&\geqslant \lambda_{k} \int_{\R}\Phi_{B}\left(\cdot-x_{k}(t)\right)\left(\left|\partial_{x} \varepsilon_1\right|^{2}+|\varepsilon_2|^{2}+|\varepsilon_1|^{2}+|\varepsilon_3|^{2}+|\varepsilon_4|^{2}\right)\,dx,
				\end{aligned}
\]
				where $x_k(t)=c_k t +\tilde{x}_k$.
								Moreover, by the properties of $\Phi_{B}$ and $\phi_{k}(t)$, for $L$ large enough, we have
				
\[
				\phi_{k}(t)-\Phi_{B}\left(\cdot-x_{k}(t)\right) \geqslant-e^{-L /(4 B)},
\]
	and  for
	 $\delta_{k}=\delta_{k}\left(c_k,\tilde{\omega}_{k}\right)>0$,
\[
		\begin{aligned}
		  &  \left|\partial_{x} \varepsilon_1\right|^{2}+|\varepsilon_1|^{2}+\frac{\beta}{2}|\varepsilon_1|^{4}+|\varepsilon_2|^{2}+\frac{\alpha}{2}|\varepsilon_3|^{2}+\frac{\alpha}{2}|\varepsilon_4|^{2}-2\tilde{\omega}_k \overline{\varepsilon}_1 \varepsilon_2 -\alpha c_{k}\varepsilon_3 \varepsilon_4-2c_{k}\partial\overline{\varepsilon}_1 \varepsilon_2\\&\geqslant \delta_{k}\left(\left|\partial_{x} \varepsilon_1\right|^{2}+|\varepsilon_1|^{2}+|\varepsilon_2|^{2}+|\varepsilon_3|^{2}+|\varepsilon_4|^{2}\right) \geqslant 0.
		\end{aligned}		
\]
				So,
\[
				\begin{aligned}
					& \int_{\R}\left(\phi_{k}-\Phi_{B}\left(\cdot-x_k(t)\right)\right)\bigg\{\left|\partial_{x} \varepsilon_1\right|^{2}+|\varepsilon_1|^{2}+\frac{\beta}{2}|\varepsilon_1|^{4}+|\varepsilon_2|^{2}+\frac{\alpha}{2}|\varepsilon_3|^{2}+\frac{\alpha}{2}|\varepsilon_4|^{2}-2\tilde{\omega}_k \overline{\varepsilon}_1 \varepsilon_2\\
     &\quad -\alpha c_{k}\varepsilon_3 \varepsilon_4-2c_{k}\partial\overline{\varepsilon}_1 \varepsilon_2\bigg\}\,dx \\
					&  \geqslant \delta_{k} \int_{\R}\left(\varphi_{k}(t)-\Phi_{B}\left(\cdot-x_{k}(t)\right)\right)\left(\left|\partial_{x} \varepsilon_1\right|^{2}+|\varepsilon_1|^{2}+|\varepsilon_2|^{2}+|\varepsilon_3|^{2}+|\varepsilon_4|^{2}\right) \,dx\\
					&\geqslant-C e^{-L /(4 B)} \int_{\R}\left(\left|\partial_{x} \varepsilon_1\right|^{2}+|\varepsilon_1|^{2}+|\varepsilon_2|^{2}+|\varepsilon_3|^{2}+|\varepsilon_4|^{2}\right) \,dx.
				\end{aligned}
\]
    Thus, our above considerations reveal that 
\[
				\begin{aligned}
					&\langle \mathcal{H}_{loc}\vec{\varepsilon}, \vec{\varepsilon}\rangle \\&\geqslant K \int_{\R}\left(\sum_{k=1}^{N} \phi_{k}\right)\left(\left|\partial_{x} \varepsilon_1\right|^{2}+|\varepsilon_1|^{2}+|\varepsilon_2|^{2}+|\varepsilon_3|^{2}+|\varepsilon_4|^{2}\right)\,dx\\
     &\quad-C e^{-L /(4 B)} \int_{\R}\left(\left|\partial_{x} \varepsilon_1\right|^{2}+|\varepsilon_1|^{2}+|\varepsilon_2|^{2}+|\varepsilon_3|^{2}+|\varepsilon_4|^{2}\right)\,dx  
				\end{aligned}
\]
				
				and since $\sum_{k=1}^{N} \phi_{k}(t) \equiv 1$, we obtain the result by taking $L$ large enough.
			\end{proof}
We now propose a new way to write the previous lemma.
\begin{lemma}\label{newS}
    There exists \(T_{0}\) such that if \(t_{0} > T_{0}\), then for all \(t \in [t_{0}, T^{n}]\),
$$
\mathcal{S}( \Vec{u}(t), t) = \mathcal{S}\left( \Vec{R}(t), t \right) + \mathcal{H}_{loc}( \vec{\varepsilon}(t), t) +\sum_{j=1}^{N}\mathcal{O}\left(|\tilde{\omega}_{j} -\omega_{j}|^{2}\right)+ O\left(e^{-3  \sqrt{\omega_{\star}} c_{\star} t}\right).
$$
\end{lemma}
\begin{proof}
    Let us $t \in [t_{0}, T^{n}].$ To prove this result, from the previous lemma, it would suffice to show that
    $$\mathcal{S}\left( \Vec{\tilde{R}}(t), t \right)=\sum_{j=1}^{N}\mathcal{S}_{j,loc}\left( \Vec{\tilde{R}}_j(t), t \right)+ O\left(e^{-3  \sqrt{\omega_{\star}} c_{\star} t}\right).$$
    This follows as in the previous lemma. In fact, let us show some steps with the components of $S$,
\begin{equation}\label{enesoma}
\begin{aligned}
     E\left(\Vec{\tilde{R}}(t),t\right)  
    & = \int_\rr\left|\sum_{s=1}^{N} \partial_{x} \tilde{R}_{s}^{(1)}\right|^{2} \,dx + \int_\rr\left|\sum_{k=1}^{N}  \tilde{R}_{k}^{(2)}\right|^{2} \,dx+ \frac{\alpha}{2}\int_\rr\left|\sum_{k=1}^{N}  \tilde{R}_{k}^{(4)}\right|^{2} \,dx \\
    & \quad + \frac{\beta}{2} \int_\rr\left|\sum_{k=1}^{N}  \tilde{R}_{k}^{(1)}\right|^{4} \,dx + \alpha\Re \int_\rr\left(\sum_{k=1}^{N} \tilde{R}_{k}^{(4)}\right)\left(\sum_{s=1}^{N} \bar{\tilde{R}}_{s}^{(1)}\right)^{2} \,dx\\
    &\quad + \frac{\alpha}{2}\int_\rr\left|\sum_{k=1}^{N}  \tilde{R}_{k}^{(3)}\right|^{2} \,dx++ \int_\rr\left|\sum_{k=1}^{N}  \tilde{R}_{k}^{(1)}\right|^{2} \,dx
\end{aligned}
\end{equation}
We will expand the terms of the above integrals, 
$$
\begin{aligned}
 \left|\sum_{s=1}^{N} \partial_{x} \tilde{R}_{s}^{(1)}\right|^{2}&=\left(\partial_{x} \tilde{R}_{1}^{(1)}+\cdots+\partial_{x} \tilde{R}_{N}^{(1)}\right)\left(\partial_{x} \bar{R}_{1}^{(1)}+\cdots+\partial_{x} \bar{R}_{N}^{(1)}\right),
\end{aligned}
$$
we have
\begin{equation}\label{somagradiente}
\begin{aligned}
 \left|\sum_{s=1}^{N} \partial_{x} \tilde{R}_{s}^{(1)}\right|^{2}&=\sum_{s=1}^{N}\left|\partial_{x} \tilde{R}_{s}^{(1)}\right|^{2}+\sum_{\substack{s, m=1 \\ s \neq m}}^{N} \partial_{x} \tilde{R}_{s}^{(1)} \partial_{x} \bar{\tilde{R}}_{j}^{(1)} \\&= \sum_{s=1}^{N}\left|\partial_{x} \tilde{R}_{s}^{(1)}\right|^{2}\phi_{s}+\sum_{s=1}^{N}\left|\partial_{x} \tilde{R}_{s}^{(1)}\right|^{2}\sum_{\substack{k=1 \\ k\neq s}}^{N}\phi_{k}+\sum_{\substack{s, m=1 \\ s \neq m}}^{N} \partial_{x} \tilde{R}_{s}^{(1)} \partial_{x} \bar{\tilde{R}}_{j}^{(1)}  
\end{aligned}
\end{equation}

Now, using the fact that the solitons are bounded (see Proposition \ref{decaimentoquadratico}) and Lemma \ref{solitons}, we obtain
\begin{equation*}\label{somadeorden4}
\begin{aligned}
&\sum_{s=1}^{N}\int_{\R}\left|\partial_{x} \tilde{R}_{s}^{(1)}\right|^{2}\sum_{\substack{k=1 \\ k\neq s}}^{N}\phi_{k}\,dx+\sum_{\substack{s, m=1 \\ s \neq m}}^{N} \int_{\mathbb{R}} \partial_{x} \tilde{R}_{s}^{(1)} \partial_{x} \bar{\tilde{R}}_{j}^{(1)}\phi_{j}\,dx\\
&\qquad\leq C\sum_{\substack{s=1 \\ s \neq j}}^{N}(1+|\theta_{s}|)e^{-4\sqrt{\omega_{\star}}c_{\star}t}+C\sum_{\substack{s, m=1 \\ s \neq m}}^{N}(1+|\theta_{s}|+|\theta_{m}|)^{2}e^{-4\sqrt{\omega_{\star}}c_{\star}t}\\
&\qquad\leq C \max_{s\neq m}\{(1+|\theta_{s}|+|\theta_{m}|)^{2}\}e^{-4\sqrt{\omega_{\star}}c_{\star}t}.
\end{aligned}
\end{equation*}
Taking $T_{0}$ sufficiently large such that 
\begin{equation}\label{T_{0}3}
\max_{s\neq m}\{(1+|\theta_{s}|+|\theta_{m}|)^{2}\}e^{-4\sqrt{\omega_{\star}}c_{\star}t}\leq e^{-3\sqrt{\omega_{\star}}c_{\star}t},  
\end{equation} 
we have that
\begin{equation}\label{resulsoma}
\begin{aligned}
\sum_{s=1}^{N}\int_{\R}\left|\partial_{x} \tilde{R}_{s}^{(1)}\right|^{2}\sum_{\substack{k=1 \\ k\neq s}}^{N}\phi_{k}\,dx+\sum_{\substack{s, m=1 \\ s \neq m}}^{N} \int_{\mathbb{R}} \partial_{x} \tilde{R}_{s}^{(1)} \partial_{x} \bar{\tilde{R}}_{j}^{(1)}\phi_{j}\,dx\leq Ce^{-3\sqrt{\omega_{\star}}c_{\star}t},
\end{aligned}
\end{equation}
from where $$\int_\rr\left|\sum_{s=1}^{N} \partial_{x} \tilde{R}_{s}^{(1)}\right|^{2} \,dx=\sum_{j=1}^{N}\int_\rr\left| \partial_{x} \tilde{R}_{j}^{(1)}\right|^{2} \phi_j\,dx+Ce^{-3\sqrt{\omega_{\star}}c_{\star}t}.$$
Thus, following the same process, we can write that
\begin{equation*}
\begin{aligned}
E\left(\Vec{\tilde{R}}(t),t\right)
& =\sum_{j=1}^{N}\bigg(\int_{\mathbb{R}}\left|\partial_{x} \tilde{R}_{j}^{(1)}\right|^{2} \phi_{j}\,dx+ \int\left|\partial_{x} \tilde{R}_{j}^{(1)}\right|^{2} \phi_{j} \,dx+ \int\left|\partial_{x} \tilde{R}_{j}^{(2)}\right|^{2} \phi_{j} \,dx \\
&\quad+ \frac{\alpha}{2}\int\left| \tilde{R}_{j}^{(3)}\right|^{2} \phi_{j} \,dx+ \frac{\alpha}{2}\int\left| \tilde{R}_{j}^{(4)}\right|^{2} \phi_{j} \,dx+\alpha \Re \int_{\mathbb{R}} \tilde{R}_{j}^{(3)}\bar{\tilde{R}}_{j}^{(1)} \phi_{j}\,dx\\
&\quad+ \frac{\beta}{2}\int\left| \tilde{R}_{j}^{(1)}\right|^{4} \phi_{j} \,dx\bigg)+\mathcal{O}\left(e^{-3\sqrt{\omega_{\star}}c_{\star}t}\right).
\end{aligned}  
\end{equation*}
Similarly, we can analyze 
$$
\begin{aligned}
Q(\Vec{\tilde{R}}(t),t) 
=2\Re\int_{\mathbb{R}} \left(\sum_{s=1}^{N} \partial_{x}\tilde{R}_{s}^{(1)}\right)\overline{\left(\sum_{s=1}^{N} \tilde{R}_{s}^{(2)}\right)}\,dx-\alpha \int_{\mathbb{R}}\left(\sum_{k=1}^{N} \tilde{R}_{k}^{(3)}\right)\left(\sum_{k=1}^{N} \tilde{R}_{k}^{(4)}\right) \,dx.
\end{aligned}
$$
Using the same argument applied in the estimate for \(E\), we have 
\begin{equation*}
\begin{aligned}
Q_1(\Vec{\tilde{R}}(t),t) =\sum_{j=1}^{N}\bigg(2\Re\int_{\mathbb{R}}\tilde{R}_{j}^{(1)}\overline{\tilde{R}_{j}^{(2)}} \phi_{j}\,dx-\alpha \int_{\mathbb{R}}\tilde{R}_{j}^{(3)}\tilde{R}_{j}^{(4)} \phi_{j}\,dx\bigg)+O\left(e^{-3 \sqrt{\omega_{\star}} c_{\star} t}\right).
\end{aligned}    
\end{equation*}
Finally, proceeding similarly, we obtain that 
\begin{equation*}\label{m1}
 \begin{aligned}
   Q_2(\Vec{\tilde{R}}(t),t) =\sum_{j=1}^{N}\Im\int_{\mathbb{R}}2\tilde{R}_{j}^{(2)}\overline{\tilde{R}_{j}^{(1)}} \phi_{j}\,dx+O\left(e^{-3 \sqrt{\omega_{\star}} c_{\star} t}\right).
\end{aligned}   
\end{equation*}
It follows that, $$\mathcal{S}\left( \Vec{\tilde{R}}(t), t \right)=\sum_{j=1}^{N}\mathcal{S}_{j,loc}\left( \Vec{\tilde{R}}_j(t), t \right)+ O\left(e^{-3  \sqrt{\omega_{\star}} c_{\star} t}\right).$$
Moreover, to conclude the result, we define the operator for $j$ fixed, 
$$\mathcal{J}_j(z)=E_j(z)-\omega_jQ_{2,j}(z)-c_jQ_{1,j}(z).$$
We know that $\vec{R}_j$ is a critical point of $\mathcal{J}_j$. Thus, by applying the Taylor expansion formula,
\begin{equation*}
\begin{aligned}
\mathcal{J}_j(\vec{\tilde{R}}_j)=\mathcal{J}_j(\vec{R}_j)+\frac{1}{2}\mathcal{J}_{j}^{\prime \prime}(\vec{R}_j)(\tilde{\omega}_{j}-\omega_{j})^{2}+|\tilde{\omega}_{j}-\omega_{j}|^{2}o\left(|\tilde{\omega}_{j}-\omega_{j}|\right).
\end{aligned}
\end{equation*}
from which we deduce that the expression \eqref{formSfinal} can be written as
$$
\mathcal{S}( \Vec{u}(t), t) = \mathcal{S}\left( \Vec{R}(t), t \right) + \mathcal{H}_{loc}( \vec{\varepsilon}(t), t) +\sum_{j=1}^{N}\mathcal{O}\left(|\tilde{\omega}_{j} -\omega_{j}|^{2}\right)+ O\left(e^{-3  \sqrt{\omega_{\star}} c_{\star} t}\right).
$$
\end{proof}
Now we will establish a small lemma that provides us with some control over the variation between a parameter and its modulation.
\begin{lemma}\label{variationomega}
    If $t_{0} >T_{0}$, there exists $C >0$ independent of $n$ and $c_{\star}$ such that, for all $t \in [t_{0},T^{n}]$ and for all $j \in \{1,2,\ldots,N\}$
    $$|\tilde{\omega}_j(t)-\omega_j|\leq C e^{-2 \sqrt{ \omega_{\star}}c_{\star}t}.$$
\end{lemma}
\begin{proof}
    Let $t \in [t_{0},T^{n}]$ and  $j \in \{1,2,\ldots,N\}$ be fixed.  

    We know that, $$R_{j}^{(2)}=e^{i(\theta_j(x-c_jt)-\omega_jt+\gamma_j)}(i\omega_j+c_j\partial_{x})\Phi_{\omega_j}^{(2)}$$
    where $\theta_j=\frac{\omega_j c_j}{1-c_j^2}.$

    Hence, \begin{equation}\label{eq1}
        Q_{2}(\vec{R}_j)=2\Im \int_{\R}\overline{R}_{j}^{(1)} R_{j}^{(2)}\,dx=-2\omega_j\int_{\R}|\Phi_{\omega_j}|^2\,dx.
    \end{equation}
Now, since $Q_2$ is a quantity that conserves the flow of the system \eqref{system2}, then  \begin{equation}\label{eq2}
    \left|Q_{2}(\vec{\tilde{R_j}}(t))-Q_{2}(\vec{\tilde{R_j}}(T^{n}))\right|\leq Ce^{-2 \sqrt{ \omega_{\star}}c_{\star}t}.
\end{equation}
Then, since $$Q_{2}(\vec{\tilde{R_j}}(T^{n}))=Q_{2}(\vec{R_j}(T^{n})),$$
so, by Taylor's expansion,
\begin{equation}\label{eq3}
\begin{aligned}
&\int_{R}\left|\Phi_{\omega_{1}}(x)\right|^{2} d x -\int_{\mathbb{R}}\left|\Phi_{\tilde{\omega}_{1}(t)}(x)\right|^{2} d x\\
&=\left.\int_{\mathbb{R}}\left|\Phi_{\tilde{\omega}_{1}}(x)\right|^{2} d x+\left.\left(\partial_{\omega} \int\left|\Phi_{\tilde{\omega}}(x)\right|^{2}\right)\right|_{\omega=\tilde{\omega}_{j}}\right)\left( \tilde{\omega}_{j}(t)- \omega_{j}\right) \\
&\quad+\left(\tilde{\omega}_{j}(t)-\omega_{j}\right) o\left(\left|\tilde{\omega}_{j}(t)-\omega_{j}\right|^{2}\right) -\int_{\mathbb{R}}\left|\Phi_{\tilde{\omega}_{1}}(x)\right|^{2} d x \\
&=\left.\left(\partial_{\omega} \int_{\mathbb{R}}\left|\Phi_{\tilde{\omega}}(x)\right|^{2} d x\right|_{\omega=\tilde{\omega}_{j}}\right)\left(\tilde{\omega}_{j}(t)-\omega_{j}\right) +\left(\tilde{\omega}_{j}(t)-\omega_{j}\right) o\left(\left|\tilde{\omega}_{j}(t)-\omega_{j}\right|^{2}\right).
\end{aligned}
\end{equation}
Therefore, combining \eqref{eq1}, \eqref{eq2} and \eqref{eq3} we have 
 $$|\tilde{\omega}_j(t)-\omega_j|^2\leq C e^{-4 \sqrt{ \omega_{\star}}c_{\star}t}.$$
\end{proof}
Before showing Lemma \ref{Bootstrap3}, we must establish one last result that provides an almost conservation law for localized momentum 1  and momentum 2.
\begin{lemma}\label{VariacionS}
    If $t_{0} >T_{0}$, there exists $C >0$ independent of $n$ and $c_{\star}$ such that, for all $t \in [t_{0},T^{n}]$,
$$
\left|\frac{\partial \mathcal{S}(t, \Vec{u}(t))}{\partial t}\right| \leqslant \frac{C}{\sqrt{t}} e^{-2 \sqrt{ \omega_{\star}}c_{\star}t}
$$
\end{lemma}
\begin{proof}
We observe that for all $t \in [t_{0},T^{n}]$,
$$
\mathcal{S} (\vec{w}(t),t)=E(\vec{w}(t))-\sum_{j=1}^{N} c_j Q_{1,j} (\vec{w}(t),t)-\tilde{\omega}_{j} Q_{2,j}( \vec{w}(t),t)
$$
Since the energy $E$ is conserved by the flow of \eqref{system2}, to estimate the variations of $\mathcal{S}(t, \Vec{u}(t))$ it is enough to study the variations of the localized momentums $Q_{2,j}( \Vec{u}(t),t)$ and  $Q_{1,j}( \Vec{u}(t),t)$. 

For $j=2, \ldots, N$, we have 

$$
\begin{aligned}
 &\frac{\partial}{\partial t}\int_{\R}Q_{1,j}(\Vec{u}(t))dx\\
&= 2 \frac{\partial}{\partial t} \Re\int_{\R} \partial_{x}u(t) \bar{\rho}(t) \psi_{j}(x,t)\, d x-\alpha\frac{\partial}{\partial t} \int_{\R} n(t) v(t) \psi_{j}(x,t) \,d x \\
& =-2\Re\int_{\R} \frac{\partial}{\partial t} u(t) \partial_{x}\bar{\rho}(t) \psi_{j}(x,t) \, d x+2\Re\int_{\R} \partial_{x}u(t)\frac{\partial}{\partial t}\bar{\rho}(t) \psi_{j}(x,t)  \,d x+2\Re\int_{\R} \partial_{x}u(t)\bar{\rho}(t)\frac{\partial}{\partial t} \psi_{j}(x,t)  \,d x \\
& \quad -\alpha \Re \int_{\R} \frac{\partial}{\partial t} n(t) v(t) \psi_{j}(x,t) \, d x-\alpha\int_{\R}v(t)n(t) \frac{\partial}{\partial t} \psi_{j}(x,t)\, d x -\alpha \Re \int_{\R} \frac{\partial}{\partial t} v(t) n(t) \psi_{j}(x,t) \, d x.
\end{aligned}$$
Now, we observe that 
\begin{equation}\label{Q0}
    \begin{aligned}
      2\Re\int_{\R} \partial_{x}u(t)\bar{\rho}(t)\frac{\partial}{\partial t} \psi_{j}(x,t)  \,d x &=2\Re\int_{\R} \partial_{x}u(t)\bar{\rho}(t)\left(\frac{x^{1}+m_{j} t}{2 t\sqrt{t}}\right)\psi_{j}^{\prime}(x,t)   \,d x,
    \end{aligned}
\end{equation}
so, 
\begin{equation}\label{Q00}
    \begin{aligned}
     \alpha\int_{\R}v(t)n(t) \frac{\partial}{\partial t} \psi_{j}(x,t)\, d x &=\alpha\int_{\R}v(t)n(t) \left(\frac{x^{1}+m_{j} t}{2 t\sqrt{t}}\right)\psi_{j}^{\prime}(x,t)\, d x
    \end{aligned}
\end{equation}
Similarly,
\begin{equation}\label{Q1}
    \begin{aligned}
        -2\Re\int_{\R} \frac{\partial}{\partial t} u(t) \partial_{x}\bar{\rho}(t) \psi_{j}(x,t) \, d x&=2\Re\int_{\R} \rho(t) \partial_{x}\bar{\rho}(t) \psi_{j}(x,t) \, d x \\
        &=2\Re\int_{\R} \frac{1}{2}\partial_{x}|\rho(t)|^2  \psi_{j}(x,t) \, d x\\
        &=-\frac{1}{\sqrt{t}}\Re\int_{\R} |\rho(t)|^2  \psi_{j}^{\prime}(x,t) \, d x,
    \end{aligned}
\end{equation}
\begin{equation}\label{Q2}
    \begin{aligned}
        2\Re\int_{\R} \partial_{x}u(t)\frac{\partial}{\partial t}\bar{\rho}(t) \psi_{j}(x,t)  \,d x&=2\Re\int_{\R} \partial_{x}u(t)\left(-\bar{u}_{xx}+\bar{u}+\al \bar{u}v+\beta|u|^2\bar{u}\right) \psi_{j}(x,t)  \,d x \\&=-2\Re\int_{\R} \partial_{x}u(t)\partial_{xx}\bar{u}(t)\psi_{j}(x,t)  \,d x+2\Re\int_{\R} \partial_{x}u(t)\bar{u}(t)\psi_{j}(x,t)  \,d x \\&\quad+2\alpha\Re\int_{\R} \partial_{x}u(t)\bar{u}(t)v(t)\psi_{j}(x,t)  \,d x+2\beta\Re\int_{\R} \partial_{x}u(t)|u|^2(t)\bar{u}(t)\psi_{j}(x,t)  \,d x\\
        &=\frac{1}{\sqrt{t}}\Re\int_{\R} |\partial_{x}u(t)|^2\psi_{j}^{\prime}(x,t)  \,d x-\frac{1}{\sqrt{t}}\Re\int_{\R} |u(t)|^2\psi_{j}^{\prime}(x,t)  \,d x\\
        & \quad -\frac{\alpha}{\sqrt{t}}\Re\int_{\R} |u(t)|^2v(t)\psi_{j}^{\prime}(x,t)  \,d x-\alpha\Re\int_{\R} |u(t)|^2 \partial_{x}v(t)\psi_{j}(x,t)  \,d x\\
        &\quad-\frac{\alpha}{\sqrt{t}}\Re\int_{\R} |u(t)|^2v(t)\psi_{j}^{\prime}(x,t)  \,d x-\frac{\beta}{2}\frac{1}{\sqrt{t}}\Re\int_{\R} |u(t)|^4\psi_{j}^{\prime}(x,t)  \,d x ,
    \end{aligned}
\end{equation}
\begin{equation}\label{Q3}
    \begin{aligned}
        -\alpha \Re \int_{\R} \frac{\partial}{\partial t} n(t) v(t) \psi_{j}(x,t) \, d x&=-\alpha \Re \int_{\R} \left(\partial_{x}(|u|^2)+\partial_{x}v\right) v(t) \psi_{j}(x,t) \, d x\\
        &= \frac{\alpha}{\sqrt{t}}\Re \int_{\R} \left(|u(t)|^2 v(t)+\frac{1}{2}|v(t)|^2 \right) \psi_{j}^{\prime}(x,t) \, d x\\
        &\quad+\alpha \Re \int_{\R} |u(t)|^2 \partial_{x}v(t) \psi_{j}(x,t) \, d x,
    \end{aligned}
\end{equation}
\begin{equation}\label{Q4}
    \begin{aligned}
        -\alpha \Re \int_{\R} \frac{\partial}{\partial t} v(t) n(t) \psi_{j}(x,t) \, d x&=-\alpha \Re \int_{\R} \partial_{x}n(t) n(t) \psi_{j}(x,t) \, d x\\
        &=\frac{\alpha}{2}\frac{1}{\sqrt{t}} \Re \int_{\R} |n(t)|^2  \psi_{j}^{\prime}(x,t) \, d x.
    \end{aligned}
\end{equation}
Then, combining \eqref{Q1}-\eqref{Q4}, we have $$\left|\frac{\partial}{\partial t}\int_{\R}Q_{1,j}(\Vec{u}(t))dx\right|\leq \frac{C}{\sqrt{t}}\int_{\R}\left(|u(t)|^2+|u(t)|^4+|\partial_{x}u(t)|^2+|n(t)|^2+|v(t)|^2+|\rho(t)|^2\right)\psi_{j}^{\prime}(x,t) \, d x.$$
Now, we proceed to estimate $Q_{1,j}$ for all $j$. In fact, 
$$\begin{aligned}
   &\frac{\partial}{\partial t}\int_{\R}Q_{2,j}(\Vec{u}(t))dx\\&=2\Im \int_{\R}\partial_{t}\overline{u}(t)\rho(t)\psi_j(x,t)\,dx+2\Im \int_{\R}\overline{u}(t)\partial_{t}\rho(t)\psi_j(x,t)\,dx +2\Im\int_{\R}\overline{u}(t)\rho(t)\partial_{t}\psi_j(x,t)\,dx\\
   &=-2\Im \int_{\R}\overline{\rho}(t)\rho(t)\psi_j(x,t)\,dx+2\Im \int_{\R}\overline{u}(t)(-\partial_{xx}^2u(t)+u(t)+\alpha u(t)v(t)+\beta|u(t)|^2u(t))\psi_j(x,t)\,dx\\
   &\quad+2\Im\int_{\R}\overline{u}(t)\rho(t)\left(\frac{x^{1}+m_{j} t}{2 t\sqrt{t}}\right)\psi_{j}^{\prime}(x,t)\,dx\\
   &=-2\Im\int_{\R}\overline{u}(t)\partial_{xx}^2u(t)\psi_j(x,t)\,dx+2\Im\int_{\R}|u(t)|^2\psi_j(x,t)\,dx+2\alpha \Im\int_{\R}|u(t)|^2v(t)\psi_j(x,t)\,dx\\
   &\quad+2\beta \Im\int_{\R}|u(t)|^4\psi_j(x,t)\,dx+2\Im\int_{\R}\overline{u}(t)\rho(t)\left(\frac{x^{1}+m_{j} t}{2 t\sqrt{t}}\right)\psi_{j}^{\prime}(x,t)\,dx\\
   &=2\Im\int_{\R}|\partial_{x}u(t)|^2\psi_j(x,t)\,dx+\frac{2}{\sqrt{t}}\Im\int_{\R}\overline{u}(t)\partial_{x}u(t)\psi_j^{\prime}(x,t)\,dx\\
   &\quad+2\Im\int_{\R}\overline{u}(t)\rho(t)\left(\frac{x^{1}+m_{j} t}{2 t\sqrt{t}}\right)\psi_{j}^{\prime}(x,t)\,dx\\
   &=\frac{2}{\sqrt{t}}\Im\int_{\R}\overline{u}(t)\partial_{x}u(t)\psi_j^{\prime}(x,t)\,dx+2\Im\int_{\R}\overline{u}(t)\rho(t)\left(\frac{x^{1}+m_{j} t}{2 t\sqrt{t}}\right)\psi_{j}^{\prime}(x,t)\,dx.
\end{aligned}$$
Similarly, we have that $$\left|\frac{\partial}{\partial t}\int_{\R}Q_{2,j}(\Vec{u}(t))dx\right|\leq \frac{C}{\sqrt{t}}\int_{\R}\left(|u(t)|^2+|\partial_{x}u(t)|^2+|\rho(t)|^2\right)\psi_{j}^{\prime}(x,t) \, d x.$$

We define $\Omega_{j}=\left[m_{j} t-\sqrt{t}, m_{j} t+\sqrt{t}\right]$, then  $$\left|\frac{\partial}{\partial t}\int_{\R}\left(Q_{1,j}(\Vec{u}(t))+Q_{2,j}(\Vec{u}(t))\right)dx\right|\leq \frac{C}{\sqrt{t}}\int_{\Omega_j}\left(|u(t)|^2+|\partial_{x}u(t)|^2+|n(t)|^2+|v(t)|^2+|\rho(t)|^2\right) \, d x.$$
Now, as for $j \in \{2,\ldots,N\}$ we have
\begin{equation}\label{termomomento1}
\begin{aligned}
 \int_{\Omega_{j}}&\left(|\partial_{x} u(t)|^{2}+|u(t)|^{2} +|v(t)|^{2}+|\rho(t)|^{2}+|n(t)|^{2} \right) \,dx\\& \leqslant \int_{\Omega_{j}}\left(|\partial_{x}\tilde{R}_{1}(t)|^{2}+|\tilde{R}_{1}(t)|^{2}+|\tilde{R}_{3}(t)|^{2}+|\tilde{R}_{4}(t)|^{2} + |\tilde{R}_{2}(t)|^{2}\right) \,dx \\
&\qquad+ \|\Vec{u}(t) - \vec{\tilde{R}}(t)\|_{X}^{2}   
\end{aligned}
\end{equation}
and
\begin{equation}\label{termomomento2}
\int_{\Omega_{j}} |u(t)|^{4} \,dx \leqslant \int_{\Omega_{j}} |\tilde{R}_{1}(t)|^{4} dx + \|\Vec{u}(t) - \vec{\tilde{R}}(t)\|_{X}^{4}.
\end{equation}
Now, let us estimate the terms involving $R_1$. For the case of $\partial_{x}R_1, R_2,R_3,R_4$, we proceed similarly. Note that, by using the soliton decay (Proposition \ref{decaimentoquadratico})
\begin{equation}\label{Rnosupor}
\begin{aligned}
\int_{\Omega_{j}}\left|\tilde{R}_{1}\right|^{2}  \,dx \leq C\sum_{k=1}^{N}\int_{\Omega_{j}}\left|R_{k}^{(1)}\right|  \,dx\leq  \sum_{k=1}^{N}\int_{\Omega_{j}}e^{-\frac{2}{3}\sqrt{\omega_{\star}}|x-c_{k}t-\tilde{x}_{k}|}\,dx.
\end{aligned}
\end{equation}

Let us estimate $\int_{\Omega_{j}}e^{-\frac{2}{3}\sqrt{\omega_{\star}}|x-c_{k}t-\tilde{x}_{k}|}\,dx$. To do so, we will use the same arguments used in Lemma \ref{solitons}, therefore, some facts will be written without much justification. In fact, if $j < k$, for sufficiently large $T_{0}$,
\begin{equation}\label{integralexpo}
\begin{aligned}
 \int_{\Omega_{j}}e^{-\frac{2}{3}\sqrt{\omega_{\star}}|x-c_{k}t-\tilde{x}_{k}|}\,dx&\leq C\int_{-\sqrt{t}+m_{j}t}^{\sqrt{t}+m_{j}t} e^{-\frac{2}{3}\sqrt{\omega_{\star}}|x^{1}-c_{k}^{1}t|}\,dx^{1} \\
 &=C\int_{-\sqrt{t}+m_{j}t}^{\sqrt{t}+m_{j}t} e^{-\frac{2}{3}\sqrt{\omega_{\star}}|x^{1}-m_{j}t+m_{j}t-c_{k}^{1}t|}\,dx^{1} \\
 &\leq C\int_{-\sqrt{t}+m_{j}t}^{\sqrt{t}+m_{j}t} e^{-\frac{2}{3}\sqrt{\omega_{\star}}\left|x^{1}-m_{j}t+\frac{c_{j-1}-c_{k}}{2}t+\frac{c_{j}-c_{k}}{2}t\right|}\,dx^{1}\\
 &\leq Ce^{-\frac{1}{3}\sqrt{\omega_{\star}}c_{\star}t}\int_{-\sqrt{t}+m_{j}t}^{\sqrt{t}+m_{j}t} e^{\frac{2}{3}\sqrt{\omega_{\star}}\left|x^{1}-m_{j}t+\frac{c_{j}-c_{k}}{2}t\right|}\,dx^{1}\\
 &\leq Ce^{-\frac{1}{3}\sqrt{\omega_{\star}}c_{\star}t}\int_{-\sqrt{t}+\frac{c_{j}-c_{k}}{2}t}^{\sqrt{t}+\frac{c_{j}-c_{k}}{2}t} e^{\frac{2}{3}\sqrt{\omega_{\star}}|x^{1}|}\,dx^{1}\\
 &\leq Ce^{-\frac{1}{3}\sqrt{\omega_{\star}}c_{\star}t}\int_{\mathbb{R}} e^{-\frac{2}{3}\sqrt{\omega_{\star}}x^{1}}\,dx^{1}\\
 &\leq Ce^{-4\sqrt{\omega_{\star}}c_{\star}t}.
\end{aligned}
\end{equation}
Now, if $j > k$,
\begin{equation}\label{integralexpo2}
\begin{aligned}
 \int_{\Omega_{j}}e^{-\frac{2}{3}\sqrt{\omega_{\star}}|x-c_{k}t-\tilde{x}_{k}|}\,dx&\leq C\int_{-\sqrt{t}+m_{j}t}^{\sqrt{t}+m_{j}t} e^{-\frac{2}{3}\sqrt{\omega_{\star}}|x^{1}-c_{k}^{1}t|}\,dx^{1}\\
 &\leq C\int_{-\sqrt{t}+m_{k+1}t}^{\infty} e^{-\frac{2}{3}\sqrt{\omega_{\star}}|x^{1}-c_{k}^{1}t|}\,dx^{1} \\
 &\leq  C \int_{-\sqrt{t}+m_{k+1} t}^{\sqrt{t}+m_{k+1} t} e^{-\frac{2}{3} \sqrt{\omega_{\star}}\left|x^{1}-c_{k}^{1} t\right|} \,d x^{1}+C \int_{\sqrt{t}+m_{k+1}t}^{\infty} e^{-\frac{2}{3} \sqrt{\omega_{\star}} |x^{1}-c_{k}^{1}t|} \,d x^{1}\\
&\leq Ce^{-\frac{1}{3} \sqrt{\omega_{\star}} c_{\star}t}(2\sqrt{t})e^{\frac{2}{3} \sqrt{\omega_{\star}} \sqrt{t}}+Ce^{-\frac{2}{3} \sqrt{\omega_{\star}} \sqrt{t}}e^{-\frac{1}{3} \sqrt{\omega_{\star}} c_{\star}t}\\
&\leq Ce^{-\frac{1}{3} \sqrt{\omega_{\star}} c_{\star}t}\left((2\sqrt{t})e^{\frac{2}{3} \sqrt{\omega_{\star}} \sqrt{t}}+e^{-\frac{2}{3} \sqrt{\omega_{\star}} \sqrt{t}}\right)\\
 &\leq Ce^{-4\sqrt{\omega_{\star}}c_{\star}t}.
\end{aligned}
\end{equation}
When $j = k$, we proceed similarly, therefore from \eqref{Rnosupor}-\eqref{integralexpo2},
$$\int_{\Omega_{j}}\left|\tilde{R}_{1}\right|^{2}\leq Ce^{-4\sqrt{\omega_{\star}}c_{\star}t}.$$
In a similar way, we can obtain that
$$\int_{\Omega_{j}}\left|\tilde{R}_{1}\right|^{4}\leq Ce^{-4\sqrt{\omega_{\star}}c_{\star}t}.$$
\noindent Thus, combining \eqref{termomomento1} - \eqref{termomomento2}, Lemma \ref{solitons}, it follows that for $j \in \{2,\ldots,N\}$, 
\begin{equation}\label{varmomem}
\left|\frac{\partial}{\partial t}\int_{\R}\left(Q_{1,j}(\Vec{u}(t))+Q_{2,j}(\Vec{u}(t))\,dx\right)\right|\leq  \frac{C}{\sqrt{t}} e^{-2\sqrt{\omega_{\star}}c_{\star}t},
\end{equation}
Note that \eqref{varmomem} are satisfied for $j=1$, as in this case it is enough to use the conservation of  momentum  and momentum 2, by the flow of system \eqref{system2}.

\noindent Consequently, for $j \in \{1,2,\ldots,N\}$,
$$\left|\frac{\partial}{\partial t}\int_{\R}\left(Q_{1,j}(\Vec{u}(t))+Q_{2,j}(\Vec{u}(t))\right)\,dx\right|\leq  \frac{C}{\sqrt{t}} e^{-2\sqrt{\omega_{\star}}c_{\star}t},$$
which shows the desired result.
\end{proof}
We now proceed to prove the main result of this section.
\begin{proof}[Proof the Proposition \ref{Bootstrap3}]
Let us fix $t \in [t_{0},T^{n}]$.

From Proposition  \ref{coercivity2}, it follows that there exists $C = C\left(\mathbf{\Phi}_{1}, \ldots, \mathbf{\Phi}_{N}\right) > 0$ such that
\begin{equation}\label{estimativaH1}
\|\vec{u}(t)\|_{X}^{2} \leq C\mathcal{H}_{loc} (\vec{u}(t), t).
\end{equation}
Now, using Lemma \ref{newS}, we obtain
\[
\left|\mathcal{H}_{loc}(\vec{u}(t), t)\right| \leq \left|\mathcal{S}(\vec{u}(t), t) - \mathcal{S}(\vec{R}(t), t)\right| +\sum_{j=1}^{N}\mathcal{O}\left(|\tilde{\omega}_{j} -\omega_{j}|^{2}\right)+ O\left(e^{-3 \sqrt{\omega_{\star}}c_{\star}t}\right).
\]
Furthermore,
\begin{equation}\label{leyesdeconser}
\begin{aligned}
\mathcal{S}(\vec{R}(T^{n}), t) &= \sum_{j=1}^{N} \mathcal{S}_{j,loc}\left(\vec{R}_{j}(T^{n}), T^{n}\right) + O\left(e^{-3 \sqrt{\omega_{\star}} c_{\star} T^{n}}\right)\\
&= \sum_{j=1}^{N} \mathcal{S}_{j,loc}\left(\vec{R}_{j}(t), t\right) + O\left(e^{-3 \sqrt{\omega_{\star}} c_{\star} t}\right)\\
&= \mathcal{S}(\vec{R}(t), t).
\end{aligned}
\end{equation}
Thus, since $\vec{u}(T^{n}) = \vec{R}(T^{n})$, from Lemma \ref{variationomega} and Lemma \ref{VariacionS} we obtain
\begin{equation}\label{estimativageneralH}
\begin{aligned}
\left|\mathcal{H} (\vec{u}(t), t)\right| &\leq \left|\mathcal{S}(\vec{u}(t), t) - \mathcal{S}(\vec{R}(t), t)\right|+\sum_{j=1}^{N}\mathcal{O}\left(|\tilde{\omega}_{j} -\omega_{j}|^{2}\right) + O\left(e^{-3 \sqrt{\omega_{\star}}c_{\star}t}\right)\\
&\leq \int_{t}^{T^{n}} \left|\frac{\partial \mathcal{S}(\vec{u}(s), s)}{\partial s}\right| ds + O\left(e^{-3 \sqrt{\omega_{\star}}c_{\star}t}\right)\\
&\leq \frac{C}{\sqrt{t}} e^{-2 \sqrt{\omega_{\star}}c_{\star} t}.
\end{aligned}
\end{equation}
Combining \eqref{estimativaH1}- \eqref{estimativageneralH}, it follows that
\[
\|\vec{u}(t)\|_{X}^{2} \leq \frac{C}{\sqrt{t}}  e^{-2\sqrt{\omega_{\star}}c_{\star} t}.
\]
We have that for sufficiently large $T_{0}$ (depending only on $c_{j}$),
\[
\|\vec{u}(t)\|_{X} \leq \frac{1}{2} e^{-\sqrt{\omega_{\star}}c_{\star} t}.
\]
This proves the desired result.
\end{proof}
\section{Existence Multi-solitons}
In this  section we will prove  Theorem \ref{maintheorem}. We first  prove that our approximate solutions $\Vec{u}^{n}$ indeed satisfy the assumption in Proposition \ref{desiesti}.

\begin{proposition}[Uniform Estimates]\label{UniformEstimates2}
 There are $T_{0}\in \R$ and $n_{0} \in\N$ such that, for every $n \geq n_{0}$, each approximate solution $\Vec{u}^{n}$ is defined in $[T_{0},T^{n}]$ and for all $t \in [T_{0},T^{n}]$
\begin{equation*}
\|\Vec{u}^{n}(t)-\vec{R}(t)\|_{X}\leq e^{-\sqrt{\omega_{\star}}c_{\star}t}.
\end{equation*}
\end{proposition}
\begin{proof}
Let $T_{0}$, $n_{0}$ and set $n \in \N$ with $n \geq n_{0}$ . As we will see, the argument below shows that as long as the approximate solution $\Vec{u}^{n}$ exist it satisfies \eqref{desiesti}, which in turn implies that it does not blow up in finite time. Consequently,  we may assume that it is defined in $[T_{0},T^{n}]$ and \[\Vec{u}^{n} \in \mathbf{C}([T_{0},T^{n}],H^1(\R)) .\]

 Furthermore, since  $\vec{R}$ is continuous with respect to time, it follows that for every $t$ sufficiently close to ${T^n}^-$.  
\[
\|\Vec{u}^{n}(t)-\vec{R}(t)\|_{ H}\\ \leq \|\Vec{u}^{n}(t)-\Vec{u}^{n}(T^{n})\|_{X} + \|\vec{R}(T^{n})-\vec{R}(t)\|_{X} \leq e^{-\sqrt{\omega_{\star}}c_{\star}t}.
\]
Now, let us consider 
\[ t_{\sharp}:= \inf\left\{t_{\dagger}\in [T_{0},T^{n}]: \, \mbox{\eqref{desiesti} is satisfied for
all $t \in [t_{\dagger},T^{n}]$}\right\}.\]
Clearly $t_{\sharp} < T^{n}$. So, to prove the proposition we  just need to show that $t_{\sharp}=T_{0}$.
Assume by contradiction that $t_{\sharp} >T_{0}$. Then, by Proposition \ref{Bootstrap3}, for all $t \in [t_{\sharp},T^{n}]$, we have  \begin{equation*}
\|\Vec{u}^{n}(t)-\vec{R}(t)\|_{X}\leq \frac{1}{2}e^{-\sqrt{\omega_{\star}}c_{\star}t}.
\end{equation*}
Consequently,
\begin{equation*}
\|\Vec{u}^{n}(t_{\sharp}))-\vec{R}(t_{\sharp})\|_{X}\leq \frac{1}{2}e^{-\sqrt{\omega_{\star}}v_{ \star}t_{\sharp}}.
\end{equation*}
By the above argument, the continuity of $\Vec{u}^{n}$ implies that for $t$ close enough to $t_{\sharp}$ (from the left),
\begin{equation*}
\|\Vec{u}^{n}(t)-\vec{R}(t)\|_{X}\leq e^{-\sqrt{\omega_{\star}}c_{\star}t},
\end{equation*}
which contradicts the minimality of $t_{\sharp}$. Thus, one must have $t_{\sharp}=T_{0}$ and the proof is completed.
\end{proof}
Before proving our main theorem we still need the following result.
\begin{proposition}\label{c.n}
There exists $\Vec{u}^{0} \in X$ and a solution 
$\vec{u}$ of \eqref{system2} such that, up to a subsequence, $\Vec{u}^{n}(t)  \rightharpoonup \Vec{u}(t) $  in $X$ and $\vec{u}(T_0)=\vec{u}^{0}$ for all  $n\to +\infty$.
\end{proposition} 
\begin{proof}
In view of \eqref{desiesti} we have that $\Vec{u}^n(T_{0})$ is bounded in $X$. Hence, up to a subsequence, there is $\Vec{u}^{0} \in X$ such that
\begin{equation*}\label{convergence}
\Vec{u}^n(T_{0}) \rightharpoonup \Vec{u}^{0} \quad \mbox{in} \quad X.
\end{equation*}
Now, since $$X \stackrel{}{\hookrightarrow} H_{loc}^{s}(\R)\times \bm{\dot{H}_{loc}}
^{s-1}(\R), \, \, \text{for}\, \, s<1,$$
where $\bm{\dot{H}_{loc}}
^{s-1}(\R)=\dot{H}_{loc}^{s-1}\times H_{loc}^{s-1} \times H_{loc}^{s-1}$. It follows that \begin{equation*}
\Vec{u}^n(T_{0}) \to \Vec{u}^{0} \quad \mbox{in} \quad H_{loc}^{s}(\R)\times \bm{\dot{H}_{loc}}
^{s-1}(\R).
\end{equation*}
Let $R>0$ and $\bm{\Upsilon} \in (\bm{\mathcal{D}}(\R))^4$ be a test function such that $\text{supp}(\bm{\Upsilon}) \subset B:=\bm{B}_R(0)$.
Let $\Vec{u}^{0}$ be as obtained in Proposition \ref{c.n}. From the local well-posedness of \eqref{system2} in $H^s(\R)\times \bf{\dot{H}}
^{s-1}(\R)$, for $1/2<s\leq 1$,
there is a solution $\Vec{u}$ of  \eqref{system2}  with initial condition $\Vec{u}(T_{0})= \Vec{u}^{0}$ defined in an interval $[T_{0},T^{\star})$, where $T^{\star}$ is the maximum time of existence, such that 
\begin{equation*}
\Vec{u}^n(t) \to \Vec{u}(t) \quad \mbox{in} \quad H_{loc}^{s}(\R)\times \bm{\dot{H}_{loc}}
^{s-1}(\R).
\end{equation*}
Notice that,
$$\begin{aligned}
    \left|\langle \vec{u}^n(t)-\vec{u}(t),\bm{\Upsilon}\rangle\right| &=\bigg|\int_{\{|x|\leq R\}}\bigg(\left(\partial_t u_1(t)-\partial_t u_1^n (t)\right)\Upsilon_2 +\left(\partial_x u_1^n(t)-\partial_x u_1(t)\right)\Upsilon_1 \\
    & \quad
    +\left( u_1^n(t)- u_1(t)\right)\Upsilon_1 +\left( u_3^n(t)- u_3(t)\right)\Upsilon_3+\left( u_4^n(t)- u_4(t)\right)\Upsilon_4 \bigg)\,dx   \bigg|\\
    &\leq \|\vec{u}^n(t)-\vec{u}(t)\|_{H^{s}(B)\times \bm{\dot{H}}^{s-1}(B)}\|\bm{\Upsilon}\|_{H^{2-s}(B)\times \bm{\dot{H}}^{1-s}(B)} \, \, \to \, \, 0 \, \, \text{as} \, \, n\to \infty.
\end{aligned}$$
So, since the sequence is bounded in $X$, it follows that \begin{equation*}
\Vec{u}^n(t) \rightharpoonup \Vec{u}(t) \quad \mbox{in} \quad X.
\end{equation*}
%%%%%%%%%%%%%%%%%%%%%%%%%%%%%%%%%%%%%%%%%%%%%%%%%%%%%%%%%%%%%%%%%%%%%%%%%%%%%%%%%%%%%%%%%%%%%%%%%%%%%%%%%%%%%%%%%%%%%%%%%%%%%
Therefore, we get the desired result.
\end{proof}
Finally we are able to prove  Theorem \ref{maintheorem}.
\begin{proof}[Proof of Theorem \ref{maintheorem}]
There exist $T_{0}\in \R$ and $n_{ 0} \in\N$ such that, for every $n \geq n_{0}$ and each $t \in [T_{0},T^{n}]$,
\begin{equation}\label{des}
\|\Vec{u}^{n}(t)-\vec{R}(t)\|_{X }\leq e^{-\sqrt{\omega_{\star}}c_{\star}t}.
\end{equation}

Let $\Vec{u}^{0}$ be as obtained in Proposition \ref{c.n}. From the local well-posedness of \eqref{system2} in $H^s(\R)\times \bf{\dot{H}}
^{s-1}(\R)$, for $1/2<s\leq 1$,
there is a solution $\Vec{u}$ of  \eqref{system2}  with initial condition $\Vec{u}(T_{0})= \Vec{u}^{0}$ defined in an interval $[T_{0},T^{\star})$, where $T^{\star}$ is the maximum time of existence. We will show that $T^{\star}= +\infty$ and that $\Vec{u}$ checks estimate \eqref{desiesti}. In fact, suppose $T^{\star}<\infty$. Then, using  Proposition \ref{c.n}, by  \eqref{des} we get (taking $n$ large enough such that $T^{n}>T^{\star}$) that $\Vec{u}^n(t)$ is bounded in $X$ for all $t\in [T_{0},T^{\star })$,  for all $ t \in [T_{0}, T^{\star})$,
\[ \Vec{u}^n(t) \rightharpoonup \Vec{u}(t) \quad \mbox{in} \quad X. \]
So,  from \eqref{des},
\begin{equation*}
 \|\Vec{u}(t)-\vec{R}(t)\|_{X} \leq 
  \liminf_{n \rightarrow +\infty} \|\Vec{u}^{n}(t)-\vec{R}(t)\|_{X}
\leq e^{-\sqrt{\omega_{\star}}c_{\star}t}.
\end{equation*}
Hence, $\Vec{u}$ is bounded in $X$  for all $t\in [T_{0},T^{\star })$. But, this contradicts the blow-up alternative  in $H$ given by Theorem \ref{wellpo}. This means that $T^{\star}=+\infty.$ Clearly, the above argument also shows that \eqref{desiesti} holds for all $t\in[T_0,\infty)$. The proof of the main theorem is thus completed.
\end{proof}
 
			\section*{Acknowledgment}

			The authors are supported by Nazarbayev University under the Faculty Development Competitive Research Grants Program for 2023-2025 (grant number 20122022FD4121).
			
			%%%%%%%%%%%%%%%%%%%%%%%%%%%%%%%%%%%%%%%%%%%%%

			\section*{Conflict of interest} The authors declare that they have no conflict of interest. 
			
			\section*{Data Availability}
			There is no data in this paper.

		\end{document}